\documentclass[final,leqno,9pt,a4paper]{article}
\usepackage{color,amsmath,amsfonts,amssymb,bbold,bbm, graphicx, caption, subcaption,float,tikz}
\usepackage{amsthm}
\usepackage{multirow}
\usetikzlibrary{arrows.meta}

\usepackage[ruled,linesnumbered]{algorithm2e}

\usepackage{float}
\newfloat{myfloat}{tbp}{lop}

\usepackage{subcaption}
\usepackage{array}
\usepackage{leftidx}

\def\ep{\varepsilon}\def\tilde{\widetilde}
\def\T{\mathbf T} \def\RR{\mathbbm R} \def\S{\mathcal S}
\def\sign{\mathrm{sign}} \def\uno{\mathbb 1}

\def\b#1{\mathbf{#1}}

\definecolor{trueblue2}{RGB}{0, 115, 230}

\newtheorem{theorem}{Theorem}
\newtheorem{lemma}{Lemma}
\newtheorem{corol}{Corollary}
\newtheorem{proposition}{Proposition}
\newtheorem{remark}{Remark}
\newtheorem{definition}{Definition}

\usepackage[left=25mm,right=25mm,top=25mm,bottom=25mm]{geometry}
\title{Shifted and extrapolated power methods for tensor $\ell^p$-eigenpairs}

\author{Stefano Cipolla\footnotemark[2] \and Michela Redivo-Zaglia\footnotemark[2]
	\and Francesco Tudisco\footnotemark[3] }

\begin{document}

\maketitle

\renewcommand{\thefootnote}{\fnsymbol{footnote}}

\footnotetext[2]{Department of Mathematics ``Tullio Levi-Civita'', University of Padua, 35121 Padua, Italy, \texttt{cipolla@math.unipd.it, michela.redivozaglia@unipd.it}}
\footnotetext[3]{GSSI $-$ Gran Sasso Science Institute, 67100 L'Aquila, Italy, \texttt{francesco.tudisco@gssi.it}}

\begin{abstract}
This work is concerned with the computation of $\ell^p$-eigenvalues and eigenvectors of square tensors with $d$ modes.
In the first part we propose two possible shifted variants of the popular (higher-order) power method and{, when the tensor is entrywise nonnegative {with a possibly reducible pattern} and $p$ is strictly larger than the number of modes, we prove} the convergence of both the schemes to the Perron $\ell^p$-eigenvector {and to} the maximal corresponding $\ell^p$-eigenvalue {of the tensor}. Then, {in the second part,} motivated by the slow rate of convergence that the proposed methods achieve for certain real-world tensors when $p\approx d$, the number of modes, we introduce an extrapolation framework  based on the simplified topological $\ep$-algorithm to efficiently accelerate the shifted power sequences.
Numerical results on synthetic and real world problems show the improvements gained by the introduction of the shifting parameter and the efficiency of the acceleration technique.

\vspace{.5em}


\vspace{.5em}

\noindent{\bf AMS}
15A69, 15A18, 65B05, 65B10, 65B99, 65F15
\end{abstract}		
\section{Introduction}
A  tensor, or hypermatrix, is a multi-dimensional array: a set of numbers $t_{i_1,\dots,i_d}$  indexed along $d$ modes. When $d=1$ the tensor is a vector, whereas for $d=2$ it reduces to a matrix.

Tensor eigenvalue problems have gained  considerable attention in recent years and a number of contributions have addressed relevant issues both from the theoretical and the numerical points of view.
The multi-dimensional nature of tensors naturally gives rise to a variety of eigenvalue problems. In fact, the classic eigenvalue and singular value problems for a matrix can be generalized to the tensor setting following different constructions which lead to different notions of eigenvalues and singular values for tensors, all of them reducing to the standard matrix case when the tensor is assumed to have  2 modes.
The best known methods for addressing eigenvalues and eigenvectors of real tensors are based on extensions of the power method. To our knowledge the first occurrence of a power method for tensors, often called \textit{higher-order} power method, is given in \cite{kofidis2002best,regalia2003monotonic}, whereas shifted versions of this method, such as the SS-HOPM and LZI, were considered e.g.\ in \cite{kolda2011shifted,liu2010always}.
When the considered tensor is real with nonnegative entries, a natural question arises if and to what extent the Perron-Frobenius theorem for matrices can be transferred to the multi-dimensional setting. The answer turns out to be non-trivial and many authors have worked in this direction in recent years, see for instance \cite{chang2008perron,friedland2013perron,gautier2016tensor,liu2010always,ng2009finding}. In particular, a summary of these results and a unifying Perron--Frobenius theorem for tensors is presented in \cite{gautier2017tensor}. %, where the authors collect all previous results under a unifying framework and provide new general conditions to ensure the existence, uniqueness, maximality and computability of various forms of tensor Perron eigenvectors.

In this work we focus on the $\ell^p$-eigenvector problem for square real tensors, being $p$ any real number larger than 1. In Section \ref{sec:tensor_eigenvalue} we review the definition and relevant properties of such eigenvalue problem. Then, in Section \ref{sec:PM}, we introduce  two new shifted power methods for this tensor eigenproblem, that include as special cases the SS-HOPM and the LZI algorithms mentioned above. In  the case of real nonnegative tensors {with possibly reducible patterns}, we prove that the proposed shifted power sequences converge to the unique Perron $\ell^p$-eigenvector and to the corresponding positive $\ell^p$-eigenvalue, showing that in this setting the methods inherit the desirable convergence guarantees of their matrix counterpart. In Section \ref{sec:experiments_part1} we study experimentally the numerical behavior of the newly introduced schemes. In Section \ref{sec:extrapolation} we discuss how the sequences produced by the new algorithms can be accelerated by means of extrapolation methods that require only few additional computations. Finally, in Section \ref{sec:experiments_extrapolated}, we show via numerical experiments that the extrapolated sequences converge significantly faster than the original power sequences for a number of different test problems borrowed from  real world applications.

% \subsection{Notation}
% The symbol $\uno$ denotes the vector of all ones. $\RR^n_+$ is  the nonnegative orthant in $\RR^n$, that is the cone of vectors with nonnegative entries. $\RR^n_{++}$ is the interior of $\RR^n_+$,  that is the open cone of vectors with strictly positive entries.
%
%
%
%
%
%
%
%
%
%
%
%
\section{Tensor $\ell^p$-eigenvalues and eigenvectors}\label{sec:tensor_eigenvalue}
In this section we review the notion of $\ell^p$-eigenvalues for a square real tensor, and discuss some of their properties.

%
% The symbol $\uno$ denotes the vector of all ones. $\RR^n_+$ is  the nonnegative orthant in $\RR^n$, that is the cone of vectors with nonnegative entries. $\RR^n_{++}$ is the interior of $\RR^n_+$,  that is the open cone of vectors with strictly positive entries.
%
Let $\T=(t_{i_1, \dots,i_d})\in \RR^{n\times \cdots\times n}$ be a real tensor with $d$ modes, each of dimension $n$. Since each mode of the tensor has the same dimension, we say that the tensor is square and we  briefly write $\T\in\RR^{[d,n]}$. Furthermore, we say that $\T$ is symmetric if the entries $t_{i_1, \dots, i_d}$ are invariant with respect to any permutation of the indices $i_1, \dots, i_d$, as introduced in \cite{comon2008symmetric}. To any square tensor $\T \in \RR^{[d,n]}$ and any vector $\b x$ of size $n$, we can associate   a new vector $\b y$ by ``multiplying'' $\T$ times $\b x$. This operation extends the matrix vector product and is defined via a polynomial map which we denote with the same capital letter denoting the corresponding tensor, but in italic normal font. More precisely
\begin{definition}
 Given a tensor $\T=(t_{i_1,\dots,i_d})\in\RR^{n\times\cdots \times n}$,   define the map
\begin{equation}\label{eq:tensor_x_vector}
T:\RR^n\to\RR^n, \qquad \b x\mapsto T(\b x)_{i_1} = \sum_{i_2,\dots,i_d}t_{i_1,i_2,\dots,i_d}x_{i_2}\cdots x_{i_d}
\end{equation}
for $i_1=1,\dots, n$.
\end{definition}

For the sake of completeness, let us point out that the vector $T(\b x)$ is sometimes denoted by $\T \b x^{d-1}$ (see e.g.\ \cite{kolda2011shifted,regalia2003monotonic}). This is because, when $d=2$,  that is $\T$ is a $n\times n$ matrix, then $T(\b x)$ coincides with the matrix vector product $\T \b x$. However, in this work we prefer to distinguish $\T$ and $T$ explicitly. This choice is made for the sake of generality, having in mind a future extension of the analysis here presented to tensor singular values and the case of rectangular tensors.
\begin{remark}\label{rem:tensor_x_vector}
Clearly, other multiplicative maps can be associated to a square tensor $\T\in \RR^{[d,n]}$. In fact, for any $1\leq k\leq d$ one can define a map $T_k:\RR^n\to\RR^n$ by replacing the summation on the right-hand side of \eqref{eq:tensor_x_vector} with the sum over all the indices $i_1, \dots, i_d$ except for $i_k$. This defines $d$ different operations $T_1, \dots, T_d$ where $T_1=T$ but $T_k\neq T$ in general, if $k\neq 1$. Precisely, one easily realizes that {all} the mappings $T_k$ coincide if and only if the tensor $\T$ is symmetric. Note, in particular, that for the matrix case $d=2$, the two maps $T_1$ and $T_2$ coincide with the matrix vector products $\T\b x$ and $\T^T\b x$, respectively.
\end{remark}

 As $T$ maps $\RR^n$ into itself, we can associate a concept of eigenvalues and eigenvectors to $\T$ via the mapping $T$.
\begin{definition}
Let $1<p<+\infty$. A real number $\lambda\in \RR$ is said to be a $\ell^p$-eigenvalue of $\T$, corresponding to the $\ell^p$-eigenvector $\b x\in\RR^n$, if the following  holds
\begin{equation}\label{eq:ellp-eigenvalue}
 T(\b x) = \lambda \Phi_p(\b x)\, , \quad \|\b x \|_p=1
\end{equation}
where $\|\b x\|_p$ denotes the usual $\ell^p$ norm $\|\b x\|_p=(|x_1|^p+\dots + |x_n|^p)^{1/p}$ and  $\Phi_p:\RR^n\to \RR^n$ is the map entrywise defined by $\Phi_p(\b x)_i = |x_i|^{p-2}x_i=\sign(x_i)|x_i|^{p-1}$ for $i=1,\dots, n$.
\end{definition}

The notion of $\ell^p$ eigenvalues was probably introduced by Lim in \cite{lim2005singular}. Note that, unlike the matrix case, if $p\neq d$ then the system of equations $ T(\b x) = \lambda \Phi_p(\b x)$ is not homogeneous and thus, in general, $\ell^p$-eigenvectors are not defined up to scalar multiples. This is the reason why we require the normalization $\|\b x \|_p=1$. %For this reason, if not explicitly specified otherwise, we will always assume eigenvectors of $p$-norm exactly one, that is an $\ell^p$-eigenvector $\b x$ is a  a solution of  \eqref{eq:ellp-eigenvalue} with $\|\b x \|_p=1$. %it is where this notion arises from a variational construction which we will briefly recall afterwards.

The map $\Phi_p$ is also known as the duality map of $\ell^p$ with gouge function $\mu(\alpha)=\alpha^{p-1}$ (see, e.g., \cite{lang2011eigenvalues}) and enjoys several useful properties. We recall two of them below:
\begin{itemize}
\item If $q$ is the H\"older conjugate of $p$, i.e.\ $p^{-1}+q^{-1}=1$, then $\Phi_q=\Phi_p^{-1}$ is the inverse of $\Phi_p$, that is we have  $\Phi_p(\Phi_q(\b x))=\Phi_q(\Phi_p(\b x))=\b x$ for any $\b x$.
\item When $p=2$ we have $\Phi_p(\b x)=\Phi_q(\b x)=\b x$, i.e.\  $\Phi_p$ is the identity map.
\end{itemize}

We say that a solution $(\lambda, \b x)\in\RR\times \RR^n$ of Equation \eqref{eq:ellp-eigenvalue} is an $\ell^p$-eigenpair of $\T$. A special name is usually reserved to the cases $p=2$ and $p=d$, the former being known as $Z$-eigenpairs and the latter as $H$-eigenpairs. Actually, when $d$ is odd, the usual definition of $H$-eigenvectors given in the literature slightly differs from \eqref{eq:ellp-eigenvalue}. In fact, a $H$-eigenpair is usually defined as the solution of the system $T(\b x) = \lambda \b x^{d-1}$  \cite{qi2017tensor}, which coincides with \eqref{eq:ellp-eigenvalue} only if $d$ is even or if $\b x$ is entrywise nonnegative. {In this work we will call $H$-eigenvector any vector $\b x$ satisfying \eqref{eq:ellp-eigenvalue} with $p=d$.} Of course one could define $\ell^p$-eigenvalues by replacing the duality map $\Phi_p$ with a standard entrywise power of the vector. However, as we notice below,  using $\Phi_p$ has the advantage that, just like symmetric matrices, $\ell^p$-eigenvectors of symmetric tensors $\T$ defined as in \eqref{eq:ellp-eigenvalue} have a  variational characterization as critical points of the Rayleigh quotient
\begin{equation}\label{eq:rayleigh}
\b x\mapsto f_p(\b x) = \frac{\b x^T T(\b x)}{\|\b x\|_p^d}, \qquad  \b x^T T(\b x)=\sum_{i_1,\dots,i_d}t_{i_1,i_2,\dots,i_d}x_{i_1}\cdots x_{i_d}\, .
\end{equation}
% and this is the reason we adopt this definition here.    %For the sake of completeness, let us point out that the vector $T(\b x)$ is sometimes denoted by $\T \b x^{d-1}$ (see f.i. \cite{kolda2011shifted,regalia2003monotonic}), so that Equation \eqref{eq:ellp-eigenvalue} becomes $\T \b x^{d-1}=\lambda \Phi_p(\b x)$. This is because, as recalled in Remark \ref{rem:tensor_x_vector}, when $d=2$,  that is $\T$ is a $n\times n$ matrix, then $T(\b x)$ coincides with the matrix vector product $\T \b x$. However, in this work we prefer to distinguish $\T$ and $T$ explicitly. This choice is made for the sake of generality, having in mind a future extension of the analysis here presented to tensor singular values and the case of rectangular tensors. %For the sake of completeness, let us point out that  from the above discussion  \eqref{eq:ellp-eigenvalue} can be also written as $\T \b x^{d-1}=\lambda \Phi_p(\b x)$.
%
 % In what follows, given any tensor $\T$ we always denote with the corresponding capital letter in normal font $T$ the associated multiplicative map \eqref{eq:tensor_x_vector}.
%
%
% It was observed by Lim in \cite{lim2005singular} that, when $\T$ is symmetric, $\ell^p$-eigenvectors of $\T$ coincide with critical points of the following Rayleigh quotient
%
% Below we give a simple and independent proof of this fact
\begin{proposition}\label{prop:variational}
Let $\T$ be a symmetric square tensor and let $p>1$. Then $\b x\in\RR^n$ is an $\ell^p$-eigenvector of $\T$  if and only if $\b x$ is a critical point of $f_p$ and the associated $\ell^p$-eigenvalue coincides with $\b x^TT(\b x)/\|\b x\|_p^p$. % $\nabla f_p(\b x)=0$ and $f_p(\b x)=\lambda$.
\end{proposition}
\begin{proof}
As both $p$ and $d$ are larger than one, the Rayleigh quotient $f_p$ is differentiable and thus $\b x$ is a critical point of $f_p$ if an only if $\nabla f_p(\b x) = \boldsymbol{0}$, where $\nabla$ denotes the gradient. % of $f_p$ evaluated at $\b x$.
By the symmetry of $\b T$, the gradient of $\b x^T T(\b x)$ is given by
$$
\nabla \{\b x^TT(\b x)\} = T_1(\b x) + T_2(\b x)+\cdots + T_d(\b x) = d \, T(\b x)
$$
where $T_k:\RR^n\to\RR^n$ are the maps considered in Remark \ref{rem:tensor_x_vector}. Similarly we compute the gradient of the $p$-norm: $
\nabla \|\b x\|_p = \Phi_p(\b x)/\|\b x\|_p^{p-1}\, .
$

Therefore, by the chain rule, we get
$$
\nabla f_p(\b x) = \frac{\|\b x \|_p^d\nabla\{\b x^T T(\b x)\} - \b x^T T(\b x)\nabla \|\b x\|_p^d  }{\|\b x\|_p^{2d}} = \frac{d}{\|\b x\|_p^d}\left\{ T(\b x) - \left( \frac{\b x^T T(\b x)}{\|\b x\|_p^p}\right)\Phi_p(\b x)  \right\}
$$
which concludes the proof.
\end{proof}

{\begin{remark} \label{eq:constrained_remark}
		Note that with $\b y = \b x / \|\b x\|_p$, we get $f_p(\b x) = \b y^T T(\b y)$. Hence, critical points of $f_p$ coincide with critical points of $\b x^T T(\b x)$, constrained by $\|\b x\|_p=1$. With this constrained optimization formulation one can obtain an equally simple proof of Proposition \ref{prop:variational} as a direct consequence of the Karush-Kuhn-Tucker conditions. This is the approach used for example in \cite{lim2005singular} and \cite[Thm.\ 3.2]{kolda2011shifted}.
	\end{remark}}

The notion of $\ell^p$-eigenvectors and eigenvalues for nonnegative tensors arises in many contexts and applications. The cases $p=2$ and $p=d$ are the natural generalization of eigenvalues and eigenvectors for matrices, as in the case $d=2$  these two notions coincide. The case $p=d$ has been used for example to characterize the positive definiteness of homogeneous polynomial forms in \cite{qi2003multivariate} or to compute the importance of nodes in a network \cite{benson2019three}. The case $p=2$ is strictly related with the best rank-1 approximation of a tensor with respect to the Frobenius norm  \cite{de2000best,friedland2013best,kolda2009tensor}. It arises in certain constraint satisfaction problems \cite{de2005tensor}, as well as problems involving  higher-order statistics \cite{swami1997bibliography},  signal processing \cite{comon2002tensor, nikias1993signal}, quantum geometry \cite{hu2016computing} and data analysis \cite{cipolla2019extrapolation,gleich2015multilinear, nguyen2016efficient,arrigo2017centrality}.

If $\b T$  is symmetric, then it has a finite set of $Z$-eigenvalues, one of which must be real if $d$ is odd \cite{cartwright2013number}. However, computing a prescribed $Z$-eigenpair is in general NP-hard \cite{hillar2013most}.
When $\b T$ is entrywise nonnegative, instead, conditions can be given on the nonzero pattern of $\b T$ to ensure that tensor versions of the power method converge globally to the largest $Z$-eigenvalue \cite{fasino2019higher,gautier2018contractivity,li2014limiting}.
These types of results belong to the developing Perron--Frobenius theory for nonnegative tensors, which naturally involves $\ell^p$-eigenvalues and eigenvectors, with $p$ not necessarily equals to $2$ nor to $d$ \cite{chang2008perron, gautier2016tensor,gautier2017theorem}. We review the relevant results below.

We say that $\T=(t_{i_1,\dots,i_d})$ is nonnegative when $t_{i_1,\dots,i_d}\geq 0$ for all $i_j=1, \dots, n$ and $j=1,\dots, d$. We  briefly write $\T \geq 0$ and, likewise, $\T >0$ when all the entries of $\T$ are strictly positive. Several notions of spectral radius can be then employed to generalize this concept from the matrix to the higher-order case (see e.g.\ \cite{gautier2017theorem}). Here we adopt the following
\begin{equation}\label{eq:spectral_radius}
 r_p(\T) = \sup\{|\lambda| : \lambda \text{ is an }\ell^p\text{-eigenvalue of }\T \}\,
\end{equation}
{Note that, when $\T$ is symmetric,  using Proposition \ref{prop:variational} and Remark \ref{eq:constrained_remark}, we have
	\begin{equation}
	{r}_p(\b T):= \max_{\|x\|_p=1} | \b x^T T(\b x) |.
	\end{equation}}

Unlike the matrix case, the Perron-Frobenius theorem for tensors relies on the choice of $p$ and the number of modes $d$. To the best of our knowledge, if the tensor is not entrywise positive or stochastic (settings that are  studied for example in \cite{gautier2018contractivity} and \cite{fasino2019higher}, respectively) $p\geq d$ is the wider range of values $p$ for which existence, uniqueness and maximality of a positive $\ell^p$-eigenpair for $\T\geq0$ has been proved (c.f.\ \cite{gautier2017theorem} for instance).  Within this range, a distinction must be made: while the case $p=d$  requires assumptions on the irreducibility of $\T$ (e.g. strong and weak irreducibility or primitivity, { see Definition \ref{def:wi} and} \cite{chang2010singular,friedland2013perron}), it was observed in \cite{gautier2017theorem} that  $p>d$ is associated with a Lipschitz contractive map and the Perron-Frobenius theorem holds without any special requirement on the non-zero pattern of $\T$.   In this work we shall mostly focus on this second case, thus we recall here below the corresponding Perron-Frobenius theorem. We refer to \cite{gautier2017theorem} for more details and for a thorough bibliography review on the subject.

\begin{theorem}[\cite{gautier2017theorem}]\label{thm:PF}
 Let $\T\in\RR^{[d,n]}$  be such that $\T\geq 0$ and such that $T(\uno)$ is entrywise positive, where $\uno$ is the vector of all ones. If $p>d$, then $r_p(\T)>0$ and there exists a unique  entrywise positive $\b u \in \RR^n$  such that $\|\b u \|_p=1$ and $T(\b u) = r_p(\T)\Phi_p(\b u)$.
\end{theorem}

\section{Shifted power method for $\ell^p$-eigenvalues}\label{sec:PM}
The  power method is arguably the best known method to address the computation of the maximal $\ell^p$-eigenpair of $\T$. Shifted variants of that method have been proposed for $Z$-eigenpairs and for $H$-eigenpairs %of irreducible tensors
in \cite{kolda2009tensor} and \cite{liu2010always}, respectively. Here we propose two novel possible extensions of those techniques to the case of a general $\ell^p$-eigenvalue problem {for nonnegative tensors}. The pseudocode  of these new methods is shown in Algorithms \ref{alg:shiftedPM1} and \ref{alg:shiftedPM2}.
{Note that the symbol $\circ$ at line \ref{line:shiftedPM1} of Algorithm \ref{alg:shiftedPM1} and at line \ref{line:shiftedPM2} of Algorithm \ref{alg:shiftedPM2} denotes the entrywise product and is required in order to ensure that, when $\T\geq 0$, both $T(\b x_k)$ and the shifting vectors $\b x_k$ and $\Phi_p(\b x_k)$  have the same zero pattern. In particular, when $\b T$ is irreducible, $\b x_0>0$ implies $T(\b x_k)>0$ for all $k$, thus $\sign(\b z) = \uno$} and we retrieve the shifted methods for both $Z$ and $H$ eigenpairs as special cases of the new schemes.

The proposed shifted methods share several interesting convergence properties with the original  unshifted power method. In particular, for nonnegative tensors and under the assumption $p>d$, we are guaranteed that if the initial $\b x_0$ is chosen {entrywise positive,} %within a certain cone $C_+(\T)$,
then {for $k$ large enough} the whole sequence of eigenvector approximations $\b x_k$ will stay {within a certain cone $C_+(\T)$} %in the cone
and it will converge to the unique eigenvector of the tensor $\T$ in that cone; at the same time, the sequence of eigenvalue approximations $\lambda_k$ will converge to the corresponding eigenvalue and, if $\T$ is symmetric, we are guaranteed that such eigenvalue corresponds to the $\ell^p$-spectral radius of $\T$. We state this convergence result in our main Theorem \ref{thm:convergence} below. %To this end we first discuss a number of preliminaries.

\SetKwBlock{Repeat}{For $k=0,1,2,3,\dots$ repeat}{ }
%\noindent\rule{\linewidth}{1pt}
%\begin{myfloat}[ht]
%\noindent\begin{minipage}[T]{\textwidth}
%\begin{minipage}{.5\textwidth}
\begin{algorithm}[hbt]
\caption{{Shifted Power Method 1}}\label{alg:shiftedPM1}
 \DontPrintSemicolon
\KwIn{$\b x_0>0$, $\sigma \geq  0$, $p>1$, tolerance $\varepsilon >0$}
$q := p/(p-1)$ (\textit{Conjugate exponent}) \;
\Repeat{
{$\b z = T(\b x_k)$}\;
$\b y = \Phi_q(\b z)+\sigma \, {\text{sign}(\b z)\circ} \b x_k$\;\label{line:shiftedPM1}
$\b x_{k+1} = \b y/\|\b y\|_p$\;
$\lambda_{k+1} = f_p(\b x_{k+1})$
}{\textbf{until} $\|\b x_{k+1}-\b x_k\|_p<\varepsilon$}
\end{algorithm}

 \begin{algorithm}[t]
\caption{{Shifted Power Method 2}}\label{alg:shiftedPM2}
 \DontPrintSemicolon
\KwIn{$\b x_0>0$,  $\sigma \geq  0$, $p >1$, tolerance $\varepsilon >0$}
$q := p/(p-1)$ (\textit{Conjugate exponent}) \;
\Repeat{
{$\b z = T(\b x_k)$}\;
$\b y = \Phi_q(\b z+\sigma \, {\text{sign}(\b z)\circ} \Phi_p(\b x_k))$\; \label{line:shiftedPM2}
$\b x_{k+1} = \b y/\|\b y\|_p$\;
$\lambda_{k+1} = f_p(\b x_{k+1})$
}{\textbf{until} $\|\b x_{k+1}-\b x_k\|_p<\varepsilon$}
\end{algorithm}

\subsection{Convergence analysis}
Let $\RR^n_+$ denote  the nonnegative orthant in $\RR^n$, i.e., the cone of vectors with nonnegative entries, and let $\RR^n_{++}$ be its interior, i.e., the open cone of vectors with strictly positive entries.
Consider the cone of nonnegative vectors whose zero pattern is preserved by $\T$: %that have the same zero pattern of $T(\uno)$:

\begin{definition}\label{def:cone}
For a nonzero tensor $\T\geq 0$  define $C_+(\T)$ as the cone
$$C_+(\T) =\bigcap_{k\geq k_0}C_+^{k} \quad \text{with} \quad  C_+^k =  \Big\{\b x\geq 0 : c_1 \b x \leq T^k(\uno) \leq c_2\b x, \text{ for some }c_1,c_2>0\Big\},$$
where  $T^k$ denote $k$ compositions of the map $T$ and $k_0=k_0(\T)$ is the smallest integer $k_0\geq 1$ such that $C_+(\T)$ is not empty.
\end{definition}

Note that, for each $k$, the set $C_+^k$ is the cone of nonnegative vectors having the same zero pattern as $T^k(\uno)$. Since $T^{k+1}(\uno)=T(T^k(\uno))$ and $\uno\geq T(\uno)$ entrywise, one easily realizes that $T^k(\uno)_i=0$ implies $T^{k+1}(\uno)_i=0$. Thus, there exists a finite integer $k_0$ such that $\cap_{k\geq k_0}C_+^k$ is nonempty, i.e.\ $C_+(\T)$ is well defined. While $C_+(\T)$ might contain only the zero vector when $\T$ is too sparse, that situation rarely occurs in practice. In fact, for example, when $T(\uno)>0$ then $C_{+}(\T)=\RR_{++}^{n}$ is the whole cone of positive vectors.

From its definition we notice that $C_+(\T)$  is the cone of nonnegative vectors with the least number of zero entries that is preserved  by $T$ and that $C_+(\T)$  is also preserved by the iterates of the shifted power methods. Precisely, if $\b x_{k+1}$ is computed from $\b x_k$ using  the updating formulas of either Algorithm \ref{alg:shiftedPM1} or Algorithm \ref{alg:shiftedPM2} and $\b x_k\in C_+(\T)$, then also $\b x_{k+1}\in C_+(\T)$. Moreover, we will show in the theorem below that $\T$ has a unique $\ell^p$-eigenvector in $C_+(\T)$ and that the two algorithms converge to it.

At the same time, we do not need to compute the cone $C_+(\T)$ beforehand, as we are guaranteed to converge within $C_+(\T)$ if we start with any positive vector $\b x_0>0$. More precisely, if $\b x_0>0$ and $(\b x_k)$ is the sequence generated  by either Algorithm \ref{alg:shiftedPM1} or Algorithm \ref{alg:shiftedPM2}, then all the vectors $\b x_k$ belong to $C_+(\T)$ when $k$ is large enough. In fact, if we start with $\b x_0$ entrywise positive, then $\b x_1$ has the same zero pattern as $T(\uno)$, $\b x_2$ has the same zero pattern as $T^2(\uno)$ and so forth. Thus,  $\b x_k\in C_+(\T)$ for all $k\geq k_0$.

In what follows,  absolute values are meant entrywise, that is for $\b x\in \RR^n$, $|\b x|$ denotes the vector with entries $|x_i|$, for $i=1,\dots,n$.

\begin{lemma}\label{lem:max_is_in_cone}
Let $\T \in \RR^{[d,n]}$ be entrywise nonnegative and let $f_p$ be defined as in \eqref{eq:rayleigh}, {then the maximum of $f_p$ is attained on a nonnegative vector.}
\end{lemma}
\begin{proof}
Let $\b x\in\RR^n$. Since $\T$ has nonnegative entries, we have $T(\b x)\leq T(|\b x|)$. Thus, if $\b x\neq 0$, we have
$
|f_p(\b x)|=\frac{|\b x^T T(\b x)|}{\|\b x \|_p^d } \leq \frac{|\b x|^T T(|\b x|)}{\| \, |\b x| \, \|_p^d} = f_p(|\b x|).
$
Hence, if $\b y\in\RR^n\setminus\{0\}$ is a global maximizer of $f_p$, then also $|\b y|$ is a global maximizer.
\end{proof}

The following main result holds:
\begin{theorem}\label{thm:convergence}
{ Let $\T\geq 0$, $k_0=k_0(\T)$ be as in Definition \ref{def:cone} and let $p>d$. Then there exists a unique solution to \eqref{eq:ellp-eigenvalue} in $C_+(\T)$, that is a unique $\ell^p$-eigenvector $\b u \in C_+(\T)$. We will define this as the Perron $\ell^p$-eigenvector. In particular, for any $\sigma \geq 0$ and any  $\b x_0 >0$,  the sequences $(\lambda_k)$ and  $(\b x_k)$ defined by either Algorithm \ref{alg:shiftedPM1} or Algorithm \ref{alg:shiftedPM2}  are such that:
 \begin{enumerate}
\item  $\b x_k \in C_+(\T)$ for all $k\geq k_0$;
\item $(\b x_k)$ converges to the Perron $\ell^p$-eigenvector $\b u$ of $\T$ in $C_+(\T)$ with $\|\b u\|_p=1$;
\item $(\lambda_k)$ converges to the $\ell^p$-eigenvalue $\lambda >0$  of $\T$ corresponding to $\b u$.
 \end{enumerate}}
%   Moreover, if $\T$ is symmetric, then  $\lambda = r_p(\T)$.
%  \begin{enumerate}
%   \item $(\lambda,u)$ is an $\ell^p$-eigenpair of $\mathcal T$
%   \item $\lambda > 0$ and $u\in S_+^p$ is the unique (up to multiples) eigenvector of $\T$ in $C_+(\T)$
%   \item If $\mathcal T$ is super-symmetric and the Jacobian of $\Phi_p^{-1}T$ evaluated at $\mathbbm 1$ is irreducible, then  $\lambda = r(u) = \max_{x\neq 0}r(x)$
%  \end{enumerate}
\end{theorem}
\begin{proof}
{Point $1$ holds by definition of $C_+(\T)$. So, for any $k\geq k_0$ all the points of the sequence $(\b x_k)$ are in $C_+(\T)$. We show below that the iterators are contractions on $C_+(\b T)$. }

 Given two vectors $\b x,\b y\in C_+(\T)$, consider the following quantities
 $$M(\b x/\b y)  = \inf\{\mu>0 : \b x \leq \mu \, \b y\}\qquad m(\b x/\b y) = \sup \{\mu>0 : \b x\geq \mu \, \b y\}\, .$$
Since $C_+({\b T})$ is a simplicial cone (see \cite{lemmens22birkhoff} e.g.), for $\b x, \b y \in C_{+}(\b T)$ it holds $M(\b x/\b y)=\max_{y_i \neq 0 } x_i/y_i$ and $m(\b x/\b y)=\min_{y_i \neq 0 } x_i/y_i$.
 % \tilde H_\sigma(x) = H_\sigma(x)/\|H_\sigma(x)\|_p\, .$$
 %
 %As $p>d$, then the Perron-Frobenius theorem for multi-homogeneous maps \cite{gautier2017Perron} implies that there exists a unique positive $\ell^p$-eigenpair $(\lambda, u)$ of $\mathcal T$.
 %
 %Let us show that $\tilde H_\sigma(x_k)$ converges to $u$.
 %
 %Assume without loss of generality that $\lambda = \Phi_p(1-\sigma)$ so that $\Phi_p^{-1}(T(u)) = \Phi_p^{-1}(\lambda \Phi_p(u)) = (1-\sigma)u$ and $H_\sigma(u)=u$. This can be done by considering the function $\lambda^{-1}\Phi_p(1-\sigma)T$ instead of $T$.
 For any $\b x,\b y\in C_+(\T)$, by definition we have $m(\b x/\b y)\b y\leq \b x\leq M(\b x/\b y)\b y$, thus
 $$m(\b x/\b y)^{d-1} T(\b y) \leq T(\b x) \leq M(\b x/\b y)^{d-1} T(\b y)\, .$$
 %and, similarly,
 %$$m(x/y)^{p-1} \Phi_p(y) \leq \Phi_p(x) \leq M(x/y)^{p-1} \Phi_p(y)\, .$$
For $\sigma\geq 0$ and $p>1$, consider the following two mappings:
$$
 H_\sigma(\b x) = \Phi_p^{-1}(T(\b x))+\sigma \b x \quad \text{ and } \quad K_\sigma(\b x) = \Phi_p^{-1}(T(\b x)+\sigma \Phi_p(\b x))\, .
$$
 Note that both $H_\sigma$ and $K_\sigma$ preserve $C_+(\T)$, that is for any $\b x\in C_+(\T)$ we have $H_\sigma(\b x)\in C_+(\T)$ and $K_\sigma(\b x) \in C_+(\T)$. {Also note that when $\b x\in C_+(\T)$ then $\b z = T(\b x)\in C_+(\T)$, thus $\sign(\b z)\circ \b x=\b x$ and $\sign(\b z)\circ \Phi_p(\b x)=\Phi_p(\b x)$. Hence, as}  $q=p/(p-1)$, we have $\Phi_q=\Phi_p^{-1}$ and $H_\sigma(\b x)$, $K_\sigma(\b x)$ are the iterations maps in Algorithms \ref{alg:shiftedPM1} and \ref{alg:shiftedPM2}.
Observe, moreover, that for any two distinct points $\b x$ and $\b y$ on the $\ell^p$-sphere $\S_p = \{\b x:\|\b x\|_p=1\}$, we have $M(\b x/\b y)>1$ and $m(\b x/\b y)<1$. In fact $M(\b x/\b y)\leq 1$ implies $x_i\leq y_i$ for any $i$ and $x_i<y_i$ for at least one $i$ (as $\b x\neq \b y$), which contradicts the fact that $\b x$ and $\b y$ have same $\ell^p$-norm. An analogous contradiction follows by assuming $m(\b x/\b y)\geq 1$.

Therefore, as $p>d$, for any two $\b x,\b y \in C_+(\T)\cap \S_p$ we have  $M(\b x/\b y)^{\frac{d-1}{p-1}}<M(\b x/\b y)$,  $m(\b x/\b y)^{\frac{d-1}{p-1}}> m(\b x/\b y)$, and
 \begin{align}
H_\sigma(\b x)&\leq M(\b x/\b y)^{\frac{d-1}{p-1}}\, \Phi_p^{-1}(T(\b y)) + \sigma M(\b x/\b y)\b y < M(\b x/\b y) H_\sigma(\b y)\label{eq:a}\\
H_\sigma(\b x)&\geq m(\b x/\b y)^{\frac{d-1}{p-1}}\, \Phi_p^{-1}(T(\b y)) + \sigma m(\b x/\b y)\b y > m(\b x/\b y) H_\sigma(\b y)\label{eq:b}\\
K_\sigma(\b x)&\leq \Phi_p^{-1}(M(\b x/\b y)^{d-1} T(\b y) + \sigma M(\b x/\b y)^{p-1}\Phi_p(\b y)) < M(\b x/\b y) K_\sigma(\b y)\label{eq:c}\\
K_\sigma(\b x)&\geq \Phi_p^{-1}(m(\b x/\b y)^{d-1} T(\b y) + \sigma m(\b x/\b y)^{p-1}\Phi_p(\b y)) > m(\b x/\b y) K_\sigma(\b y)\, . \label{eq:d}
 %\frac{H_\sigma(x)_i}{H_\sigma(y)_i} \leq
% \frac{M(x/y)^{\frac{d-1}{p-1}}\, \Phi_p^{-1}T(y)_i+\sigma\, M(x/y)y_i}{\Phi_p^{-1}T(y)_i+\sigma y_i} < M(x/y) \quad \forall i
 \end{align}
%  $$
%  \frac{H_\sigma(x)_i}{H_\sigma(y)_i} \geq \frac{m(x/y)^{\frac{d-1}{p-1}}\, \Phi_p^{-1}T(y)_i+\sigma\, m(x/y)y_i}{\Phi_p^{-1}T(y)_i+\sigma y_i} > m(x/y)  \quad \forall i
%  $$
Let us now consider the normalized maps
$$
\tilde H_\sigma(\b x) = H_\sigma(\b x)/\|H_\sigma(\b x)\|_p\quad  \text{ and } \quad \tilde  K_\sigma(\b x) = K_\sigma(\b x)/\|K_\sigma(\b x)\|_p
$$
and the Hilbert's projective distance between $\b x$ and $\b y$ in $C_+(\T)$, defined by
\begin{equation}\label{eq:hilbert-distance}
d_H(\b x,\b y) = \log\left(\frac{M(\b x/\b y)}{m(\b x/\b y)}\right)\, .
\end{equation}
Note that both $\tilde H_\sigma$ and $\tilde K_\sigma$ preserve the set $C_+(\T)\cap \S_p$. Moreover, as $d_H$ is scale invariant, combining inequalities \eqref{eq:a}--\eqref{eq:d} we see that both  $\tilde H_\sigma$  and $\tilde K_\sigma$ are contractions on $C_+(\T)\cap \S_p$ with respect to $d_H$. %The same conclusion holds for $K_\sigma$ from \eqref{eq:c} and \eqref{eq:d}.
In fact we get
$$d_H(\tilde H_\sigma(\b x),\tilde H_\sigma(\b y))=\log\left(\frac{M(H_\sigma(\b x)/H_\sigma(\b y))}{m(H_\sigma(\b x)/H_\sigma(\b y))}\right)<\log\left(\frac{M(\b x/\b y)}{m(\b x/\b y)}\right)=d_H(\b x,\b y)$$
and similarly for $\tilde K_\sigma$.
%As $d_H$ is scale invariant, also $\tilde H_\sigma(x) = H_\sigma(x)/\|H_\sigma(x)\|_p$ and $\tilde K_\sigma(x) = K_\sigma(x)/\|K_\sigma(x)\|_p$ have to be contractive.

As $C_+(\T)\cap \S_p$ is complete (see \cite{gautier2018contractivity} e.g.), the Banach's fixed point theorem (see \cite{edelstein1962fixed}) implies that there exits a unique $\b u\in C_+(\T)\cap \S_p$ such that, for any initial choice $\b x_0>0$, we have $\tilde H_\sigma(\b x_k)\to  \b u=\tilde H_\sigma(\b u)$ as $k\to \infty$. Similarly $\tilde K_p(\b x_k)\to \b v=\tilde K_\sigma(\b v)$ for a unique $\b v\in C_+(\T)\cap \S_p$.  Thus $H_\sigma(\b u) = \lambda_1 \b u$ and $K_\sigma(\b v)=\lambda_2 \b v$ with $\lambda_1 = \|H_\sigma(\b u)\|_p$ and $\lambda_2=\|K_\sigma(\b v)\|_p$.
We now show that $\b u$ and $\b v$ coincide.  In fact, from $K_\sigma(\b v)=\lambda_2\b v$, we have
$\Phi_p^{-1}(T(\b v)) = \Phi_p^{-1}(\Phi_p(\lambda_2)-\sigma)\b v =: \alpha \, \b v$.
Therefore $H_\sigma(\b v) = (\alpha + \sigma) \b v$, that is $\b v$ is the fixed point  of $\tilde H_\sigma$ in $C_+(\T)\cap \S_p$, i.e. $\b v=\b u$.

 Finally, note that $\b u$ is an   eigenvector of  $\T$ with eigenvalue $\lambda = \Phi_p(\lambda_2)-\sigma = \Phi_p(\lambda_1-\sigma)$ and such $\lambda$ is positive. In fact, as both $\Phi_p(\b u)$ and $T(\b u)$ are nonnegative vectors in $C_+(\T)$, the identity $T(\b u)=\lambda\Phi_p(\b u)$ implies $\lambda>0$.
%
 %
% To conclude the proof we need to show that when $\T$ is symmetric, $\lambda = r_p(\T)$.  To this end note that, if $\b w$ such that $\|\b w\|_{p}=1$ is any eigenvector of $\T$ with eigenvalue $\mu$, we have
% $$|\mu| = |\b w^T T(\b w)| \leq |\b w|^T T(|\b w|)  \leq \max_{\|\b x\|_p=1}f_p(\b x)=: \tilde r\, .$$
% Let $\b x^*$ be the maximizer of $f_p$ and let $\b u^* = \b x^*/\|\b x^*\|_p$. As $f_p(\b x^*)=f_p(\b u^*)$ we deduce that $\b u^*$ is a maximizer too.  By Lemma \ref{lem:max_is_in_cone} we {can consider} $\b u^*\in C_+(\T)$ and, by definition,  we have $\nabla f_p(\b u^*)=\boldsymbol 0$. From Proposition \ref{prop:variational} we deduce that $\b u^*$ is an $\ell^p$-eigenvector of $\T$, that is $T(\b u^*)=\tilde r \, \Phi_p(\b u^*)$, with $\tilde r = f_p(\b u^*)$. As {$\b u^*\in C_+(\T)$},  the uniqueness of $\b u$ implies $\b u^*=\b u$ and $\tilde r = \lambda$, concluding that $|\mu|\leq \lambda$, i.e., $\lambda = r_p(\T)$.
\end{proof}

\begin{corol}
If $\T \geq 0$ is symmetric and $p>d$, then the sequence $(\lambda_k)$ generated by either Algorithm \ref{alg:shiftedPM1} or Algorithm \ref{alg:shiftedPM2} converges to $r_p(\T)$.
\end{corol}
\begin{proof}
Let  $\b w$ be any $\ell^p$-eigenvector of $\T$ with $\|\b w\|_{p}=1$ and eigenvalue $\mu$. We have
$$|\mu| = |\b w^T T(\b w)| \leq |\b w|^T T(|\b w|)  \leq \max_{\|\b x\|_p=1}f_p(\b x)=: \tilde r\, .$$
Let $\b x^*$ be a maximizer of $f_p$ and let $\b u^* = \b x^*/\|\b x^*\|_p$. As $f_p(\b x^*)=f_p(\b u^*)$ we deduce that $\b u^*$ is a maximizer too and by Lemma \ref{lem:max_is_in_cone} we can assume $\b u^*\geq 0$.  By Proposition \ref{prop:variational} we have that $\b u^*$ is an $\ell^p$-eigenvector of $\T$, that is $T(\b u^*)=\tilde r \, \Phi_p(\b u^*)$, with $\tilde r = r_p(\T)= f_p(\b u^*)$. In particular, $r_p(\T)^{\frac{1}{p-1}} \b u^* = \Phi_p^{-1}(T(\b u^*))$ implies that for all $k\geq 1$ there exists $\lambda_k>0$ such that $\lambda_k \b u^* = (\Phi_p^{-1}T)^k(\b u^*)$. Thus,  if $T^k(\b u^*)_i=0$ for some integer $k$, then also $(\b u^*)_i=0$. In other words, $\b u^*$ is zero at least in all the positions where the vectors of $C_+(\T)$ are zero.

Now, let $\b u$ be the unique eigenvector of $\T$ in $C_+(\T)$ with eigenvalue $\lambda>0$, limit of the sequences generated by either Algorithm \ref{alg:shiftedPM1} or Algorithm \ref{alg:shiftedPM2} as in Theorem \ref{thm:convergence}. Suppose that $\lambda < r_p(\T)$ and let $\alpha,\beta\in(0,1)$ be such that $\alpha^p+\beta^p=1$. We have $\b u\neq \b u^*$ and if $\b y = \alpha \b u^* + \beta \b u$, then $\b y \in C_+(\T)$ and $\|\b y\|_p=1$. Therefore,
$$
f_p(\b y) = \b y^TT(\b y) \geq (\alpha \b u + \beta \b u^*)^T \Big(\alpha^{d-1}T(\b u) + \beta^{d-1}T(\b u^*)\Big) \geq \alpha^d f_p(\b u) + \beta^d f_p(\b u^*)\, .
$$
As $\alpha = (1-\beta^p)^{1/p}$ and $p>d$, we have $\alpha^d = (1-\beta^p)^{d/p}>(1-\beta^p)>(1-\beta^d)$ and thus we can choose $\alpha$ and $\beta$ so that
$$
f_p(\b y) \geq \alpha^d f_p(\b u) + \beta^d f_p(\b u^*) = \alpha^d \lambda + \beta^d r_p(\T) > \lambda.
$$
On the other hand, the uniqueness of $\b u$ implies that $    \lambda=\max_{\b x \in C_+(\b T)} f_p(\b x)$ which yields a contradiction. Thus $\lambda = r_p(\T)$, concluding the proof.
\end{proof}

% Lemma \ref{lem:max_is_in_cone} we {can consider} $\b u^*\in C_+(\T)$ and, by definition,  we have $\nabla f_p(\b u^*)=\boldsymbol 0$. From Proposition \ref{prop:variational} we deduce that $\b u^*$ is an $\ell^p$-eigenvector of $\T$, that is $T(\b u^*)=\tilde r \, \Phi_p(\b u^*)$, with $\tilde r = f_p(\b u^*)$. As {$\b u^*\in C_+(\T)$},  the uniqueness of $\b u$ implies $\b u^*=\b u$ and $\tilde r = \lambda$, concluding that $|\mu|\leq \lambda$, i.e., $\lambda = r_p(\T)$.

Before moving forward, let us briefly comment on Theorems \ref{thm:PF} and \ref{thm:convergence}. By considering the general $\ell^p$-eigenvalue problem, Theorem \ref{thm:convergence} completes the work of \cite{liu2010always} and shows that both Algorithms \ref{alg:shiftedPM1} and \ref{alg:shiftedPM2} converge for any square tensor $\T$ and any choice of $p>d$. Moreover, by using the cone $C_+(\T)$ rather than the whole positive orthant in $\RR^n$,  we partially extend the Perron-Frobenius theorem of \cite{gautier2017tensor} to the case of tensors with some zero unfolding. In fact, Theorem \ref{thm:convergence} shows that Theorem \ref{thm:PF} holds unchanged without the assumption $T(\uno)>0$ and replacing the positive orthant  with $C_+(\T)$.
However, we point out that when $C_+(\T)=\RR^n_{++}$    and $\sigma=0$,  Theorem \ref{thm:PF} actually tells us more than the above Theorem \ref{thm:convergence}. First of all the proof of Theorem \ref{thm:PF} shows that the power sequence in that case is associated with a Lipschitz contractive map, which is a stronger property than the contractivity shown above in Theorem \ref{thm:convergence}. This allows us to show that the sequence for $\sigma=0$ converges linearly, with an explicit  convergence ratio. Moreover, Theorem \ref{thm:PF} ensures convergence of the power sequence   to $r_p(\T)$, without assuming any symmetry on $\T$. We believe that the discussed results of Theorem \ref{thm:PF} can be transferred to the case of shifted power sequences and to tensors with some zero unfolding, but this analysis goes beyond the scope of this paper and is postponed to future works.

\subsection{Dependence upon the parameter $p$}
\def\map{\b x^*\!} \def\epsilon{\varepsilon} \def\ep{\epsilon}
{In this subsection we show that the Perron $\ell^p$-eigenvector depends continuously on the parameter $p$. To this end, let us first define the following maps,}
\begin{equation*}
	H(\b x, p):= \Phi_p^{-1}(T(\b x))\quad  \hbox{ and } \quad  \tilde{H}(\b x, p):= \frac{H(\b x, p)}{\|H(\b x, p)\|_p}.
\end{equation*}
{Note that $H(\b x, p)=H_0(\b x)=K_0(\b x)$ (respectively $\tilde{H}(\b x, p)=\tilde{H}_0(\b x)=\tilde{K}_0(\b x)$) being $H_0, \; K_0$  (respectively $\tilde{H}_0, \; \tilde{K}_0$) the maps introduced in the proof of Theorem \ref{thm:convergence} for the choice $\sigma=0$.}
%Note that, these maps coincide with either of the ones  introduced in the proof of Theorem \ref{thm:convergence} ($H_\sigma$ or $K_\sigma$, for example) for the choice $\sigma=0$.
With this notation and due to Theorem \ref{thm:convergence} we have that for any $p>d$, there exists a unique ${\b x}^{*}$ in $C_+(\b T)$ such that $\tilde{H}(\b x^*, p)=\b x^*$ and such that $\b x^*$ is the unique Perron $\ell^p$-eigenvector of $\T$.  Thus, with a slight abuse of notation, consider the following mapping,
$$
\map:(d,+\infty) \to C_+(\T), \qquad p \mapsto \map(p) \text{ such that }\tilde{H}(\map(p), p)=\map(p)\,,
$$
which to a given $p\in (d,+\infty)$  assigns the unique Perron $\ell^p$-eigenvector of $\T$.

In order to prove  Theorem \ref{theo:fixed_point_continuity}, which states the continuous dependence on the parameter $p$,  we need the following additional technical lemma.
\begin{lemma} \label{lem:taylor_phi}
	Let  $\T\geq 0$ and $p >d\geq 2$. % and let $q$ be its conjugate exponent such that $1/p+1/q=1$.
	For $\b x \in C_+(\T)$ it holds
	\begin{equation}
	\Phi_{p+\epsilon}^{-1}(\b x)=\Phi_p^{-1}(\b x)+ \Psi(\b x) \epsilon +o(\epsilon)\, ,
	\end{equation}
	where $\Psi(\b x)_i = -\frac{1}{(p-1)^2}\,x_i^{\frac{1}{p-1}} \ln (x_i)$ if $x_i>0$ and $\Psi(\b x)_i = 0$, otherwise.
\end{lemma}

% \begin{lemma} \label{lem:taylor_phi}
% 	Let us consider $p_1 >d$ and define $q_{1,\epsilon}$ the conjugate exponent of $p_1+\epsilon$, i.e., $q_{1,\epsilon}:=(p_1+\epsilon)/(p_1-1+\epsilon)$. For $\b x \in \RR^n_{++}$ it holds
% 	\begin{equation}
% 		\Phi_{q_{1,\epsilon}}(\b x)=\Phi_{q_1}(\b x)+\b \phi(\b x) \epsilon +o(\epsilon).
% 	\end{equation}
% \end{lemma}
\begin{proof}
	Let $q_{\epsilon}=(p+\epsilon)/(p-1 +\epsilon) = 1 + 1/ (p-1 +\epsilon)$ be the conjugate exponent of $p+\epsilon$. For $ \epsilon$ {sufficiently small} it holds
	\begin{equation}
	q_\epsilon -1 =  \frac{1}{p-1 +\epsilon}=\frac{1}{p-1}-\frac{1}{(p-1)^2} \epsilon +o(\epsilon),
	\end{equation}
	and thus, for a sufficiently small $\epsilon$  we can write,
	\begin{equation}\label{eq:taylor1}
	\Phi_{p+\epsilon}^{-1}(\b x) = \Phi_{q_{\epsilon}}(\b x)=\b x^{\frac{1}{p-1}}\b x^{-\frac{1}{(p-1)^2} \epsilon +o(\epsilon)}.
	\end{equation}
	Moreover, since for $x>0$ it holds $x^{\epsilon}=1+\epsilon \ln(x)+o(\epsilon)$,
	we have
	\begin{equation} \label{eq:taylor2}
	x_i^{-\frac{1}{(p-1)^2} \epsilon +o(\epsilon)}=1-\frac{1}{(p-1)^2} \ln (x_i) \epsilon+o(\epsilon),
	\end{equation}
	and hence the thesis follows by combining \eqref{eq:taylor2} and \eqref{eq:taylor1}. % the thesis follows, being $\phi(\b x):=-\frac{1}{(p_1-1)^2}(\b x)^{\frac{1}{p_1-1}} \ln (\b x)$
\end{proof}
\begin{theorem} \label{theo:fixed_point_continuity} If $\T\geq 0$ and $p>d\geq 2$, then for any small enough $\ep>0$ we have %and assume that  $\tilde{H}_{\sigma}(\b x, p)$ is a strict contraction, i.e.\ there exists $\gamma \in (0,1)$ such that $d_H(\tilde{H}_{\sigma}(\b x, p),\tilde{H}_{\sigma}(\b y, p)) \leq \gamma \,  d_H(\b x, \b y)$, for any two $\b x, \b y \in C_+(\T)$. Then $\map$ is Lipschitz continuous, i.e.,
$$
d_H\big(\map(p),\map(p+\epsilon)\big) = O(\epsilon)\, ,
$$
where $d_H$ is the Hilbert distance defined in \eqref{eq:hilbert-distance}. % and $\epsilon>0$ is small enough.
\end{theorem}
\begin{proof}
%Let $p_1=p>d\geq 2$ and let $p_2=p+\epsilon$, for a sufficiently small $\epsilon>0$.
Using the triangle inequality and since the Hilbert metric is scale invariant, we have
\begin{align}\label{eq:s1}
\begin{aligned}
&d_{H}(\map(p),\map(p+\ep)) = d_{H}(\tilde{H}(\map(p), p),\tilde{H}(\map(p+\ep), p+\ep)) \\
&\leq
d_H(\tilde{H}(\map(p), p),\tilde{H}(\map(p+\ep), p))+
d_H(\tilde{H}(\map(p+\ep), p+\ep),\tilde{H}(\map(p+\ep), p))\\
&=
\underbrace{d_H(H(\map(p), p),H(\map(p+\ep), p))}_{(a)}+
\underbrace{d_H(H(\map(p+\ep), p+\ep),H(\map(p+\ep), p)).}_{(b)}
\end{aligned}
\end{align}	
We now upper bound both terms in $(a)$ and in $(b)$. Concerning part $(a)$, note that using \eqref{eq:a} and \eqref{eq:b} for $\sigma=0$, for any two $\b x, \b y \in C_+(\T)$ we have %from the definition of $d_H$ and $H$, for any two  we immediately have
\begin{equation*}
d_H(H(\b x),H(\b y))=\log\left(\frac{M(H(\b x)/H(\b y))}{m(H(\b x)/H(\b y))}\right)\leq \log\left(\left\{\frac{M(\b x/\b y)}{m(\b x/\b y)}\right\}^{\frac{d-1}{p-1}}\right)=\Big(\frac{d-1}{p-1}\Big)d_H(\b x,\b y)
\end{equation*}
thus
\begin{equation}\label{eq:s2}
d_H(H(\map(p), p),H(\map(p+\epsilon), p)) \leq\gamma\, d_H(\map(p),\map(p+\epsilon)).
\end{equation}
with $\gamma = {(d-1)}/{(p-1)}$. As for  part $(b)$, notice that, for any $\b x \in C_+(\T)$ and $\epsilon$ sufficiently small, using Lemma \ref{lem:taylor_phi}, we have
\begin{align*}
	H(\b x, p+\epsilon)&=\Phi_{p+\ep}^{-1} (T(\b x))  = \Phi_{p}^{-1}(T(\b x))+\ep \Psi(T(\b x)) + o(\ep) \\
	&=H(\b x, p)+\epsilon \Psi(T(\b x))+o(\epsilon)\, .
\end{align*}
% where the previous equation is meant to hold just for the components s.t. $(\b x^{*}(p_1+\epsilon))_i>0$.
Hence we have,
\begin{equation*}
\begin{split}
& M(H(\b x, p+\ep) /H(\b x, p)) = \max_{i:x_i\neq 0}\frac{H(\b x, p)_i+\epsilon \Psi(T(\b x))_i+o(\epsilon)}{H(\b x, p)_i} \leq  1 + \epsilon C_1  +o(\epsilon),
\end{split}
\end{equation*}
 and
\begin{equation*}
\begin{split}
m(H(\b x, p+\ep) /H(\b x, p)) = \min_{i:x_i\neq 0}\frac{H(\b x, p)_i+\epsilon \Psi(T(\b x))_i+o(\epsilon)}{H(\b x, p)_i} \geq  1 + \epsilon C_2  +o(\epsilon),
%
% & m({H}_{\sigma}(\b x^*(p_1+\epsilon), p_1)+\epsilon \phi(T(\b x^{*}(p_1+\epsilon)))+o(\epsilon)/{H}_{\sigma}(\b x^*(p_1+\epsilon), p_1))\geq \\
% & \geq  1-C_1 \epsilon +o(\epsilon)
\end{split}
\end{equation*}
where
$$
C_1:=\max_{i : x_i \neq 0}  \frac{\Psi(T(\b x))_i}{H(\b x, p)_i}\qquad \text{and} \qquad C_2:=\min_{i : x_i \neq 0}  \frac{\Psi(T(\b x))_i}{H(\b x, p)_i}\, .
$$
Eventually, we obtain
\begin{equation*}
d_H(H(\b x, p+\ep), H(\b x, p)) \leq \log\Big(\frac{1+\ep C_1 + o(\ep)}{1+\ep C_2 +o(\ep)}\Big) = \ep (C_1-C_2) + o(\ep)
	% d_H({H}_{\sigma}(\b x^*(p_1+\epsilon), p_1),{H}_{\sigma}(\b x^*(p_1+\epsilon), p_1+\epsilon)) \leq \log(\frac{1+C_1 \epsilon +o(\epsilon)}{1-C_1 \epsilon +o(\epsilon)})=\epsilon 2 C_1+o(\epsilon).
\end{equation*}
Using the above inequality for $\b x = \map(p+\ep)$, together with  \eqref{eq:s1} and \eqref{eq:s2} we obtain
\begin{equation*}
d_{H}(\map(p),\map(p+\ep)) \leq \ep \frac{(C_1-C_2+1)}{(1-\gamma)}	
\end{equation*}
which concludes the proof.
\end{proof}

Theorem \ref{theo:fixed_point_continuity} shows that when $p>d$, the Perron $\ell^p$-eigenvector of a nonnegative tensor depends continuously on the parameter $p$. This theorem can be useful for example in the context of $H$-eigenvectors  because it allows us to argue that, in order to approximate a $H$-eigenvector of {$\T\geq 0$} we can use an $\ell^p$-eigenvector with $p\approx d$. In fact,  note that for nonnegative eigenvectors the duality map $\Phi_p$ coincides with taking the entrywise powers of the vector and thus $\ell^d$-eigenvectors coincide with $H$-eigenvectors. Note moreover that, if the contractivity property \eqref{eq:s2} holds also for the case $p=d$ (which does happen in practice, as discussed in \cite{gautier2018contractivity,fasino2019higher}), then  Theorem \ref{theo:fixed_point_continuity} would work also for $p=d$. This is particularly interesting because this argument holds without any irreducibility requirement on $\T$. Instead, for the case of $H$-eigenvectors, the convergence of the power method requires the tensor to be weakly irreducible (see {Definition \ref{def:wi} and} \cite[Thm.\ 3.3]{gautier2017tensor} or \cite[Thm.\ 3.1]{zhou2013efficient}), which is an expensive property to verify, especially for large tensors.

So, for example, if $T(\uno)>0$ but $\T\geq 0$ is not weakly irreducible, we can use the positive Perron $\ell^p$-eigenvector of $\T$ with $p\approx d$ to compute an approximation of a positive $H$-eigenvector. To our knowledge, there exists no other method with this type of guarantees to approximate a positive $H$-eigenvector of a reducible tensor.
We clarify this idea with an example in the {remainder} of this section.
\subsubsection{Computing positive $H$-eigenvectors {(case $p=d$)}} \label{sec:approx_Heigenpair} $ $
Given a nonnegative tensor $\T \in \RR^{[d,n]}$, consider the graph $G(\T)=(V,E)$ defined as follows: $V = \{1,\dots,n\}$ and  there is an edge from node $u$ to node $v$, i.e.\ $uv\in E$, if
	$$
	t_{u, j_2,\dots,j_{m-1},v,j_{m+1}, \dots, j_d}>0
	$$
	for some $m=2,3,\dots,d-1$.  We have
\begin{definition}[Weakly irreducible tensor] \label{def:wi}
	The tensor $\T$ is said to be weakly irreducible or, in the more general notation of \cite{gautier2017tensor}, $\{1,\dots,n\}$-weakly irreducible if $G(\T)$ is strongly connected.
\end{definition}

In \cite{gautier2017tensor}, points $(iii)$ of Theorems 3.1 and 3.3, it is shown that if a tensor is weakly irreducible then there exists a unique positive $H$-eigenvector and the power method converges to such $H$-eigenvector. However, in general, if the tensor is not weakly irreducible, then we have no guarantees of  convergence for the power method to a $H$-eigenpair (see \cite{gautier2017tensor}). Instead, {using our Theorem \ref{thm:convergence},} convergence is still ensured when $p>d$ and we can use values of $p$ slightly larger than $d$ to efficiently compute good approximations to positive $H$-eigenvectors.

For example, consider the following $3\times 3\times 3$ nonnegative tensor
\begin{equation}\label{eq:example_T}
\T(:,:,1) = \begin{bmatrix}0 &0 &1\\ 0 &0 &1\\ 0& 0& 0\end{bmatrix}\quad \T(:,:,2) = \begin{bmatrix}0 &1 &0\\ 0 &0 &0\\ 0& 0& 0\end{bmatrix} \quad \T(:,:,3) = \begin{bmatrix}0 &1 &0\\ 0 &0 &0\\ 0& 0& 1\end{bmatrix}\, .
\end{equation}
It is readily seen that this tensor is not weakly irreducible. In fact,  the graph of this tensor is
\begin{center}
\begin{tikzpicture}[thick, main node/.style={circle, fill=white, draw=black, align=center,inner sep=2,scale=1.2}]
  \newcommand*{\MainNum}{3}
  \newcommand*{\MainRadius}{1cm}
  \newcommand*{\MainStartAngle}{90}

  % Print main nodes, node names: p1, p2, ...
  \path
    (0, 0) coordinate (M)
    \foreach \t [count=\i] in {$3$,$1$,$2$}  {
      +({\i-1)*360/\MainNum + \MainStartAngle}:\MainRadius)
      node[main node, align=center] (p\i) {\t}
    }
  ;

  % Calculate the angle between the equal sides of the triangle
  % with side length \MainRadius, \MainRadius and radius of circle node
  % Result is stored in \p1-angle, \p2-angle, ...
  \foreach \i in {1, ..., \MainNum} {
    \pgfextracty{\dimen0 }{\pgfpointanchor{p\i}{north}}
    \pgfextracty{\dimen2 }{\pgfpointanchor{p\i}{center}}
    \dimen0=\dimexpr\dimen2 - \dimen0\relax
    \ifdim\dimen0<0pt \dimen0 = -\dimen0 \fi
    \pgfmathparse{2*asin(\the\dimen0/\MainRadius/2)}
    \global\expandafter\let\csname p\i-angle\endcsname\pgfmathresult
  }

  % Draw the arrow arcs
  \pgfmathsetmacro\StartAngle{ 0 + \MainStartAngle + \csname p1-angle\endcsname }
  \pgfmathsetmacro\EndAngle{  360/\MainNum + \MainStartAngle - \csname p2-angle\endcsname }
  \draw[{<[scale=1.3]}-]  (M) ++(\StartAngle:\MainRadius) arc[start angle=\StartAngle, end angle=\EndAngle, radius=\MainRadius];

  \pgfmathsetmacro\StartAngle{ 360/\MainNum + \MainStartAngle + \csname p2-angle\endcsname }
  \pgfmathsetmacro\EndAngle{  2*360/\MainNum + \MainStartAngle - \csname p3-angle\endcsname }
  \draw[{<[scale=1.3]}-{>[scale=1.5]}]  (M) ++(\StartAngle:\MainRadius) arc[start angle=\StartAngle, end angle=\EndAngle, radius=\MainRadius];

  \pgfmathsetmacro\StartAngle{ 2*360/\MainNum + \MainStartAngle + \csname p3-angle\endcsname }
  \pgfmathsetmacro\EndAngle{  360 + \MainStartAngle - \csname p1-angle\endcsname }
  \draw[-{>[scale=1.3]}]  (M) ++(\StartAngle:\MainRadius) arc[start angle=\StartAngle, end angle=\EndAngle, radius=\MainRadius];

  \node at (-2,0) {$G(\T)=$};
\end{tikzpicture}
\end{center}
which is certainly not strongly connected. Thus, uniqueness of a nonnegative $H$-eigenvector and convergence of the power method are not guaranteed. In fact, for this particular tensor, the eigenvector equation $T(\b x) = \lambda \Phi_p(\b x)$ for $p=d$ boils down to the system of equations
\begin{equation*}
\left\{\begin{array}{l}
x_3x_1 + x_2^2 + x_2x_3 =  \lambda\, x_1|x_1| \\
x_3x_1  = \lambda\, x_2 |x_2|\\
x_3^2 =  \lambda\, x_3|x_3|
\end{array}\right.
\end{equation*}
and thus, a few algebraic manipulations imply that all the $H$-eigenpairs  of $\T$ are (up to normalization) of the form
$$
\lambda=1,\; \b u = \begin{bmatrix}\mu\,|\mu|\\ \mu \\ 1\end{bmatrix} \qquad \text{or} \qquad \lambda=-1,\; \b v = \begin{bmatrix}\eta \, |\eta|\\ \eta \\ -1\end{bmatrix}
$$
where $\mu$ and $\eta$ are real numbers such that
$$
\mu(\mu^3-2\mu-1)(-\mu^3-1)=0,\qquad \eta(\eta^3-1)(-\eta^3+2\eta-1)=0\, .
$$
Therefore, the following nonnegative $H$-eigenvectors of \eqref{eq:example_T} correspond to the positive eigenvalue $\lambda =1$
$$
\b u_1 = \begin{bmatrix}0\\ 0 \\ 1\end{bmatrix}\qquad \text{and}\qquad \b u_2 = \frac 1 2 \begin{bmatrix}3+\sqrt 5\\ 1+\sqrt 5  \\ 2\end{bmatrix}
$$
and any of their nonnegative multiple.

\begin{figure}[t]
\centering
\includegraphics[width=.35\textwidth]{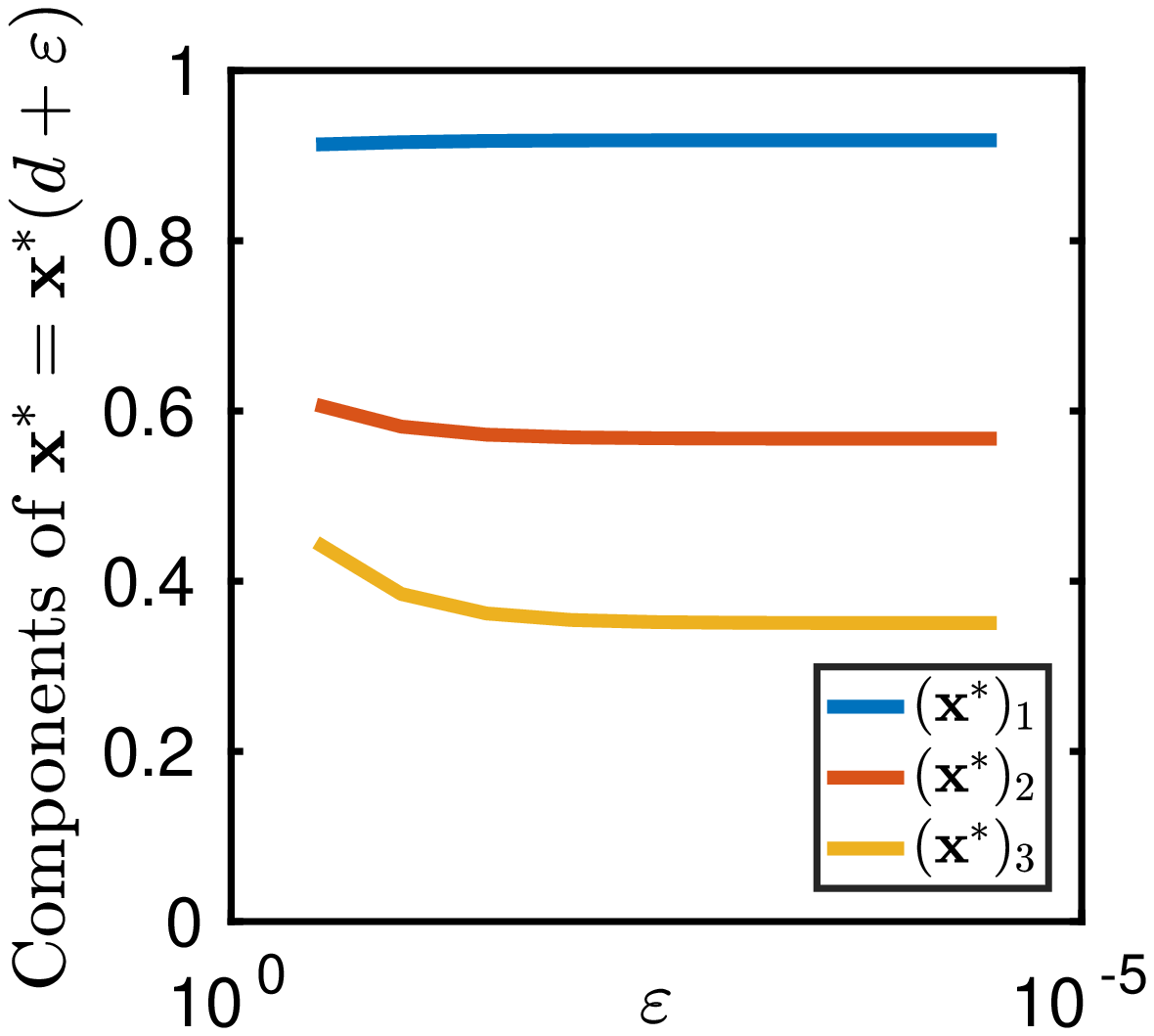}\hspace{2em}
\includegraphics[width=.35\textwidth]{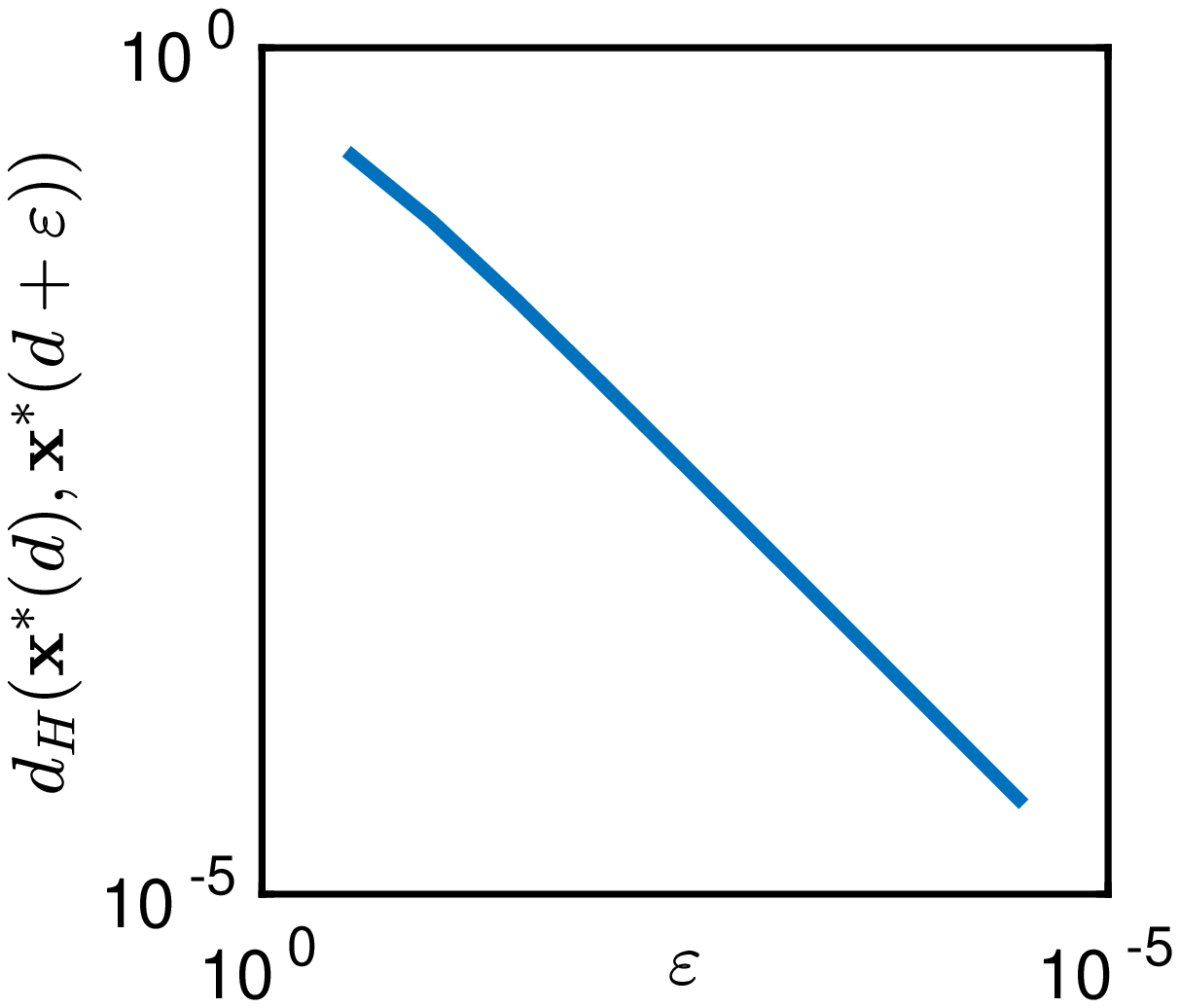}
\caption{Left: The three entries of the unique positive Perron $\ell^p$-eigenvector $\map(d+\ep)$, as $\ep$ decreases from $10^{-1}$ to $10^{-5}$. Right: LogLog plot of the Hilbert distance between $\map(d+\ep)$ and the $H$-eigenvector ${\tilde{\textbf{u}}_2}$, as $\ep$ decreases. The plot confirms the behavior of Theorem \ref{theo:fixed_point_continuity}. }\label{fig:dependence_on_p}
\end{figure}

On the other hand, $T(\uno) = [3, 1, 1]^T >0$, and thus $C_+(\T) = \RR^3_{++}$. Thus, Theorem \ref{thm:convergence} implies that for any $p>d=3$, such a tensor has a unique Perron $\ell^p$-eigenvector $\map(p)$ which is entrywise positive and which can be computed via either of Algorithms \ref{alg:shiftedPM1} and \ref{alg:shiftedPM2}. Moreover, Theorem \ref{theo:fixed_point_continuity} suggests that for $p\approx d$ we have $\map(p)\approx {\tilde{\textbf{u}}_2}$ {being $\tilde{\textbf{u}}_2:= \b u_2/ \|\b u_2\|_d $}, i.e.\ the Perron $\ell^p$-eigenvector approximates the {(normalized)} positive $H$-eigenvector. This phenomenon is shown in Figure \ref{fig:dependence_on_p}, which plots the values of the entries of  $\map(d+\epsilon)$ and the distance $d_H(\map(d+\ep),{\tilde{\textbf{u}}_2})$ for the tensor in \eqref{eq:example_T}, as $\ep>0$ decreases to zero.

\section{Numerical Experiments: Part 1}\label{sec:experiments_part1}
In this section we investigate experimentally the power method in its two shifted variants  described in Algorithms \ref{alg:shiftedPM1} and \ref{alg:shiftedPM2} with different values of the shift $\sigma$ and $p$ and on different test problems. For the sake of clarity we use  different subsections to describe the problem set and present the corresponding numerical results. All the numerical experiments are performed on a laptop running Linux with 16Gb memory and CPU Intel\textsuperscript{\textregistered} Core\texttrademark\ i7-4510U with clock 2.00GHz. The code is written and executed in MATLAB R2018b.

Let $\b x_k$ be the sequence of vectors generated by any of the methods when applied to address the $\ell^p$-eigenvector problem for the tensor $\T$, and let $\lambda_k = f_p(\b x_k)$ as defined in \eqref{eq:rayleigh}. All the presented figures in the following sections show the behavior of  the residual
\begin{equation} \label{eq:residual}
{\|T(\b x_k)-\lambda_k\Phi_p(\b x_k)\|_\infty}%{\|\b x_k\|_\infty}
\end{equation}
{up to iteration $30$.}

It is worth pointing out that often (e.g.\ in \cite{gleich2015multilinear,kolda2011shifted}) the convergence plots of power sequences for tensors show the relative error between consecutive iterates $\|\b x_k-\b x_{k+1}\|/\|\b x_{k+1}\|$ or the behavior of the eigenvalue sequence $\lambda_k$, rather than the point-wise residual \eqref{eq:residual}.
%However we decided to show the behavior of the latter quantity as we believe it gives more insights on the quality of the computed approximation.
For the sake of brevity here we do not show plots of the relative error nor of the eigenvalue sequence. However we underline that: (a) the relative error between consecutive iterates $\|\b x_k-\b x_{k+1}\|_p/\|\b x_{k+1}\|_p$, as expected, always decreases at least as fast as the shown point-wise residual \eqref{eq:residual} and (b) the eigenvalue sequence always stabilizes for the test problem considered.

{In the following numerical results, as pointed out in Section \ref{sec:approx_Heigenpair}, we will use $p \approx d$ in order to investigate the numerical properties of Algorithms \ref{alg:shiftedPM1} or \ref{alg:shiftedPM2} when used to approximate $H$-eigenpairs.  Let us add, moreover, that this is the only numerical relevant case since, as shown in the proof of Theorem \ref{theo:fixed_point_continuity} the contractivty of our fixed point map increases accordingly to $d$.} {Finally, let us point out that in all the following numerical results we use a random vector ${\b{x}_0}$ as initial guess obtained using MATLAB's function \texttt{rand}}.

In all the figures we denote by \textsf{Alg1} and \textsf{Alg2} the sequences computed by Algorithms \ref{alg:shiftedPM1} or \ref{alg:shiftedPM2}, respectively, and by \textsf{PM} the power method sequence i.e.\ the case where $\sigma=0$.

\subsection{Nonnegative tensors with different irreducibility pattern}
As mentioned above, a shifted version of the power method for $H$-eigenvalue problems for nonnegative tensors has been introduced in \cite{liu2010always}. It is simple to note that, {in this case, i.e.,} when $p=d$, {the latter }method coincides with  Algorithm \ref{alg:shiftedPM2} proposed above.  In what follows, we refer to that method as the LZI algorithm, following the notation of \cite{zhang2012linear}. The latter paper introduces the notion of weakly positive tensor and proposes a convergence analysis of the LZI algorithm, proving the linear convergence of the method  for weakly positive tensors. However, due to Theorem \ref{thm:convergence}, {this} requirement on the structure of the tensor is unnecessary when $p>d$.
In this section we analyze the behavior of Algorithms \ref{alg:shiftedPM1} and \ref{alg:shiftedPM2} on the three test problems $\mathbf A, \mathbf B$ and $\mathbf C$ defined in \eqref{eq:tensorsABC} and considered in \cite{zhang2012linear}. All these three tensors are square of size $n\times n \times n$. $\mathbf A$ is irreducible, but not primitive nor weakly positive; $\mathbf B$ is primitive and weakly positive, but not essentially positive; $\mathbf C$ is primitive but not weakly positive. This implies that $C_+(\mathbf A)= C_+(\mathbf B)=C_+(\mathbf C)=\RR^n_{++}$.
\begin{align}\label{eq:tensorsABC}
\begin{aligned}
%\begin{equation}\label{eq:tensorA}
\mathbf A &= (a_{ijk})\quad \text{with} \quad a_{ijk}=\begin{cases}
1 & i=1,\,  j=k, \, 2\leq j\leq n\\
1 & j=k=1, \, 2\leq i \leq n\\
0 & otherwise
\end{cases}
%\end{equation}
\\
%\begin{equation}\label{eq:tensorB}
\mathbf B &= (b_{ijk})\quad \text{with} \quad b_{ijk}=\begin{cases}
i+j & j=k, \, i\neq j, \, 1\leq i,j\leq n\\
0 & otherwise
\end{cases}
%\end{equation}
\\
%\begin{equation}\label{eq:tensorC}
\mathbf C &= (c_{ijk})\quad \text{with} \quad c_{ijk}=\begin{cases}
1 & i=1, \, j=k=n\\
1 & j=k=1, \, 2\leq i\leq n\\
1 & i=n, \, 1\leq j=k\leq n-1\\
0 & otherwise
\end{cases}
\end{aligned}
\end{align}
\begin{figure}[p]
	\centering
	\includegraphics[width=.35\textwidth]{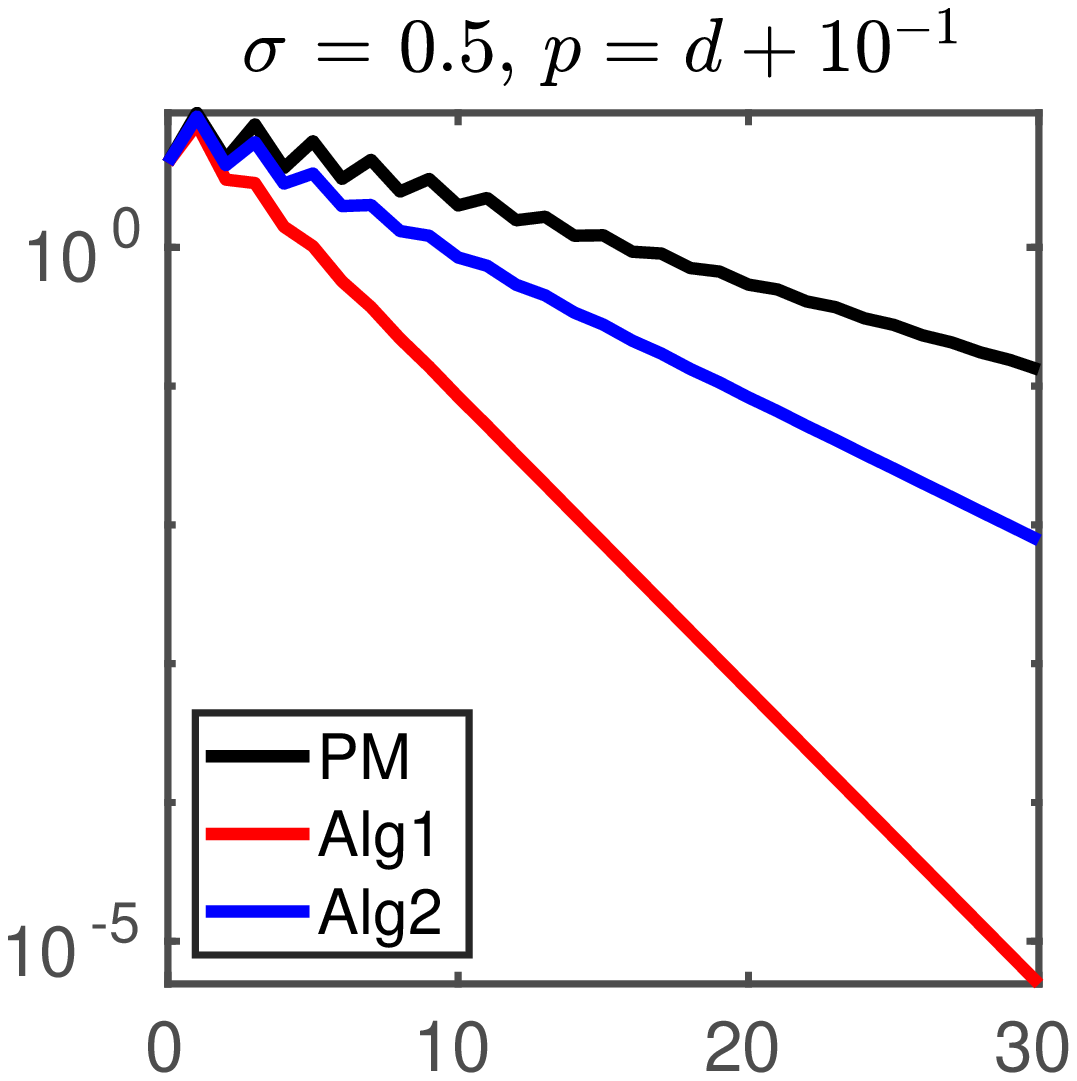}\quad
	\includegraphics[width=.35\textwidth]{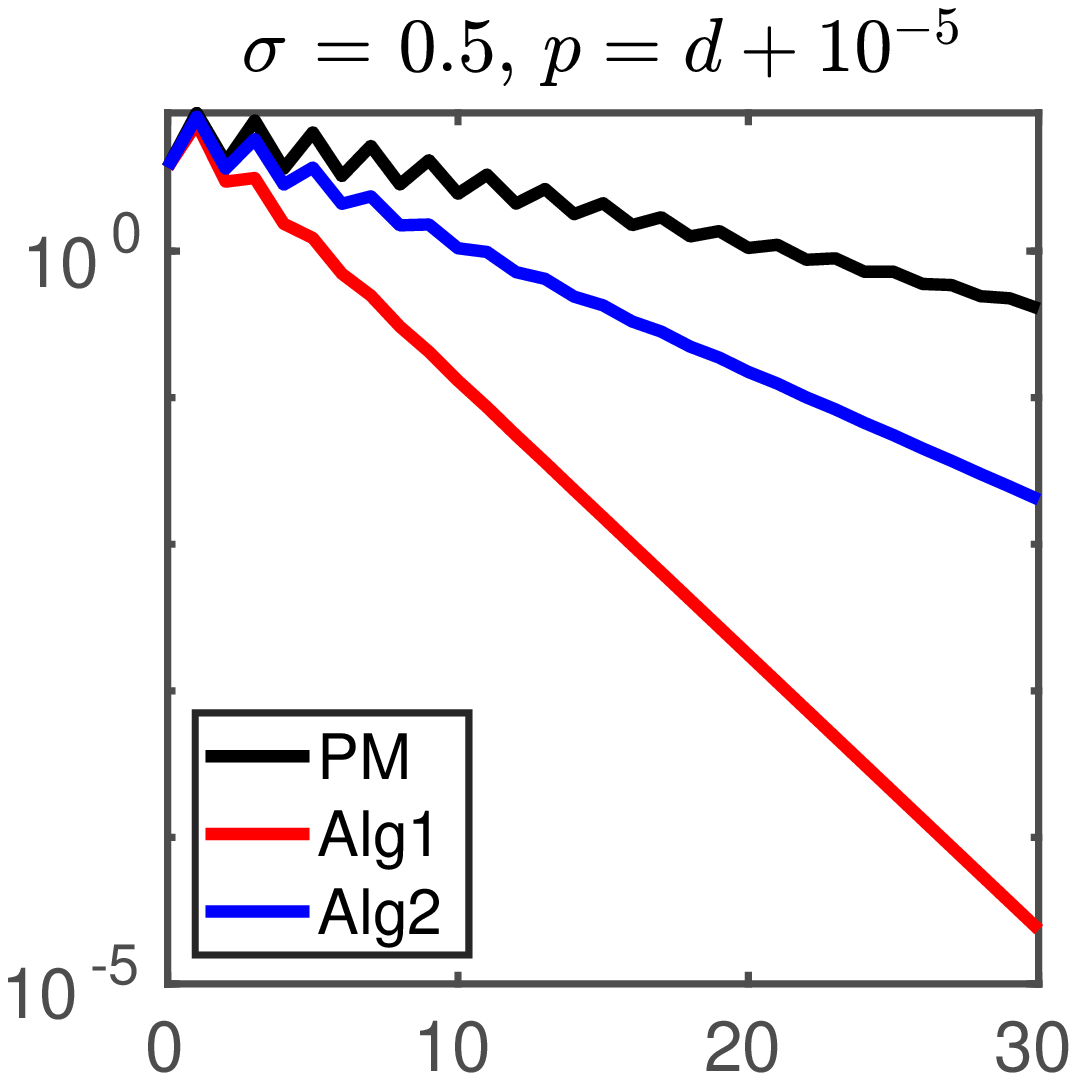}\\[1em]
	\includegraphics[width=.35\textwidth]{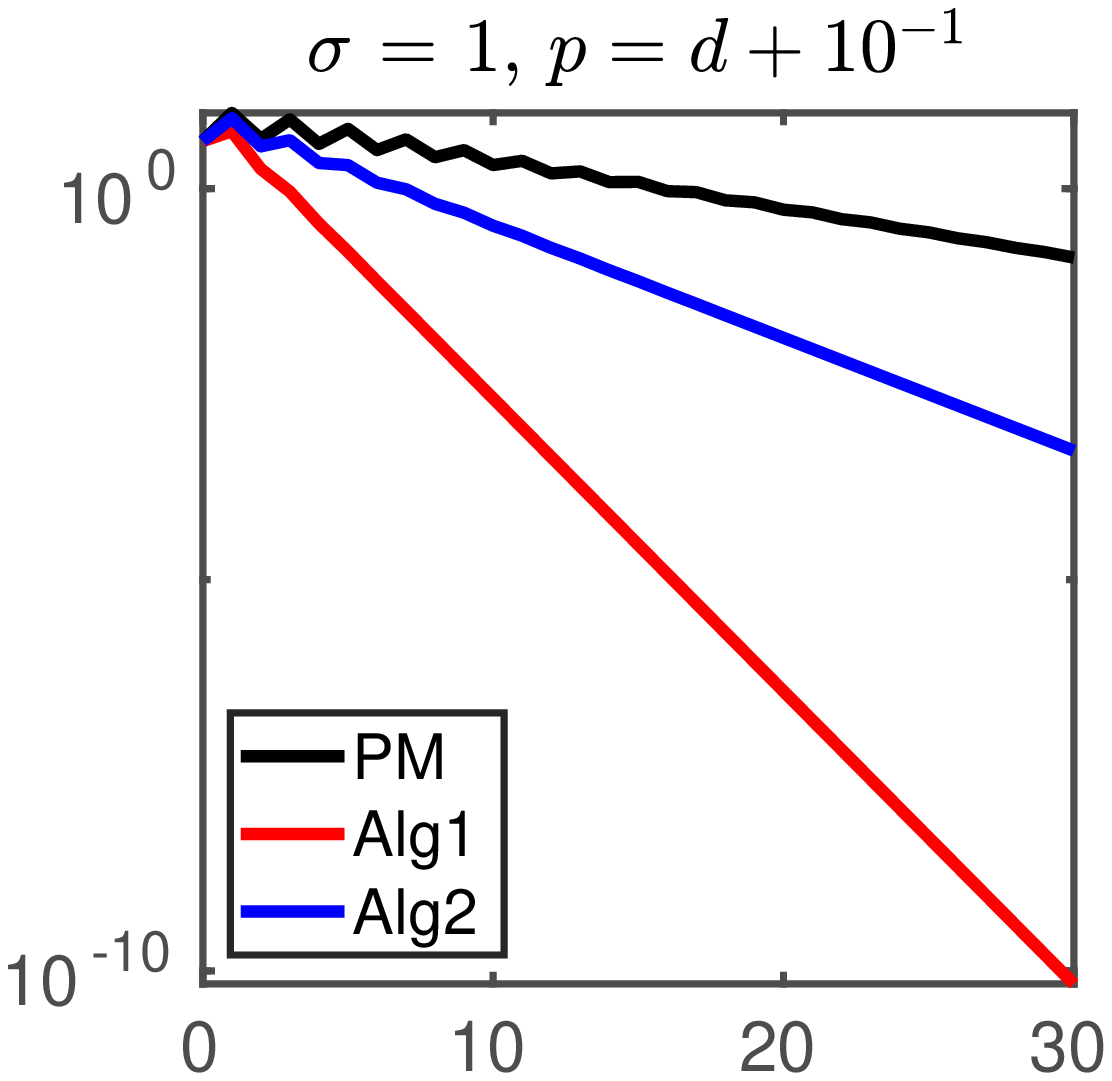}\quad
	\includegraphics[width=.35\textwidth]{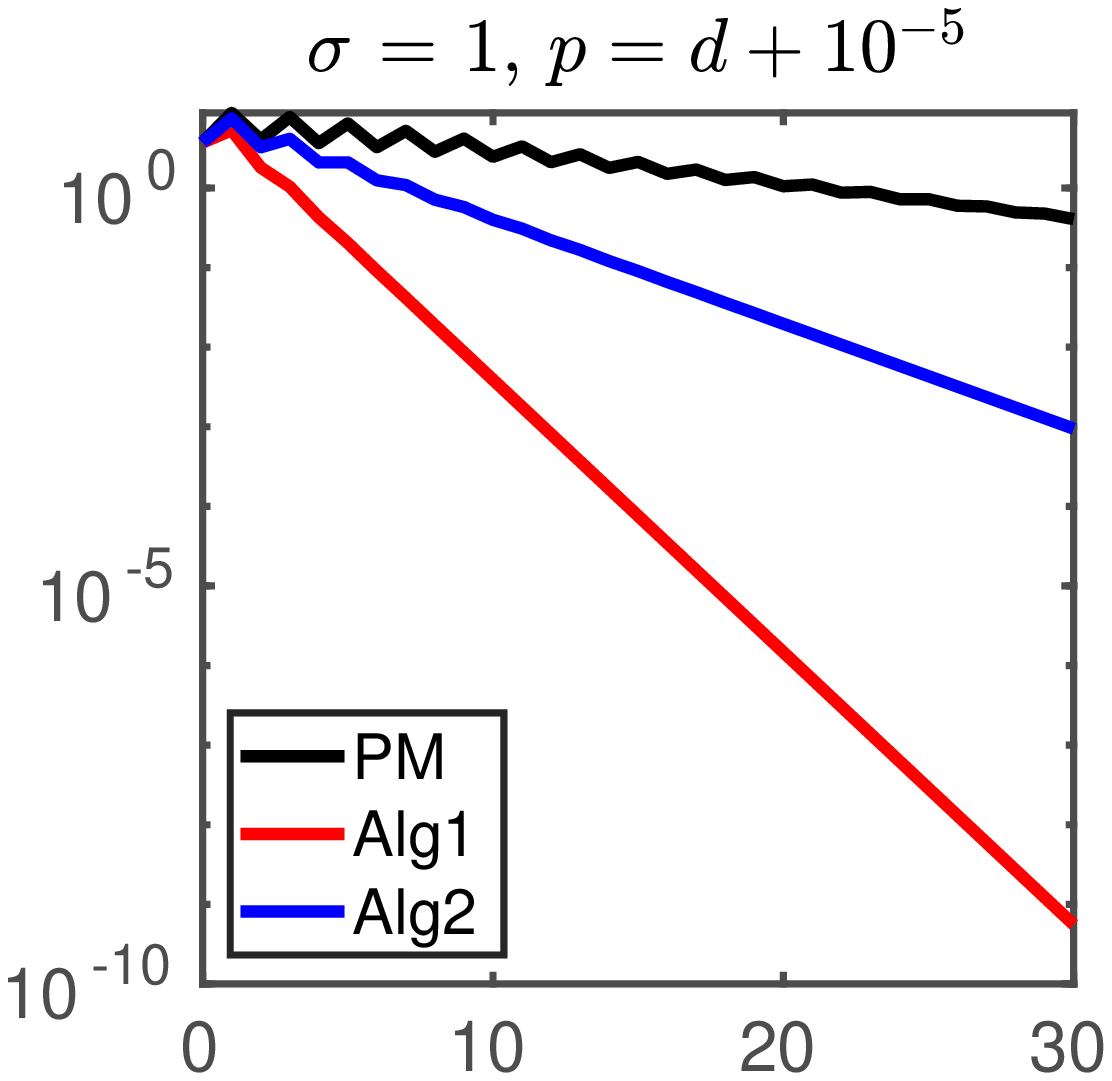}
\vspace{-0.3cm}
\caption{Experiments on the test problem $\mathbf A \in \RR^{n\times n\times n}$, $n=100$, considered in \cite{zhang2012linear} and defined in \eqref{eq:tensorsABC}.} \label{fig:tensorsA}

\vspace{0.3cm}
	\centering
	\includegraphics[width=.35\textwidth]{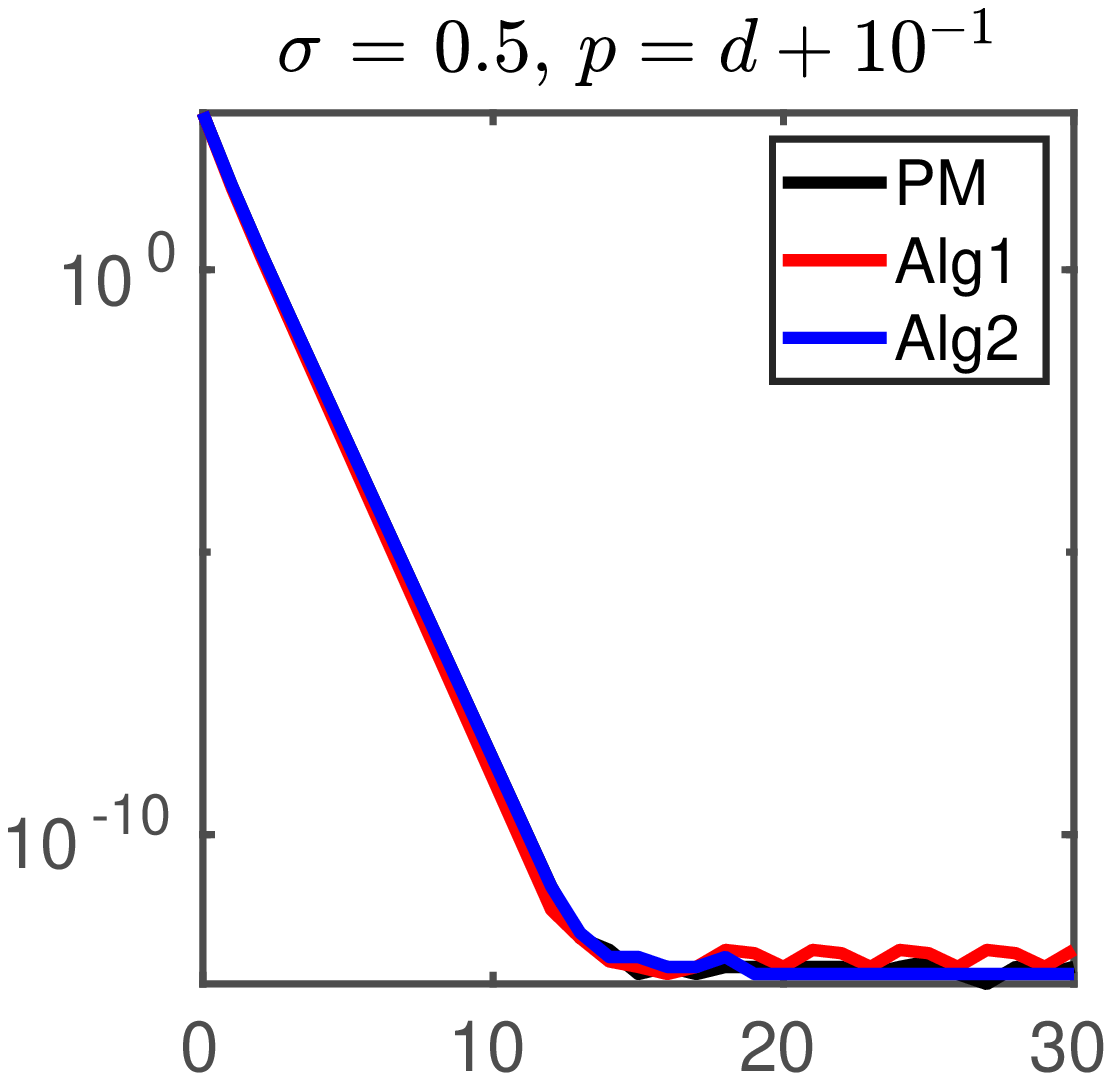}\quad
	\includegraphics[width=.35\textwidth]{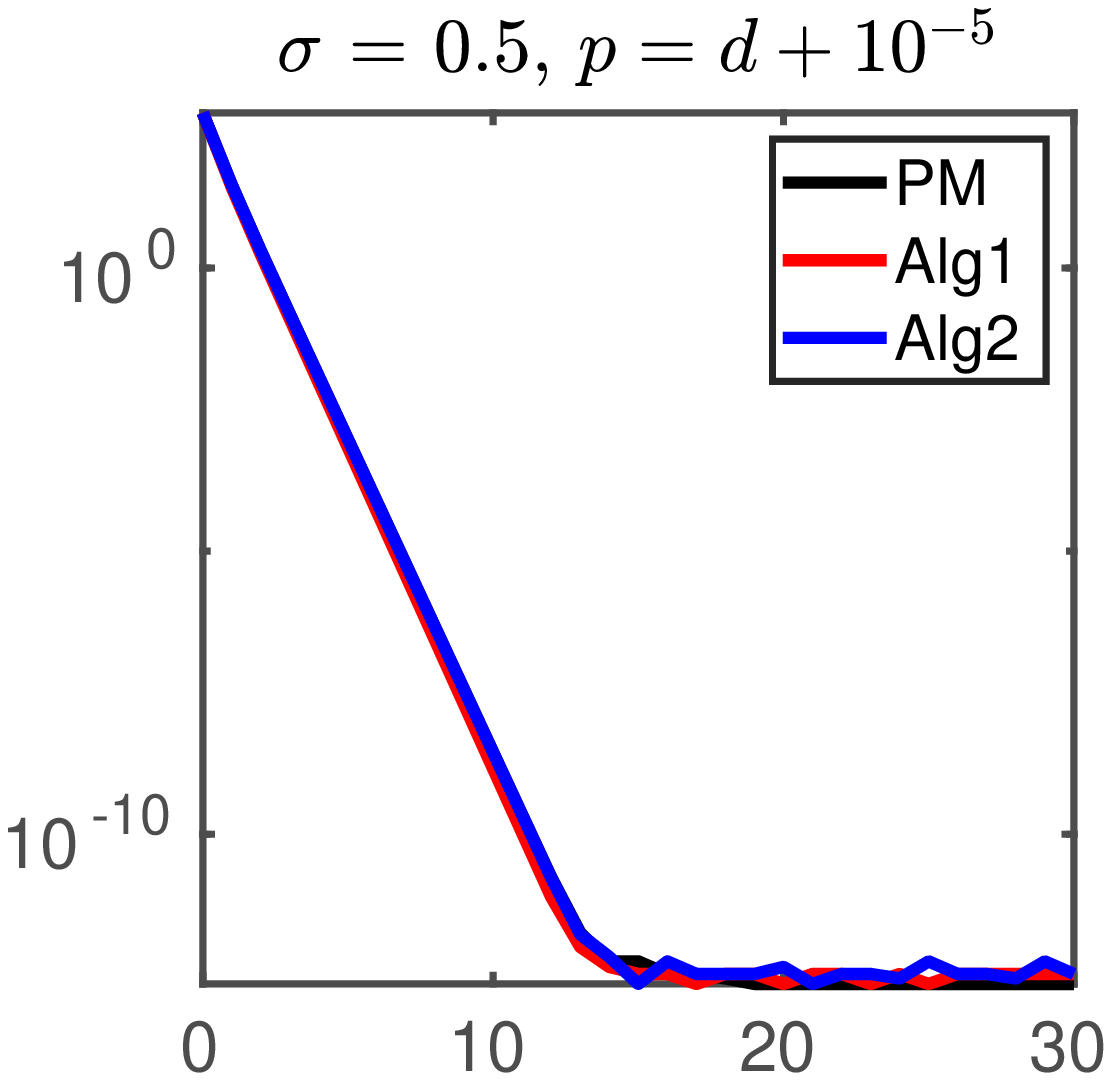}\\[0.3em]
	\includegraphics[width=.35\textwidth]{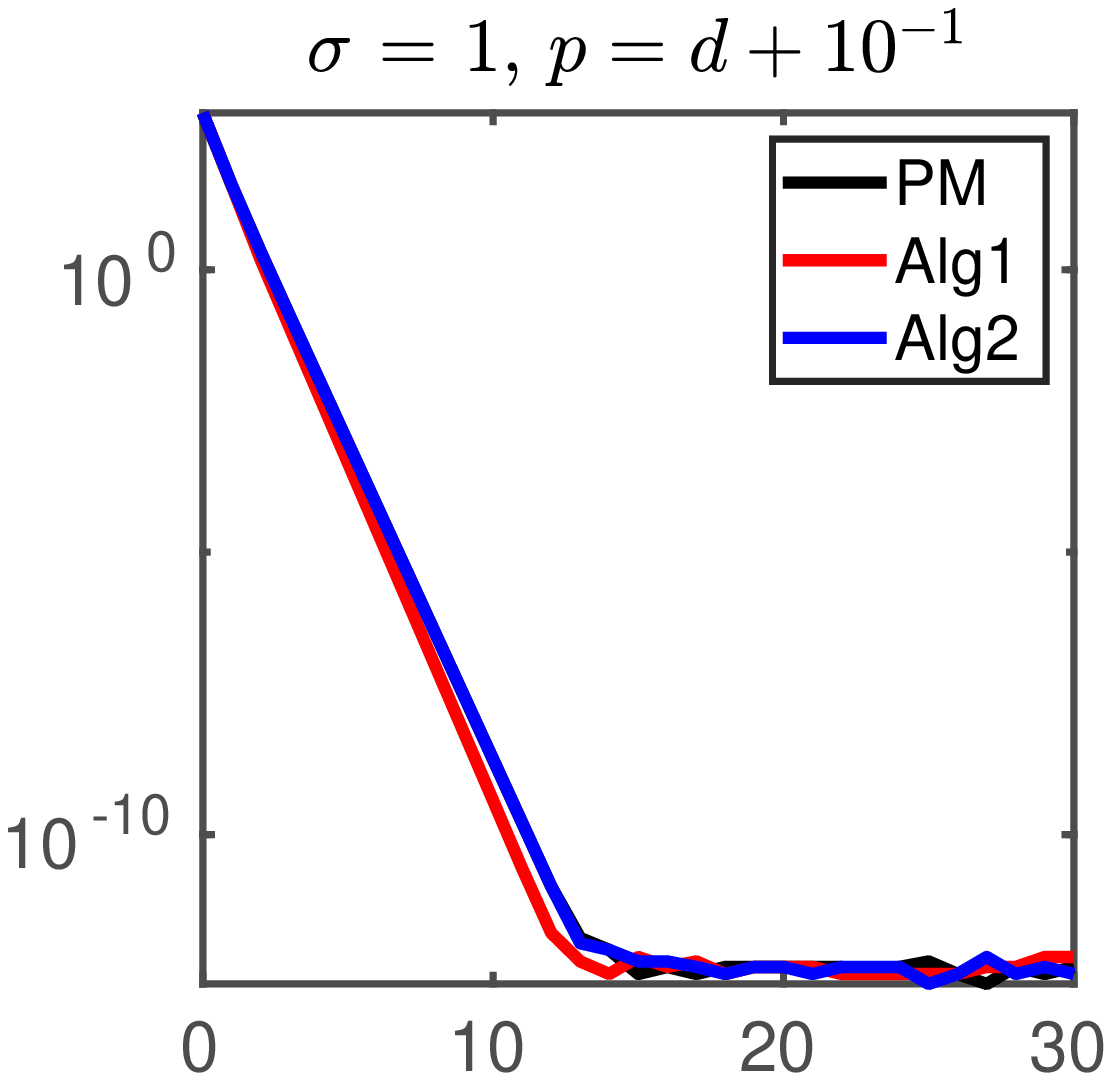}\quad
	\includegraphics[width=.35\textwidth]{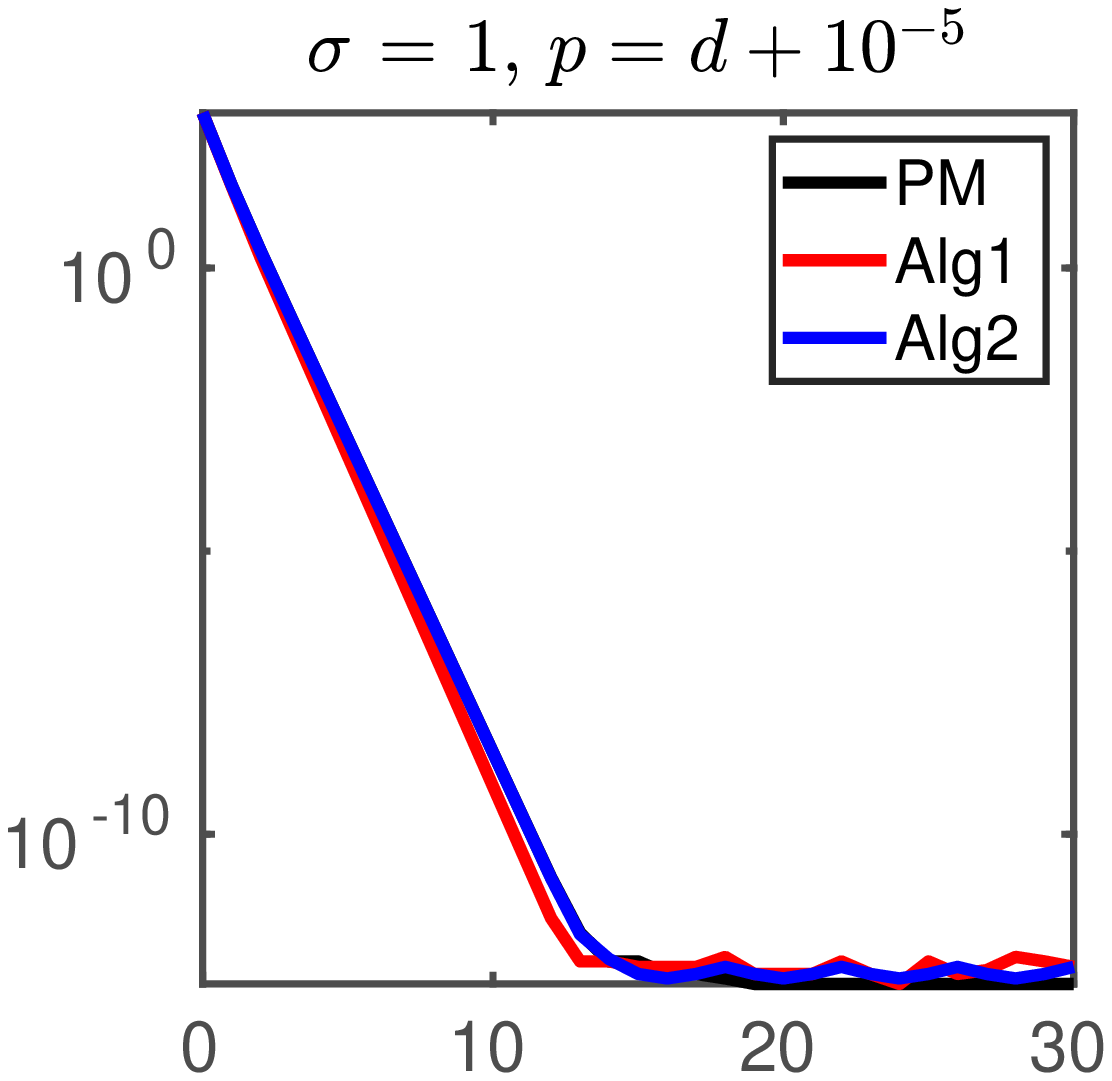}
	\vspace{-0.3cm}
	\caption{Experiments on the test problem $\mathbf B \in \RR^{n\times n\times n}$, $n=100$, considered in \cite{zhang2012linear} and defined in \eqref{eq:tensorsABC}.} \label{fig:tensorsB}
%\end{figure}
% \vspace*{\floatsep}
%\begin{figure}[t]
\end{figure}
On this data set our analysis shows that Algorithm \ref{alg:shiftedPM1} often outperforms the LZI algorithm \cite{liu2010always,zhang2012linear}, when addressing {the problem of approximating} $H$-eigenpairs, i.e. for $p\approx d$. In fact, as figures \ref{fig:tensorsA}, \ref{fig:tensorsB} and \ref{fig:tensorsC} show, the point-wise residual  of Algorithm \ref{alg:shiftedPM1} always converges to zero faster than the one of Algorithm~\ref{alg:shiftedPM2}.
\begin{figure}[p]
	\centering
	\includegraphics[width=.35\textwidth]{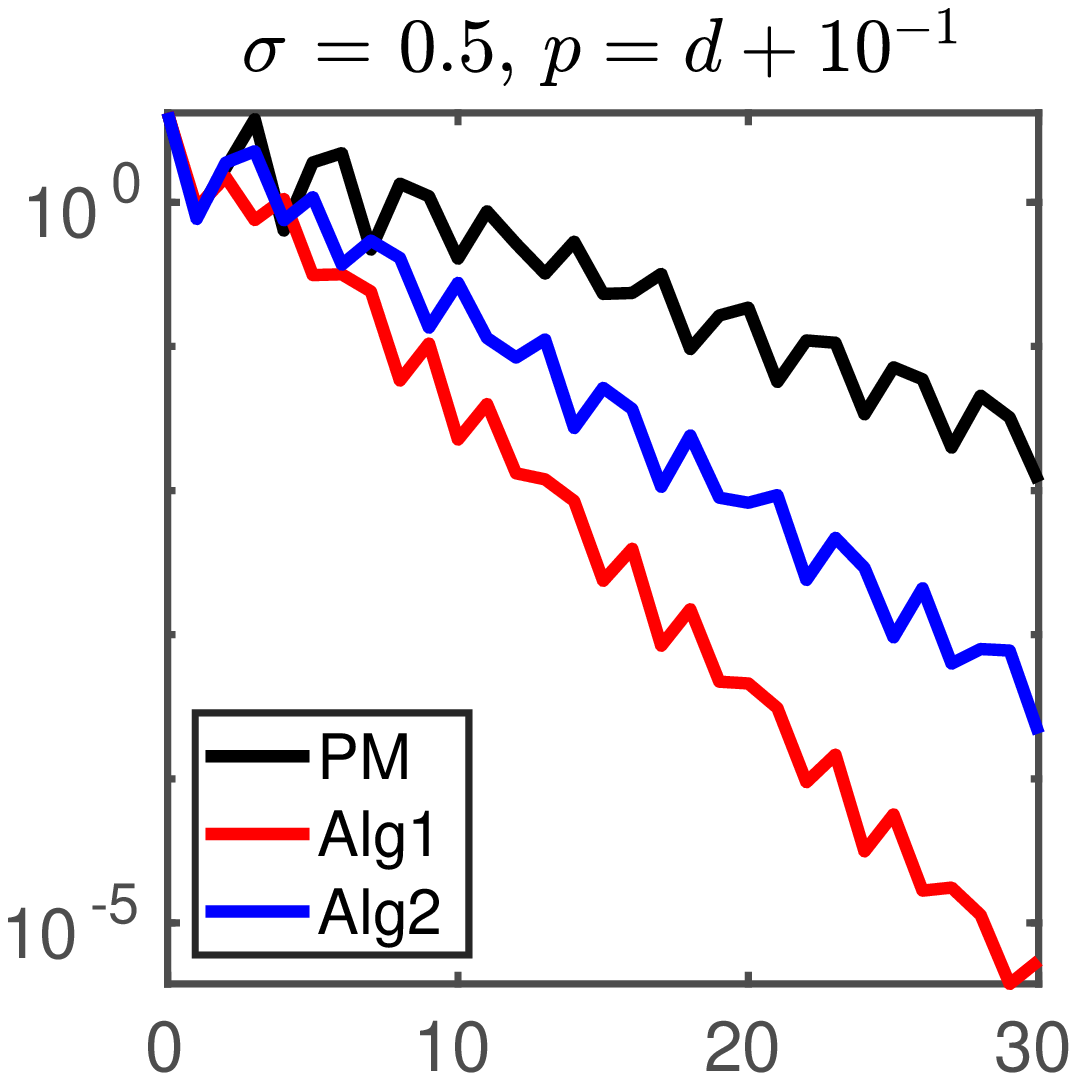}\quad
	\includegraphics[width=.35\textwidth]{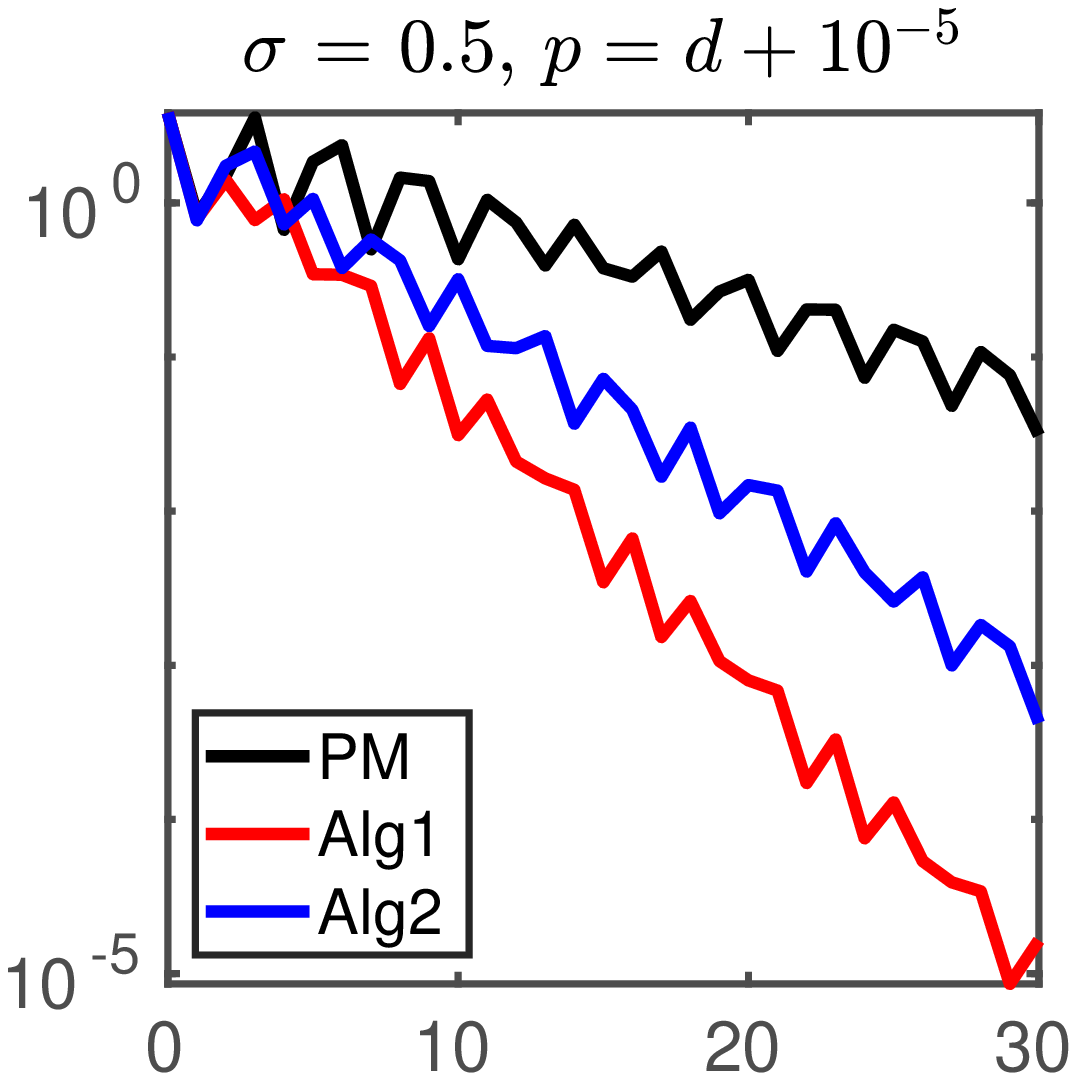}\\[0.3em]
	\includegraphics[width=.35\textwidth]{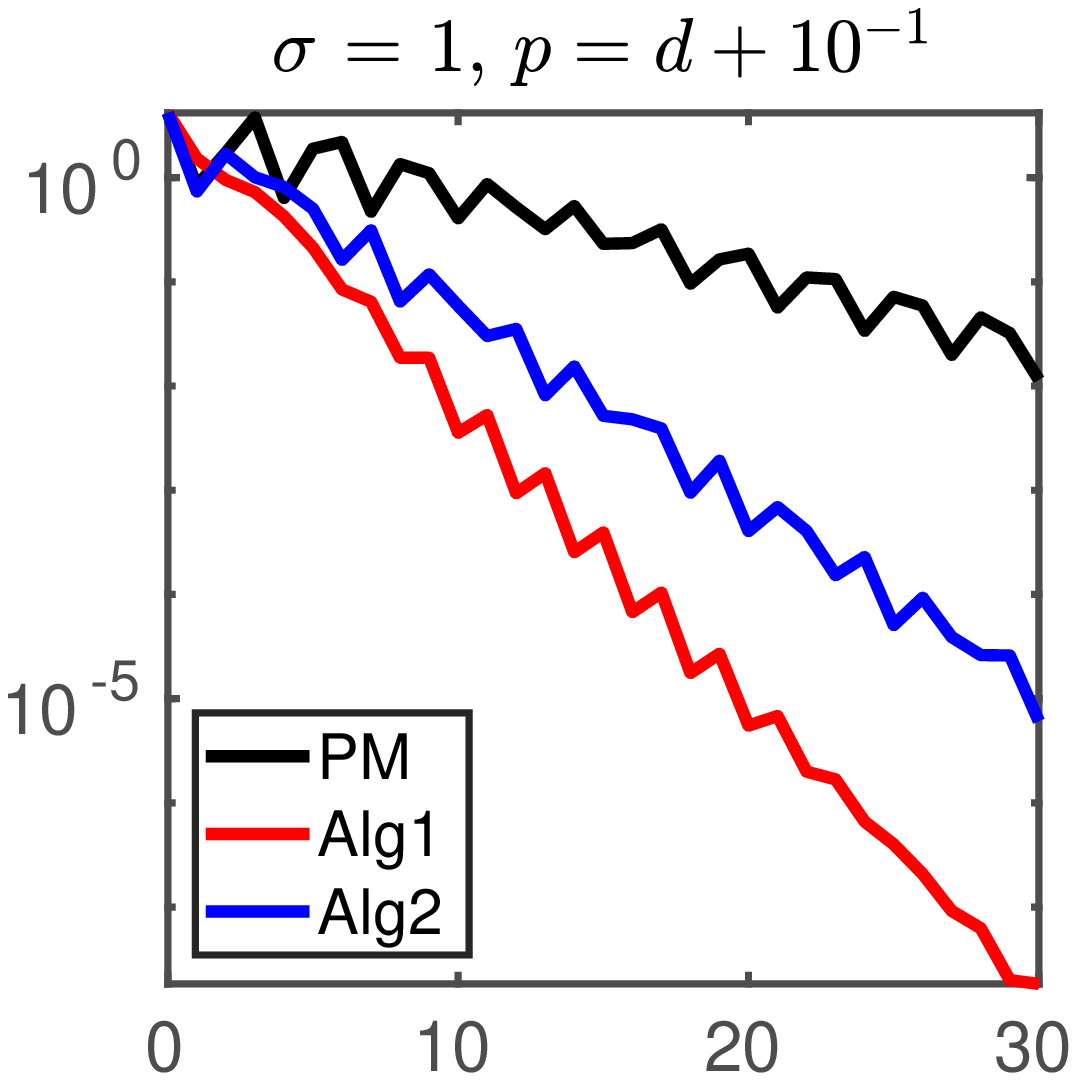}\quad
	\includegraphics[width=.35\textwidth]{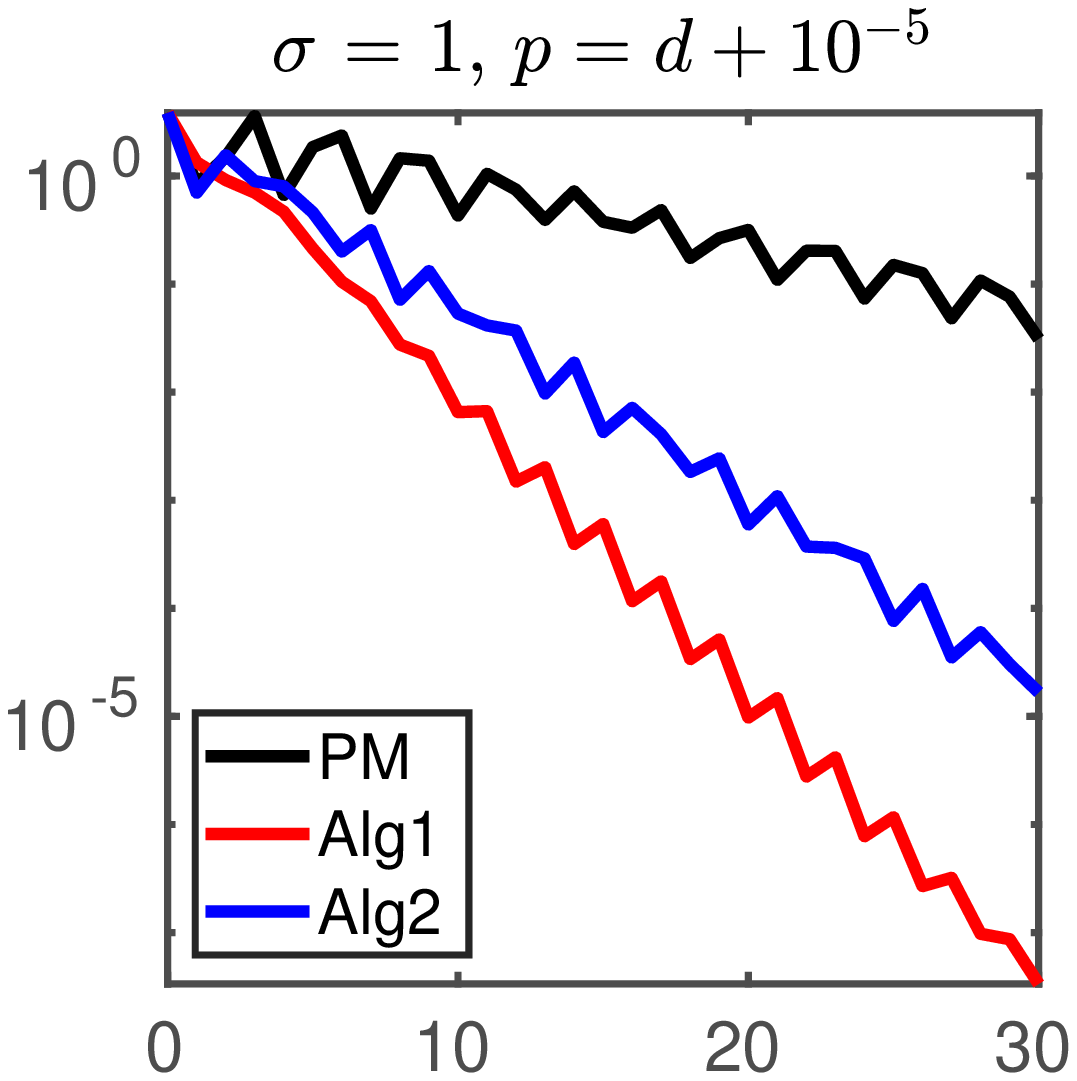}
	\vspace{-0.3cm}
	\caption{Experiments on the test problem $\mathbf C \in \RR^{n\times n\times n}$, $n=100$, considered in \cite{zhang2012linear} and defined in \eqref{eq:tensorsABC}.} \label{fig:tensorsC}

%\begin{figure}[t]
\vspace{0.3cm}
	\centering
	\includegraphics[width=.35\textwidth]{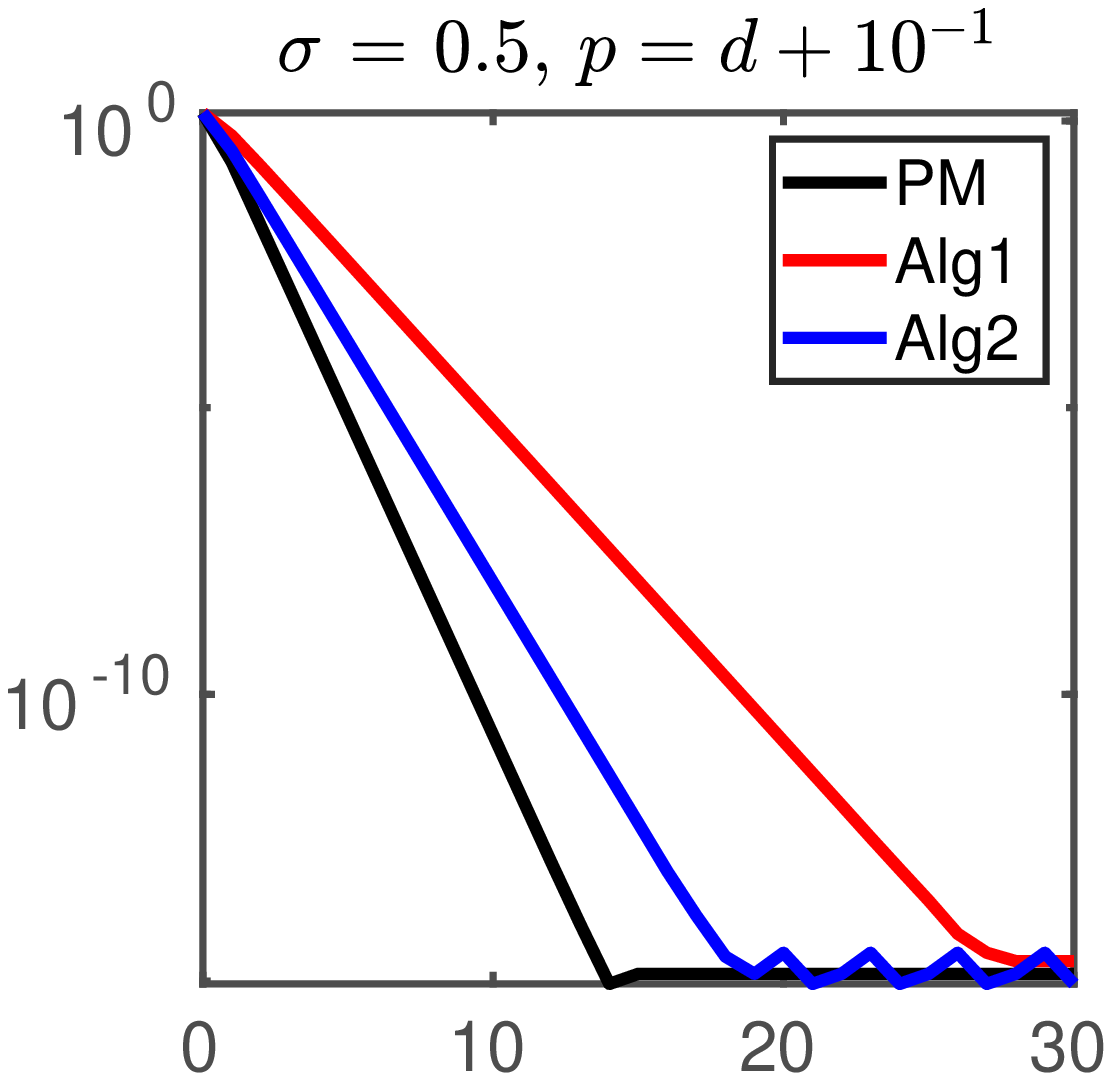}
	\includegraphics[width=.35\textwidth]{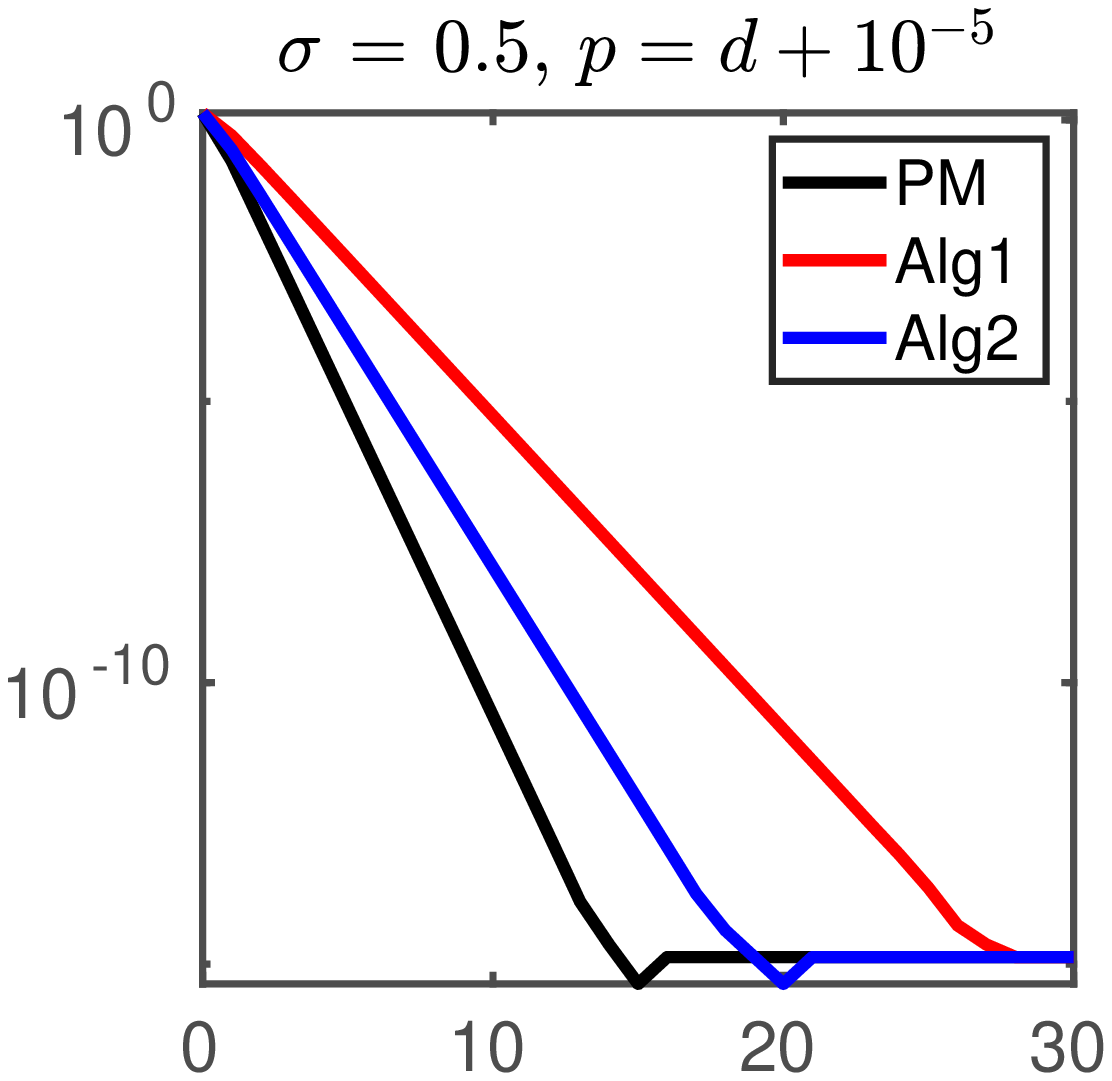}\\
	\includegraphics[width=.35\textwidth]{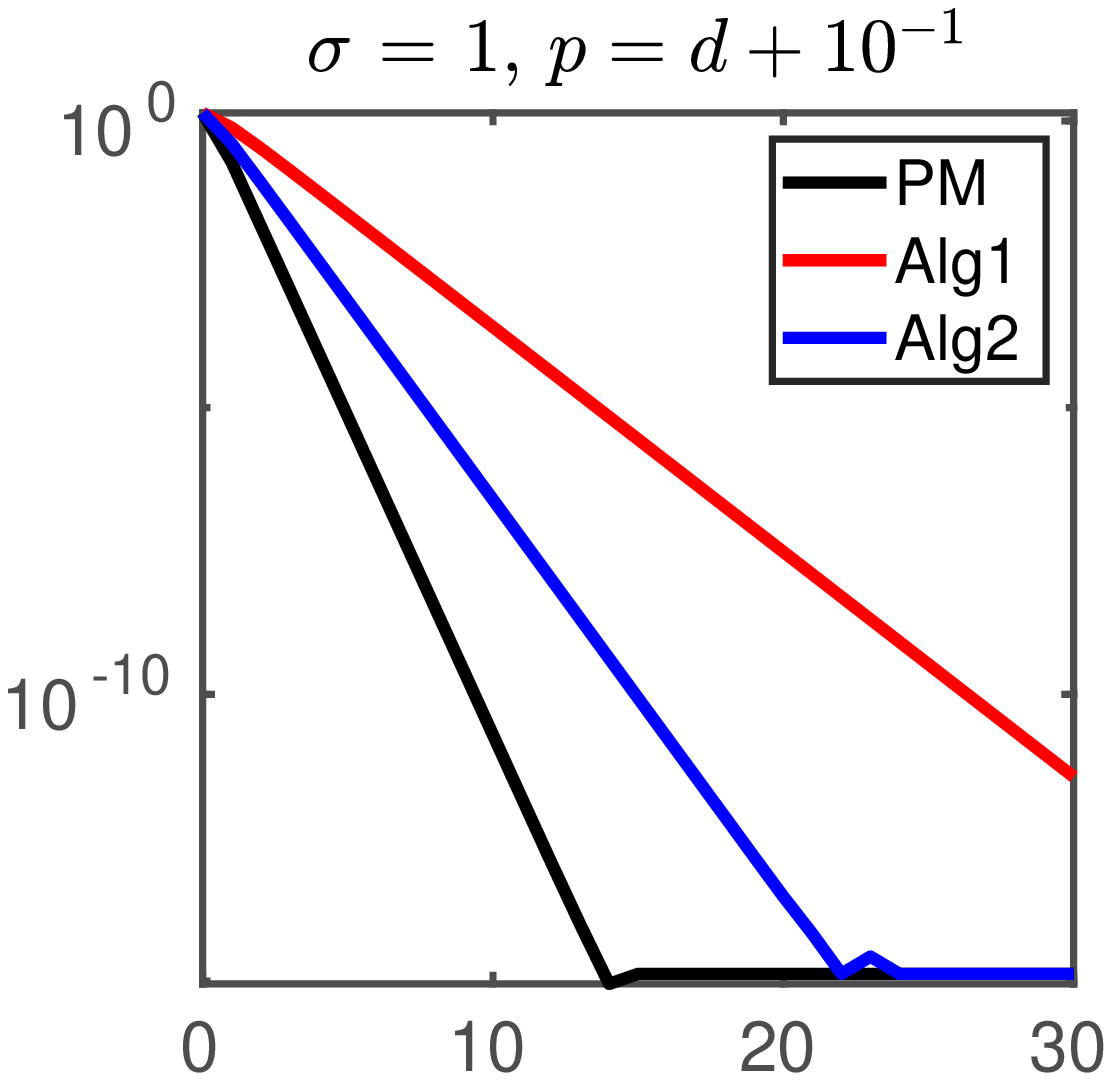}
	\includegraphics[width=.35\textwidth]{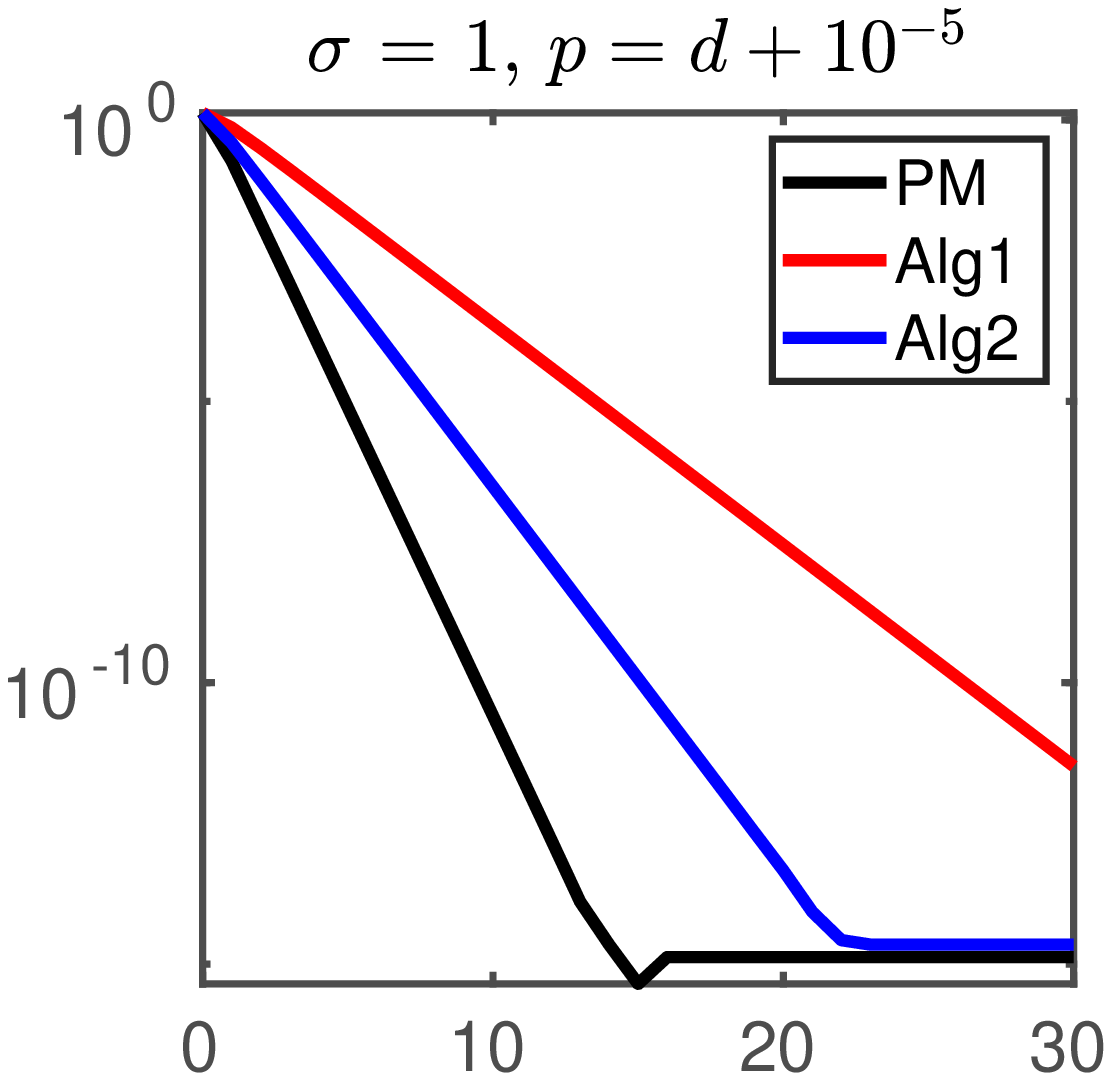}
\vspace{-0.3cm}
	\caption{Experiments on the test problem $\mathbf{K}\in\RR^{n\times n\times n\times n}$. } \label{fig:kR}
\end{figure}

\newcolumntype{C}{ >{\centering\arraybackslash} m{6cm} }
\newcolumntype{D}{ >{\centering\arraybackslash} m{.5cm} }

\subsection{A real symmetric tensor: Kofidis and Regalia example}
One of the first appearances of the power method for tensor eigenvalues was related to $Z$-eigenvalues, that is the ``Euclidean'' case $p=2$.  The symmetric higher-order power method for $Z$-eigenpairs have been introduced in \cite{de2000best}. Afterwards, Kofidis and Regalia \cite{kofidis2002best} note that the convergence of that method is guaranteed only if certain not necessarily mild conditions are met. Moreover, they provide a particularly bad-behaving example tensor $\mathbf K$ where the power iteration fails to converge. We recall that example tensor below  \cite[Ex.\ 1]{kofidis2002best}. Consider the tensor $\tilde{\mathbf K}$ with nonzero entries defined by
$$
\tilde{\mathbf K} = \left\{\begin{array}{llll}
\kappa_{1 1 1 1} = 0.2883  &
\kappa_{ 1 1 1 2 } = -0.0031  &
\kappa_{ 1 1 1 3 } = 0.1973  &
\kappa_{ 1 1 2 2 } = -0.2485 \\
\kappa_{ 1 1 2 3 } = -0.2939  &
\kappa_{ 1 1 3 3 } = 0.3847 &
\kappa_{ 1 2 2 2 } = 0.2972 &
\kappa_{ 1 2 2 3 } = 0.1862 \\
\kappa_{ 1 2 3 3 } = 0.0919 &
\kappa_{ 1 3 3 3 } = -0.3619 &
\kappa_{ 2 2 2 2 } = 0.1241 &
\kappa_{ 2 2 2 3 } = -0.3420 \\
\kappa_{ 2 2 3 3 } = 0.2127 &
\kappa_{ 2 3 3 3 } = 0.2727 &
\kappa_{ 3 3 3 3 } = -0.3054 &
\end{array}\right\} \, .
$$
The tensor $\mathbf K$ is obtained from $\tilde {\mathbf K}$ by symmetrizing it
with respect to any permutation of the indices $i,j,k,\ell$.

Later on, a shifted symmetric higher-order power method (SS-HOPM) for $Z$-eigenpairs is proposed in \cite{kolda2011shifted,regalia2003monotonic}. Clearly, when $p=2$, the proposed Algorithms \ref{alg:shiftedPM1} and \ref{alg:shiftedPM2} both coincide with the SS-HOPM.  In the following Figure \ref{fig:kR} we analyze the behavior of Algorithms \ref{alg:shiftedPM1} and \ref{alg:shiftedPM2} and of power method on the {element-wise absolute value of the Kofidis and Regalia tensor, i.e.,  $\mathbf{T}:=|\mathbf K|=(|\kappa_{ijk\ell}|)\in \RR^{3\times 3 \times 3 \times 3}$.} In particular, plots in Figure \ref{fig:kR} show the residual $\|{T}(\b x_k)-\lambda_k\Phi_p(\b x_k)\|_\infty$ from the {element-wise absolute value of the} same starting point $\b x_0 = (-0.2695, 0.1972, 0.3270)$ proposed in \cite{kofidis2002best,kolda2011shifted}.

\begin{figure}[tb]
	\centering
	\includegraphics[width=.35\textwidth]{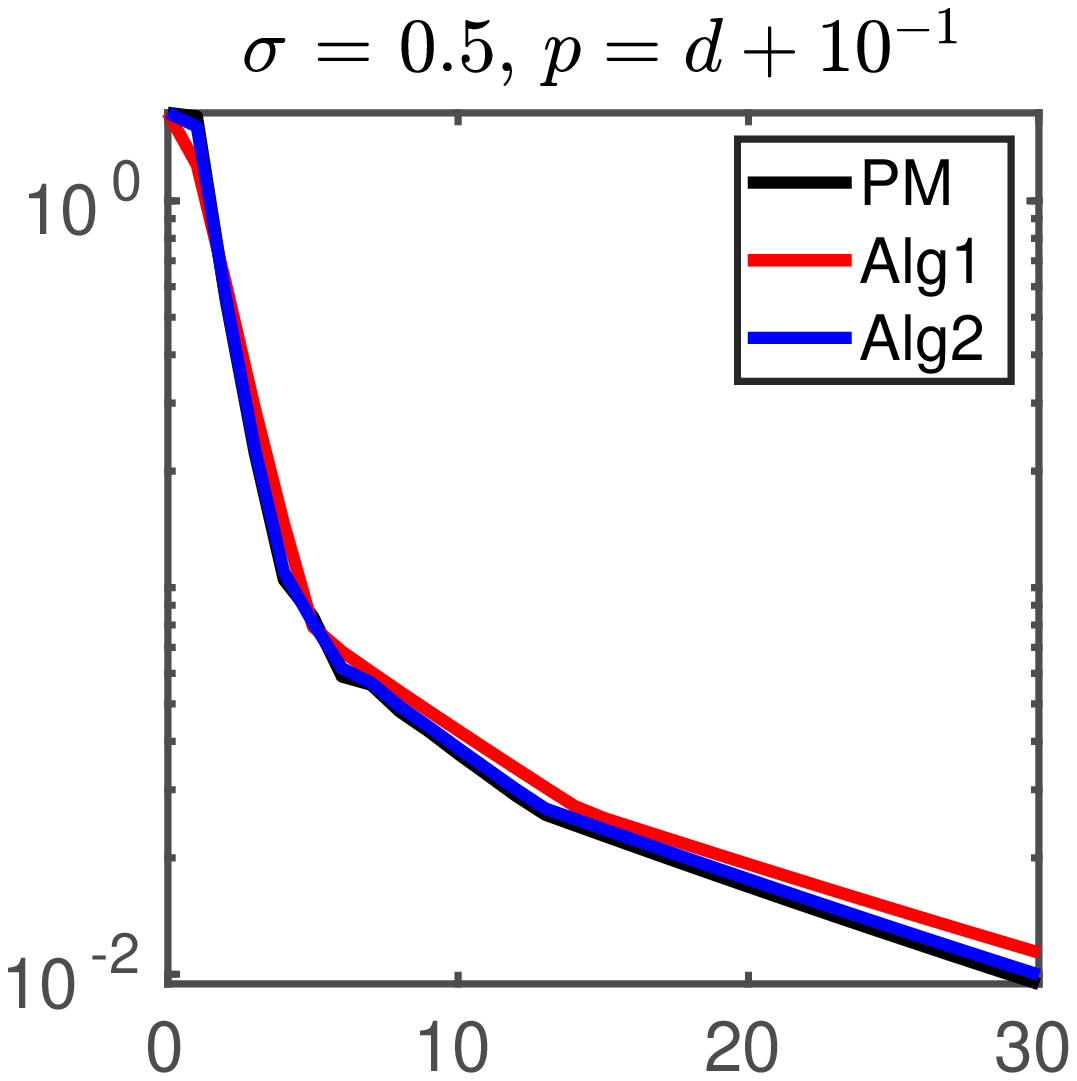}
	\includegraphics[width=.35\textwidth]{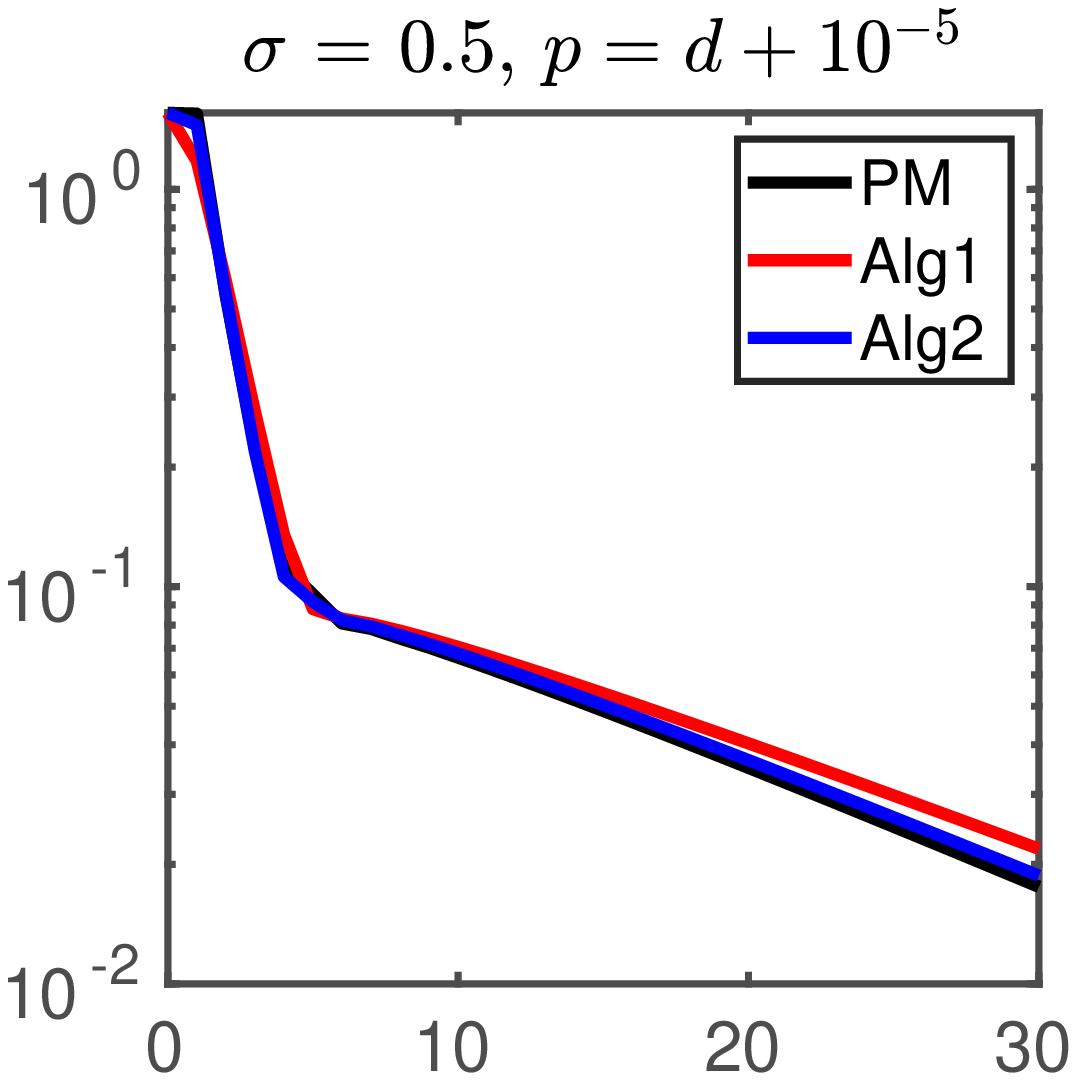}\\
	\includegraphics[width=.35\textwidth]{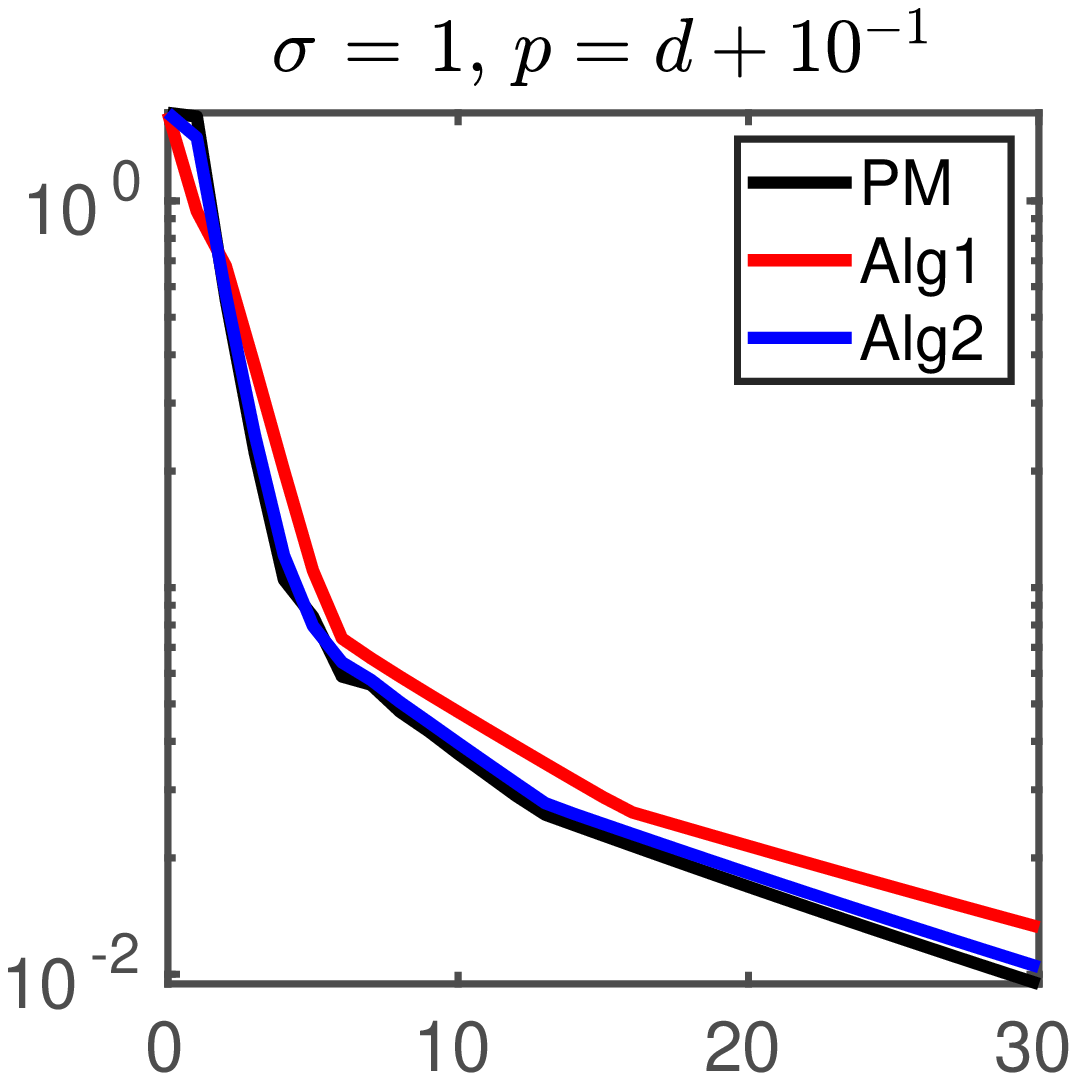}
	\includegraphics[width=.35\textwidth]{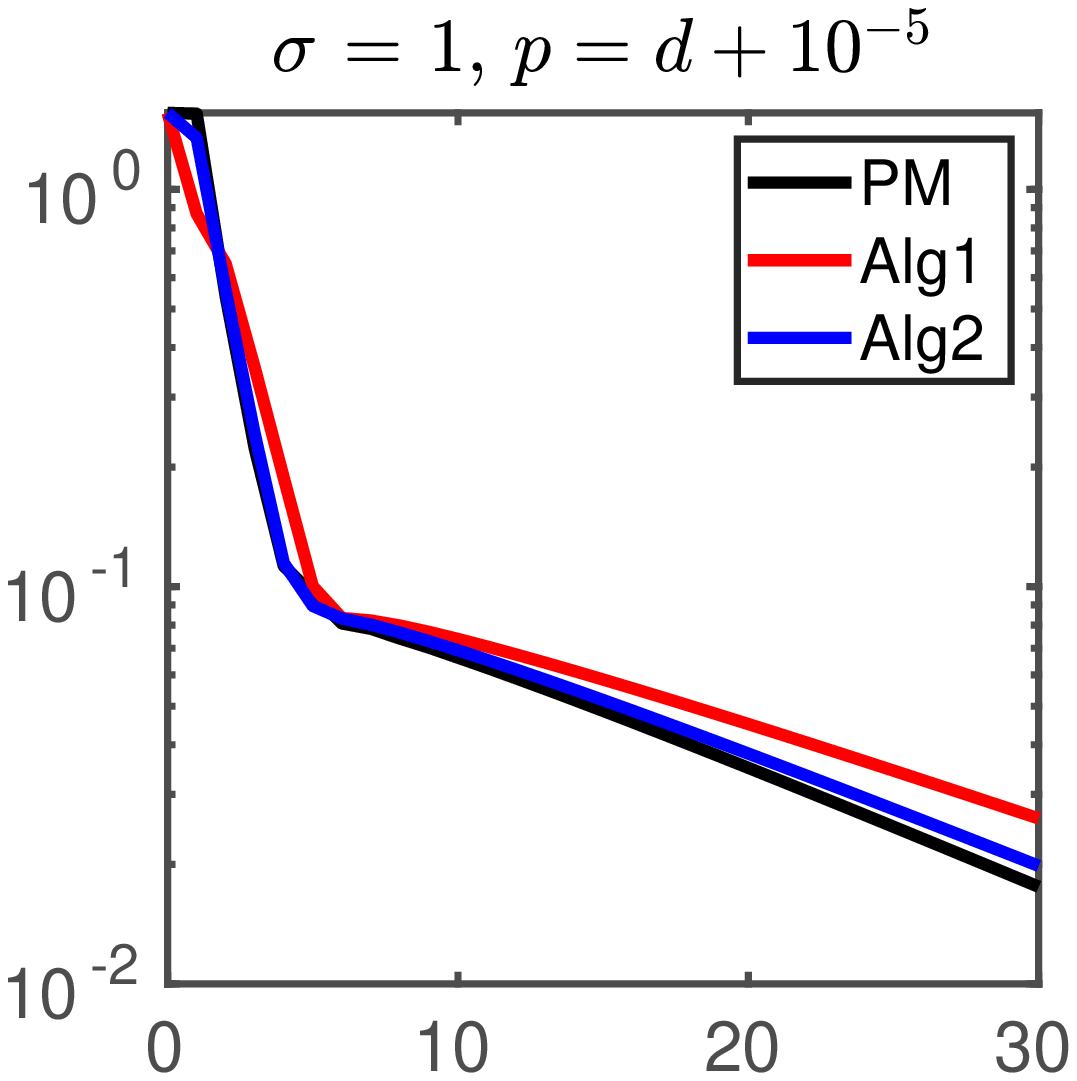}\\
	\caption{Experiments on $\b T\in\RR^{n\times n\times n}, \; n=62$, generated using \texttt{Dolphins}.}\label{fig:tcycle1}
\end{figure}

\subsection{Real world network data tensors} \label{sec:real_world_data_tensors}
Higher-order network analysis is recently gaining an increasing amount of attention since a large number of real-world complex networks shows higher-order features and a higher-organization \cite{benson2016higher}. In this context one often describes higher-order network data via a nonnegative tensor and then implements an analysis based on the eigen or singular vectors in order to compute, for example,  importance scores for network components, see for example \cite{arrigo2019multi,benson2019three,li2012har,arrigo2017centrality}. To test the performances of our methods in this context, here we consider a  network modeled by the graph $G=(V,E)$, with $V =\{1,\dots,n\}$ being the set of nodes and $E$ the set of edges between the nodes. Then, we build the third order three-cycle tensor {$\b T=(t_{ijk})$} with entries
\begin{equation*}
{t_{ijk}}= \begin{cases}
1, \hbox{ if there is a three-cycle between nodes } i,j,k \\
0, \hbox{ otherwise.}
\end{cases}\,
\end{equation*}
where a three-cycle between $i,j,k$ is any closed sequence of edges involving those three nodes. While there is only one possible type of three-cycle in an undirected network, a triangle between $i,j$ and $k$, in the directed case there are up to 7 different types of three-cycle, as we show in the illustration below

\begin{center}
\def\colorenodi{white} %trueblue2!7
\begin{tikzpicture}[
{<[scale=1.5]}-{>[scale=1.5]},
thick,
main node/.style={circle, fill=\colorenodi, draw=black, align=center,inner sep=1}
]
\newcommand*{\MainNum}{3}
\newcommand*{\MainRadius}{.5cm}
\newcommand*{\MainStartAngle}{90}

% Print main nodes, node names: p1, p2, ...
\path
(-.5, 0) coordinate (M)
\foreach \t [count=\i] in {$i$,$j$,$k$}  {
	+({\i-1)*360/\MainNum + \MainStartAngle}:\MainRadius)
	node[main node, align=center] (p\i) {\t}
}
;

% Calculate the angle between the equal sides of the triangle
% with side length \MainRadius, \MainRadius and radius of circle node
% Result is stored in \p1-angle, \p2-angle, ...
\foreach \i in {1, ..., \MainNum} {
	\pgfextracty{\dimen0 }{\pgfpointanchor{p\i}{north}}
	\pgfextracty{\dimen2 }{\pgfpointanchor{p\i}{center}}
	\dimen0=\dimexpr\dimen2 - \dimen0\relax
	\ifdim\dimen0<0pt \dimen0 = -\dimen0 \fi
	\pgfmathparse{2*asin(\the\dimen0/\MainRadius/2)}
	\global\expandafter\let\csname p\i-angle\endcsname\pgfmathresult
}

% Draw the arrow arcs
\foreach \i [evaluate=\i as \nexti using {int(mod(\i, \MainNum)+1}] in {1, ..., \MainNum} {
	\pgfmathsetmacro\StartAngle{ (\i-1)*360/\MainNum + \MainStartAngle + \csname p\i-angle\endcsname }
	\pgfmathsetmacro\EndAngle{ (\nexti-1)*360/\MainNum + \MainStartAngle - \csname p\nexti-angle\endcsname }
	\ifdim\EndAngle pt < \StartAngle pt
	\pgfmathsetmacro\EndAngle{\EndAngle + 360}
	\fi
	\draw[<->]  (M) ++(\StartAngle:\MainRadius) arc[start angle=\StartAngle, end angle=\EndAngle, radius=\MainRadius];
}
\end{tikzpicture}
\hfill
\begin{tikzpicture}[
{<[scale=1.5]}-{>[scale=1.5]},
thick,
main node/.style={circle, fill=\colorenodi, draw=black, align=center,inner sep=1}
]
\newcommand*{\MainNum}{3}
\newcommand*{\MainRadius}{.5cm}
\newcommand*{\MainStartAngle}{90}

% Print main nodes, node names: p1, p2, ...
\path
(-.5, 0) coordinate (M)
\foreach \t [count=\i] in {$i$,$j$,$k$}  {
	+({\i-1)*360/\MainNum + \MainStartAngle}:\MainRadius)
	node[main node, align=center] (p\i) {\t}
}
;

% Calculate the angle between the equal sides of the triangle
% with side length \MainRadius, \MainRadius and radius of circle node
% Result is stored in \p1-angle, \p2-angle, ...
\foreach \i in {1, ..., \MainNum} {
	\pgfextracty{\dimen0 }{\pgfpointanchor{p\i}{north}}
	\pgfextracty{\dimen2 }{\pgfpointanchor{p\i}{center}}
	\dimen0=\dimexpr\dimen2 - \dimen0\relax
	\ifdim\dimen0<0pt \dimen0 = -\dimen0 \fi
	\pgfmathparse{2*asin(\the\dimen0/\MainRadius/2)}
	\global\expandafter\let\csname p\i-angle\endcsname\pgfmathresult
}

% Draw the arrow arcs
\foreach \i [evaluate=\i as \nexti using {int(mod(\i, \MainNum)+1}] in {1, ..., \MainNum} {
	\pgfmathsetmacro\StartAngle{ (\i-1)*360/\MainNum + \MainStartAngle + \csname p\i-angle\endcsname }
	\pgfmathsetmacro\EndAngle{ (\nexti-1)*360/\MainNum + \MainStartAngle - \csname p\nexti-angle\endcsname }
	\ifdim\EndAngle pt < \StartAngle pt
	\pgfmathsetmacro\EndAngle{\EndAngle + 360}
	\fi
	\draw[->]  (M) ++(\StartAngle:\MainRadius) arc[start angle=\StartAngle, end angle=\EndAngle, radius=\MainRadius];
}
\end{tikzpicture}
\hfill
\begin{tikzpicture}[
{<[scale=1.5]}-{>[scale=1.5]},
thick,
main node/.style={circle, fill=\colorenodi, draw=black, align=center,inner sep=1}
]
\newcommand*{\MainNum}{3}
\newcommand*{\MainRadius}{.5cm}
\newcommand*{\MainStartAngle}{90}

% Print main nodes, node names: p1, p2, ...
\path
(-.5, 0) coordinate (M)
\foreach \t [count=\i] in {$i$,$j$,$k$}  {
	+({\i-1)*360/\MainNum + \MainStartAngle}:\MainRadius)
	node[main node, align=center] (p\i) {\t}
}
;

% Calculate the angle between the equal sides of the triangle
% with side length \MainRadius, \MainRadius and radius of circle node
% Result is stored in \p1-angle, \p2-angle, ...
\foreach \i in {1, ..., \MainNum} {
	\pgfextracty{\dimen0 }{\pgfpointanchor{p\i}{north}}
	\pgfextracty{\dimen2 }{\pgfpointanchor{p\i}{center}}
	\dimen0=\dimexpr\dimen2 - \dimen0\relax
	\ifdim\dimen0<0pt \dimen0 = -\dimen0 \fi
	\pgfmathparse{2*asin(\the\dimen0/\MainRadius/2)}
	\global\expandafter\let\csname p\i-angle\endcsname\pgfmathresult
}

% Draw the arrow arcs
\pgfmathsetmacro\StartAngle{ 0 + \MainStartAngle + \csname p1-angle\endcsname }
\pgfmathsetmacro\EndAngle{  360/\MainNum + \MainStartAngle - \csname p2-angle\endcsname }
\draw[->]  (M) ++(\StartAngle:\MainRadius) arc[start angle=\StartAngle, end angle=\EndAngle, radius=\MainRadius];

\pgfmathsetmacro\StartAngle{ 360/\MainNum + \MainStartAngle + \csname p2-angle\endcsname }
\pgfmathsetmacro\EndAngle{  2*360/\MainNum + \MainStartAngle - \csname p3-angle\endcsname }
\draw[<-]  (M) ++(\StartAngle:\MainRadius) arc[start angle=\StartAngle, end angle=\EndAngle, radius=\MainRadius];

\pgfmathsetmacro\StartAngle{ 2*360/\MainNum + \MainStartAngle + \csname p3-angle\endcsname }
\pgfmathsetmacro\EndAngle{  360 + \MainStartAngle - \csname p1-angle\endcsname }
\draw[<-]  (M) ++(\StartAngle:\MainRadius) arc[start angle=\StartAngle, end angle=\EndAngle, radius=\MainRadius];
\end{tikzpicture}
\hfill
\begin{tikzpicture}[
{<[scale=1.5]}-{>[scale=1.5]},
thick,
main node/.style={circle, fill=\colorenodi, draw=black, align=center,inner sep=1}
]
\newcommand*{\MainNum}{3}
\newcommand*{\MainRadius}{.5cm}
\newcommand*{\MainStartAngle}{90}

% Print main nodes, node names: p1, p2, ...
\path
(-.5, 0) coordinate (M)
\foreach \t [count=\i] in {$i$,$j$,$k$}  {
	+({\i-1)*360/\MainNum + \MainStartAngle}:\MainRadius)
	node[main node, align=center] (p\i) {\t}
}
;

% Calculate the angle between the equal sides of the triangle
% with side length \MainRadius, \MainRadius and radius of circle node
% Result is stored in \p1-angle, \p2-angle, ...
\foreach \i in {1, ..., \MainNum} {
	\pgfextracty{\dimen0 }{\pgfpointanchor{p\i}{north}}
	\pgfextracty{\dimen2 }{\pgfpointanchor{p\i}{center}}
	\dimen0=\dimexpr\dimen2 - \dimen0\relax
	\ifdim\dimen0<0pt \dimen0 = -\dimen0 \fi
	\pgfmathparse{2*asin(\the\dimen0/\MainRadius/2)}
	\global\expandafter\let\csname p\i-angle\endcsname\pgfmathresult
}

% Draw the arrow arcs
\pgfmathsetmacro\StartAngle{ 0 + \MainStartAngle + \csname p1-angle\endcsname }
\pgfmathsetmacro\EndAngle{  360/\MainNum + \MainStartAngle - \csname p2-angle\endcsname }
\draw[->]  (M) ++(\StartAngle:\MainRadius) arc[start angle=\StartAngle, end angle=\EndAngle, radius=\MainRadius];

\pgfmathsetmacro\StartAngle{ 360/\MainNum + \MainStartAngle + \csname p2-angle\endcsname }
\pgfmathsetmacro\EndAngle{  2*360/\MainNum + \MainStartAngle - \csname p3-angle\endcsname }
\draw[<->]  (M) ++(\StartAngle:\MainRadius) arc[start angle=\StartAngle, end angle=\EndAngle, radius=\MainRadius];

\pgfmathsetmacro\StartAngle{ 2*360/\MainNum + \MainStartAngle + \csname p3-angle\endcsname }
\pgfmathsetmacro\EndAngle{  360 + \MainStartAngle - \csname p1-angle\endcsname }
\draw[<-]  (M) ++(\StartAngle:\MainRadius) arc[start angle=\StartAngle, end angle=\EndAngle, radius=\MainRadius];
\end{tikzpicture}
\hfill
\begin{tikzpicture}[
{<[scale=1.5]}-{>[scale=1.5]},
thick,
main node/.style={circle, fill=\colorenodi, draw=black, align=center,inner sep=1}
]
\newcommand*{\MainNum}{3}
\newcommand*{\MainRadius}{.5cm}
\newcommand*{\MainStartAngle}{90}

% Print main nodes, node names: p1, p2, ...
\path
(-.5, 0) coordinate (M)
\foreach \t [count=\i] in {$i$,$j$,$k$}  {
	+({\i-1)*360/\MainNum + \MainStartAngle}:\MainRadius)
	node[main node, align=center] (p\i) {\t}
}
;

% Calculate the angle between the equal sides of the triangle
% with side length \MainRadius, \MainRadius and radius of circle node
% Result is stored in \p1-angle, \p2-angle, ...
\foreach \i in {1, ..., \MainNum} {
	\pgfextracty{\dimen0 }{\pgfpointanchor{p\i}{north}}
	\pgfextracty{\dimen2 }{\pgfpointanchor{p\i}{center}}
	\dimen0=\dimexpr\dimen2 - \dimen0\relax
	\ifdim\dimen0<0pt \dimen0 = -\dimen0 \fi
	\pgfmathparse{2*asin(\the\dimen0/\MainRadius/2)}
	\global\expandafter\let\csname p\i-angle\endcsname\pgfmathresult
}

% Draw the arrow arcs
\pgfmathsetmacro\StartAngle{ 0 + \MainStartAngle + \csname p1-angle\endcsname }
\pgfmathsetmacro\EndAngle{  360/\MainNum + \MainStartAngle - \csname p2-angle\endcsname }
\draw[<-]  (M) ++(\StartAngle:\MainRadius) arc[start angle=\StartAngle, end angle=\EndAngle, radius=\MainRadius];

\pgfmathsetmacro\StartAngle{ 360/\MainNum + \MainStartAngle + \csname p2-angle\endcsname }
\pgfmathsetmacro\EndAngle{  2*360/\MainNum + \MainStartAngle - \csname p3-angle\endcsname }
\draw[<->]  (M) ++(\StartAngle:\MainRadius) arc[start angle=\StartAngle, end angle=\EndAngle, radius=\MainRadius];

\pgfmathsetmacro\StartAngle{ 2*360/\MainNum + \MainStartAngle + \csname p3-angle\endcsname }
\pgfmathsetmacro\EndAngle{  360 + \MainStartAngle - \csname p1-angle\endcsname }
\draw[->]  (M) ++(\StartAngle:\MainRadius) arc[start angle=\StartAngle, end angle=\EndAngle, radius=\MainRadius];
\end{tikzpicture}
\hfill
\begin{tikzpicture}[
{<[scale=1.5]}-{>[scale=1.5]},
thick,
main node/.style={circle, fill=\colorenodi, draw=black, align=center,inner sep=1}
]
\newcommand*{\MainNum}{3}
\newcommand*{\MainRadius}{.5cm}
\newcommand*{\MainStartAngle}{90}

% Print main nodes, node names: p1, p2, ...
\path
(-.5, 0) coordinate (M)
\foreach \t [count=\i] in {$i$,$j$,$k$}  {
	+({\i-1)*360/\MainNum + \MainStartAngle}:\MainRadius)
	node[main node, align=center] (p\i) {\t}
}
;

% Calculate the angle between the equal sides of the triangle
% with side length \MainRadius, \MainRadius and radius of circle node
% Result is stored in \p1-angle, \p2-angle, ...
\foreach \i in {1, ..., \MainNum} {
	\pgfextracty{\dimen0 }{\pgfpointanchor{p\i}{north}}
	\pgfextracty{\dimen2 }{\pgfpointanchor{p\i}{center}}
	\dimen0=\dimexpr\dimen2 - \dimen0\relax
	\ifdim\dimen0<0pt \dimen0 = -\dimen0 \fi
	\pgfmathparse{2*asin(\the\dimen0/\MainRadius/2)}
	\global\expandafter\let\csname p\i-angle\endcsname\pgfmathresult
}

% Draw the arrow arcs
\pgfmathsetmacro\StartAngle{ 0 + \MainStartAngle + \csname p1-angle\endcsname }
\pgfmathsetmacro\EndAngle{  360/\MainNum + \MainStartAngle - \csname p2-angle\endcsname }
\draw[->]  (M) ++(\StartAngle:\MainRadius) arc[start angle=\StartAngle, end angle=\EndAngle, radius=\MainRadius];

\pgfmathsetmacro\StartAngle{ 360/\MainNum + \MainStartAngle + \csname p2-angle\endcsname }
\pgfmathsetmacro\EndAngle{  2*360/\MainNum + \MainStartAngle - \csname p3-angle\endcsname }
\draw[<->]  (M) ++(\StartAngle:\MainRadius) arc[start angle=\StartAngle, end angle=\EndAngle, radius=\MainRadius];

\pgfmathsetmacro\StartAngle{ 2*360/\MainNum + \MainStartAngle + \csname p3-angle\endcsname }
\pgfmathsetmacro\EndAngle{  360 + \MainStartAngle - \csname p1-angle\endcsname }
\draw[->]  (M) ++(\StartAngle:\MainRadius) arc[start angle=\StartAngle, end angle=\EndAngle, radius=\MainRadius];
\end{tikzpicture}
\hfill
\begin{tikzpicture}[
{<[scale=1.5]}-{>[scale=1.5]},
thick,
main node/.style={circle, fill=\colorenodi, draw=black, align=center,inner sep=1}
]
\newcommand*{\MainNum}{3}
\newcommand*{\MainRadius}{.5cm}
\newcommand*{\MainStartAngle}{90}

% Print main nodes, node names: p1, p2, ...
\path
(-.5, 0) coordinate (M)
\foreach \t [count=\i] in {$i$,$j$,$k$}  {
	+({\i-1)*360/\MainNum + \MainStartAngle}:\MainRadius)
	node[main node, align=center] (p\i) {\t}
}
;

% Calculate the angle between the equal sides of the triangle
% with side length \MainRadius, \MainRadius and radius of circle node
% Result is stored in \p1-angle, \p2-angle, ...
\foreach \i in {1, ..., \MainNum} {
	\pgfextracty{\dimen0 }{\pgfpointanchor{p\i}{north}}
	\pgfextracty{\dimen2 }{\pgfpointanchor{p\i}{center}}
	\dimen0=\dimexpr\dimen2 - \dimen0\relax
	\ifdim\dimen0<0pt \dimen0 = -\dimen0 \fi
	\pgfmathparse{2*asin(\the\dimen0/\MainRadius/2)}
	\global\expandafter\let\csname p\i-angle\endcsname\pgfmathresult
}

% Draw the arrow arcs
\pgfmathsetmacro\StartAngle{ 0 + \MainStartAngle + \csname p1-angle\endcsname }
\pgfmathsetmacro\EndAngle{  360/\MainNum + \MainStartAngle - \csname p2-angle\endcsname }
\draw[->]  (M) ++(\StartAngle:\MainRadius) arc[start angle=\StartAngle, end angle=\EndAngle, radius=\MainRadius];

\pgfmathsetmacro\StartAngle{ 360/\MainNum + \MainStartAngle + \csname p2-angle\endcsname }
\pgfmathsetmacro\EndAngle{  2*360/\MainNum + \MainStartAngle - \csname p3-angle\endcsname }
\draw[<->]  (M) ++(\StartAngle:\MainRadius) arc[start angle=\StartAngle, end angle=\EndAngle, radius=\MainRadius];

\pgfmathsetmacro\StartAngle{ 2*360/\MainNum + \MainStartAngle + \csname p3-angle\endcsname }
\pgfmathsetmacro\EndAngle{  360 + \MainStartAngle - \csname p1-angle\endcsname }
\draw[<->]  (M) ++(\StartAngle:\MainRadius) arc[start angle=\StartAngle, end angle=\EndAngle, radius=\MainRadius];
\end{tikzpicture}
\end{center}

\begin{figure}[p]
	\centering
	\includegraphics[width=.35\textwidth]{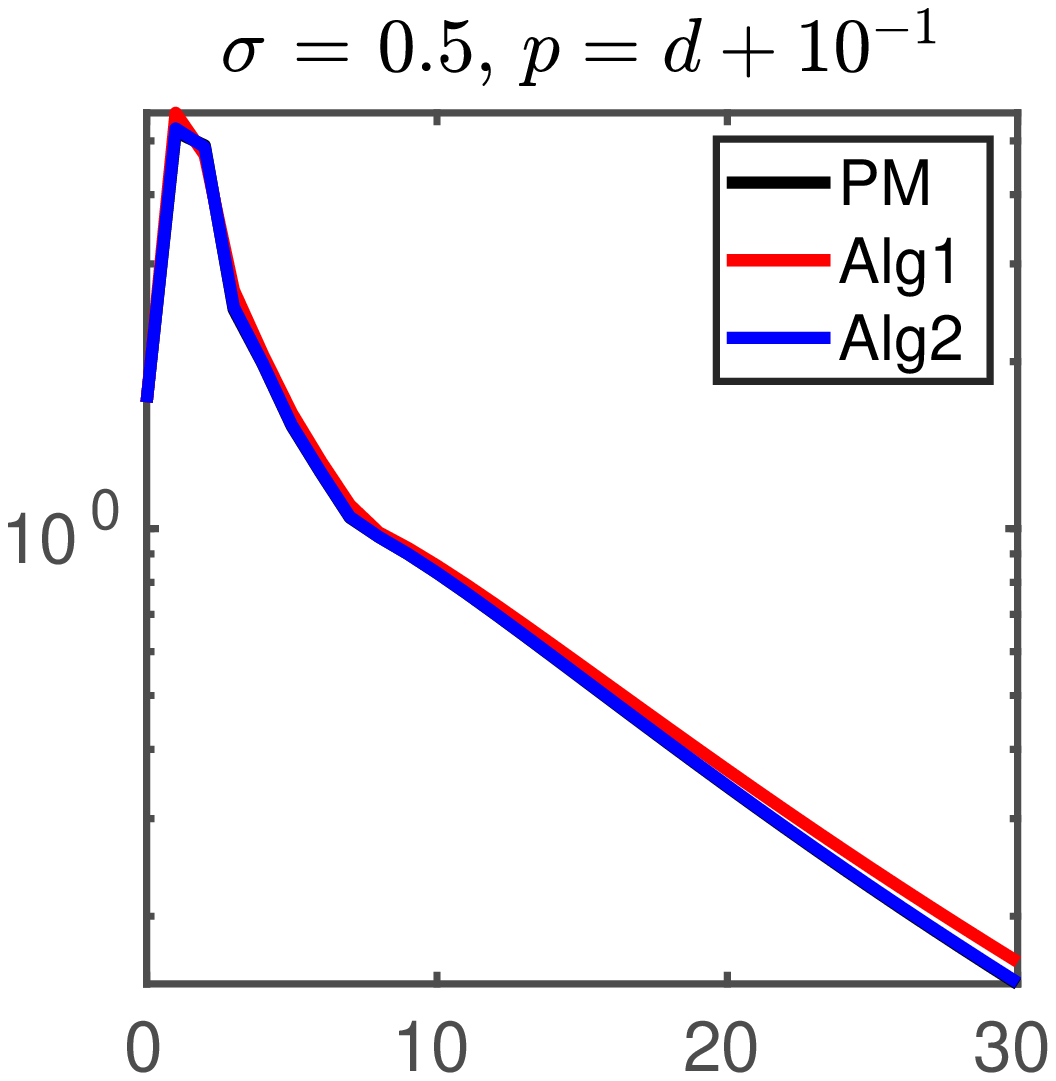}
	\includegraphics[width=.35\textwidth]{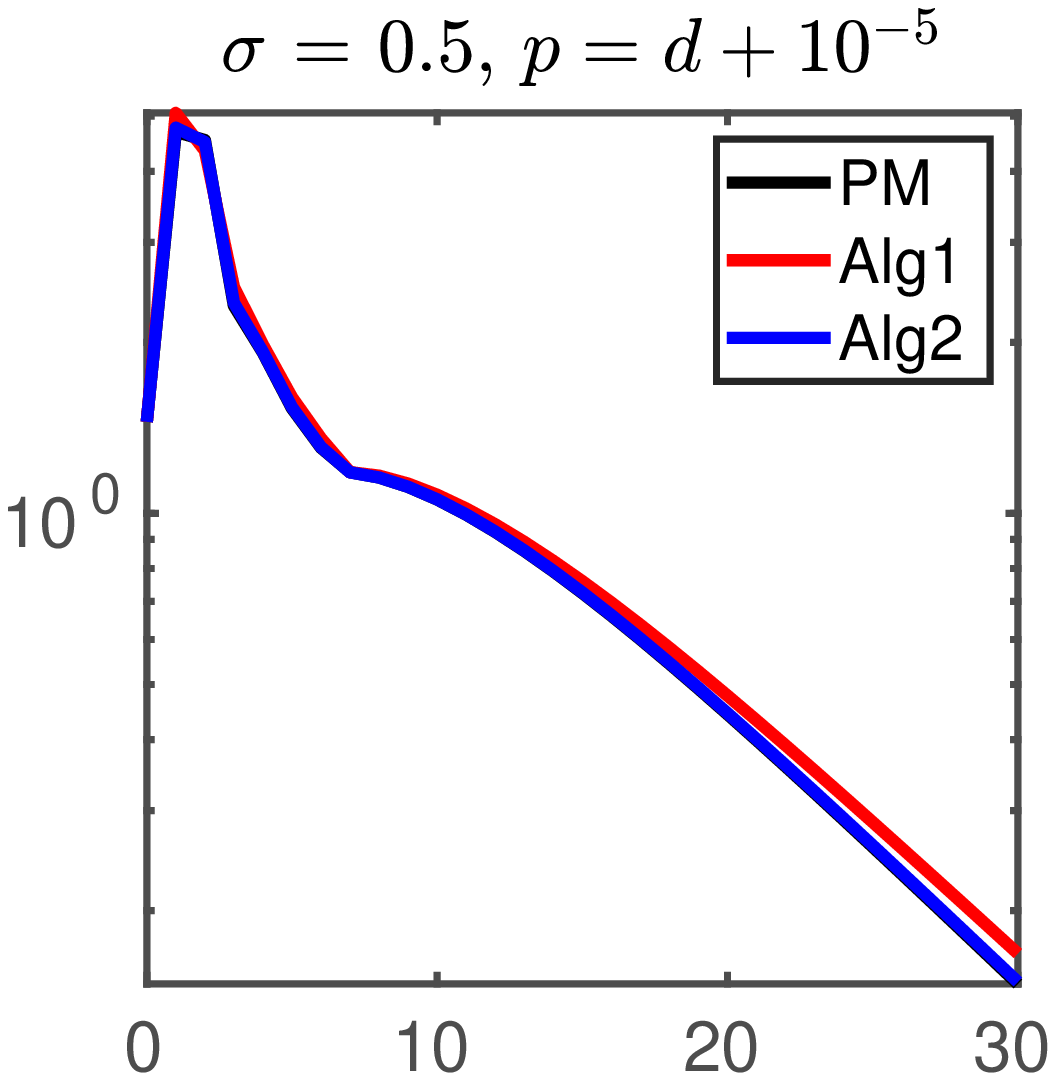}\\
	\includegraphics[width=.35\textwidth]{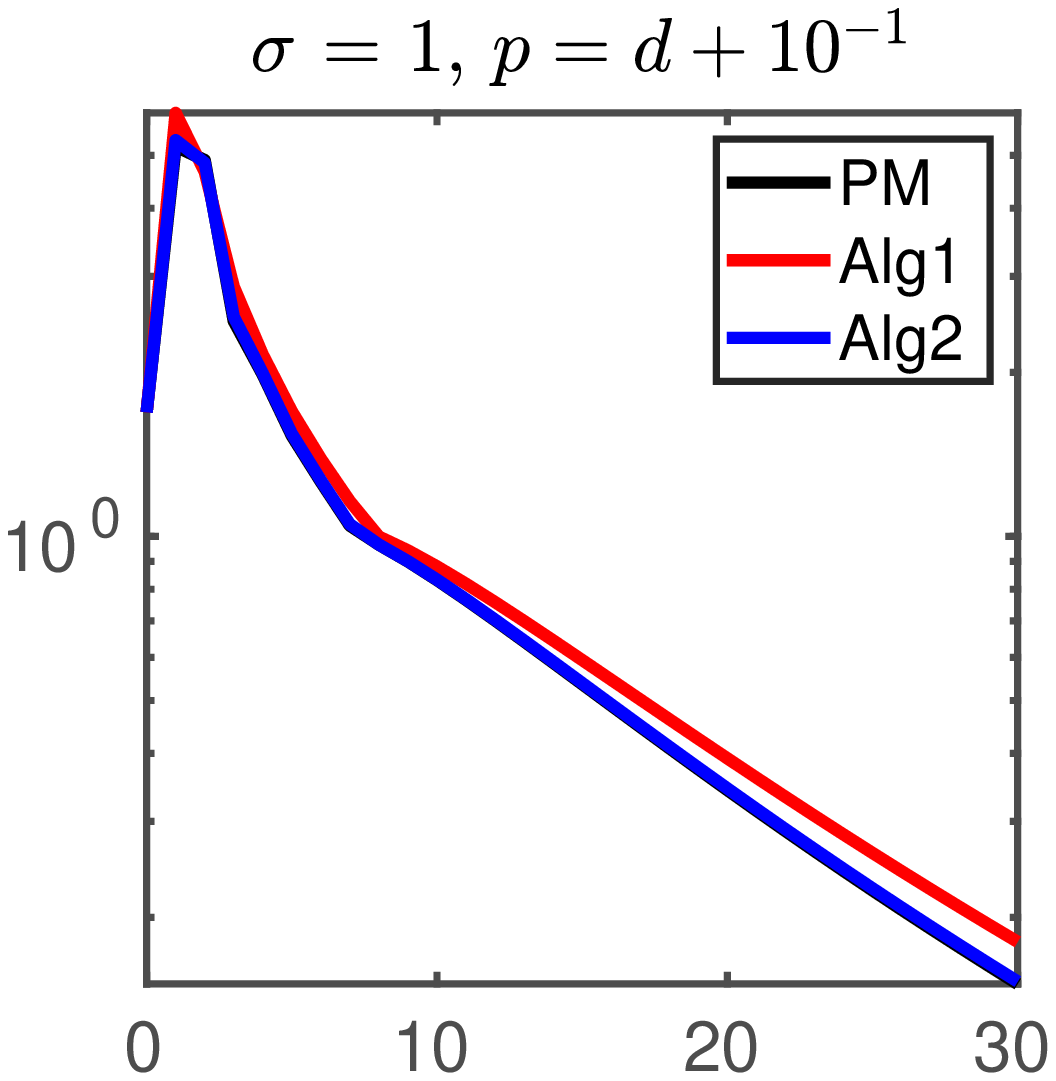}
	\includegraphics[width=.35\textwidth]{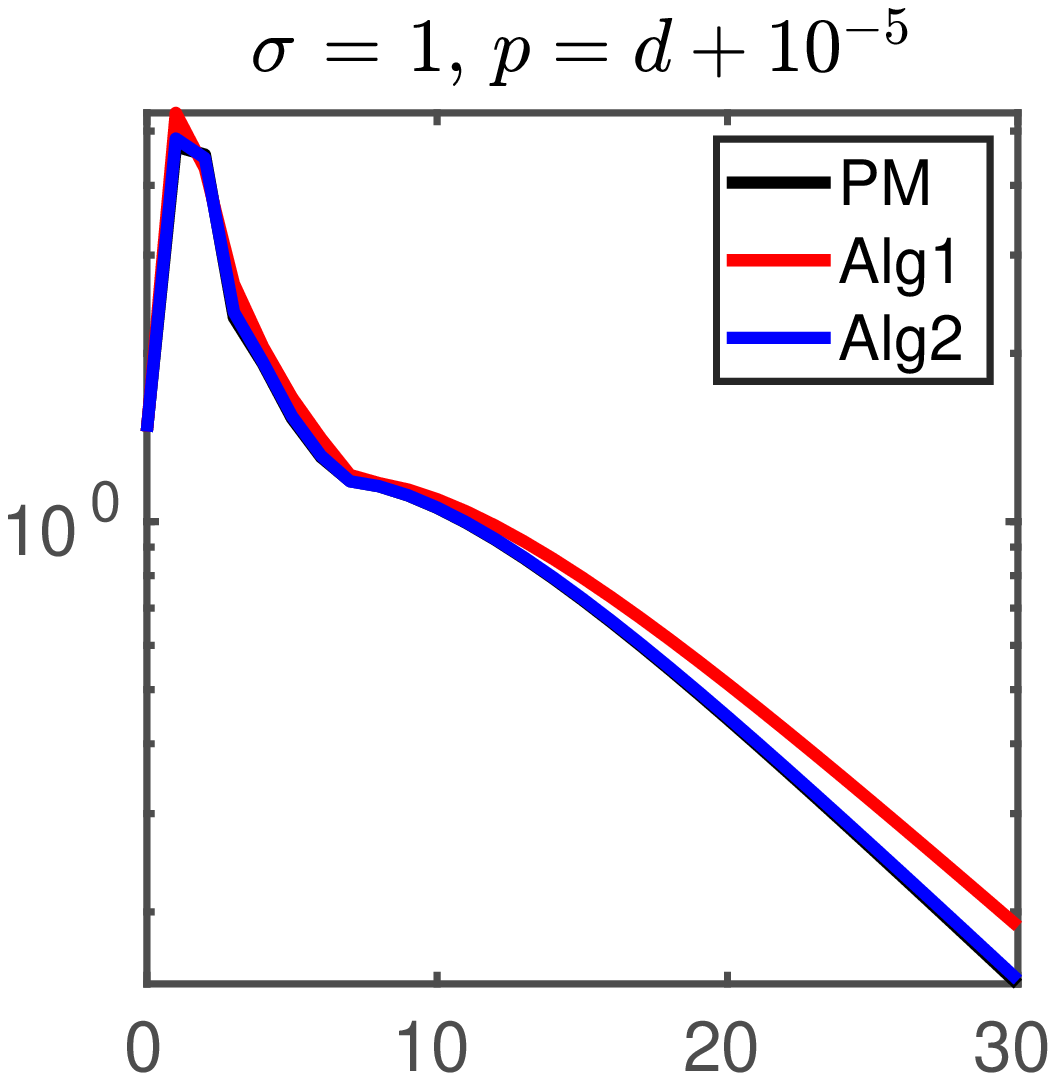}
	\vspace{-0.3cm}
	\caption{Experiments on $\b T\in\RR^{n\times n\times n}, \; n=2361$, generated using \texttt{yeast}. }\label{fig:tcycle2}
%\end{figure}
%\begin{figure}[p]

\vspace{0.3cm}
	\centering
	\includegraphics[width=.35\textwidth]{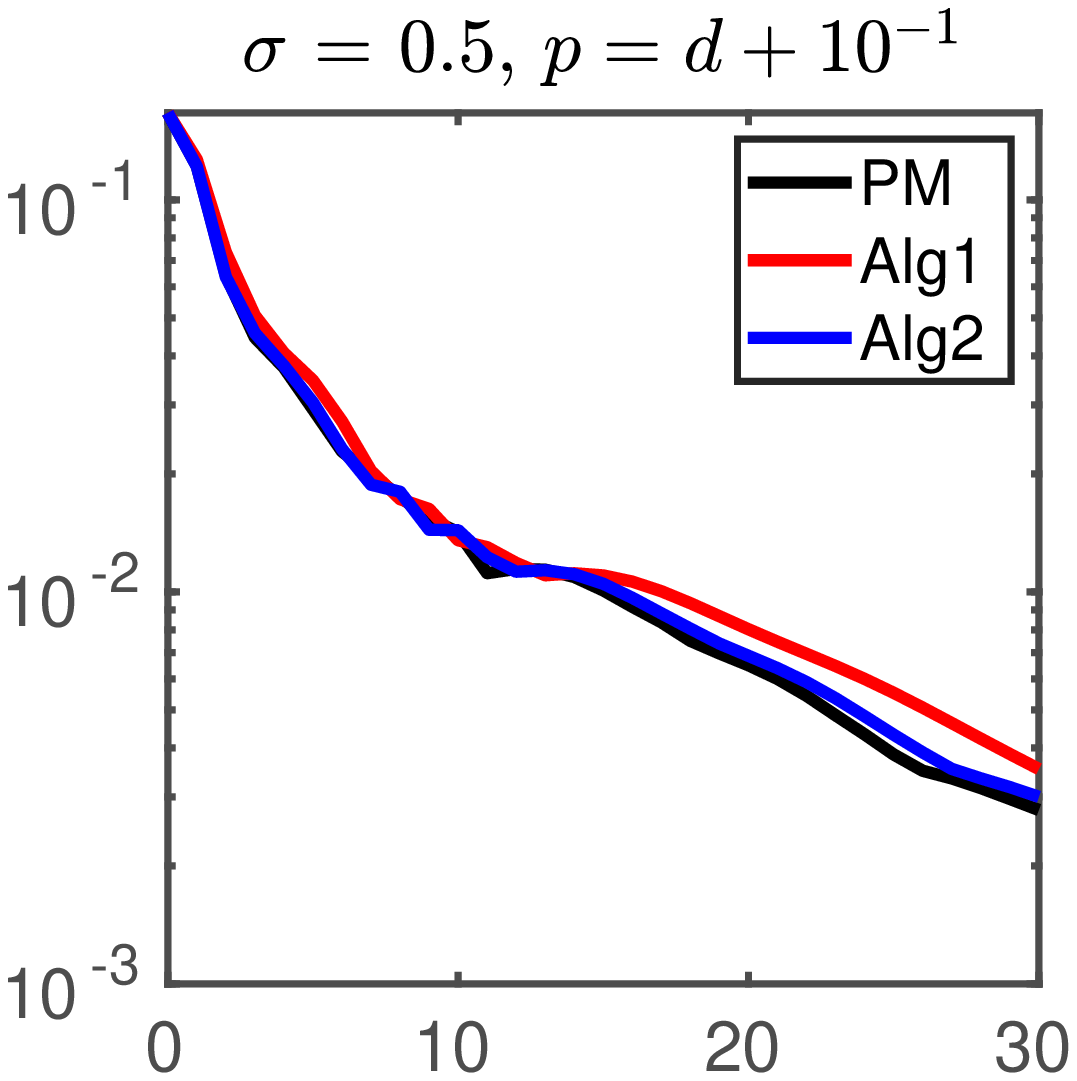}
	\includegraphics[width=.35\textwidth]{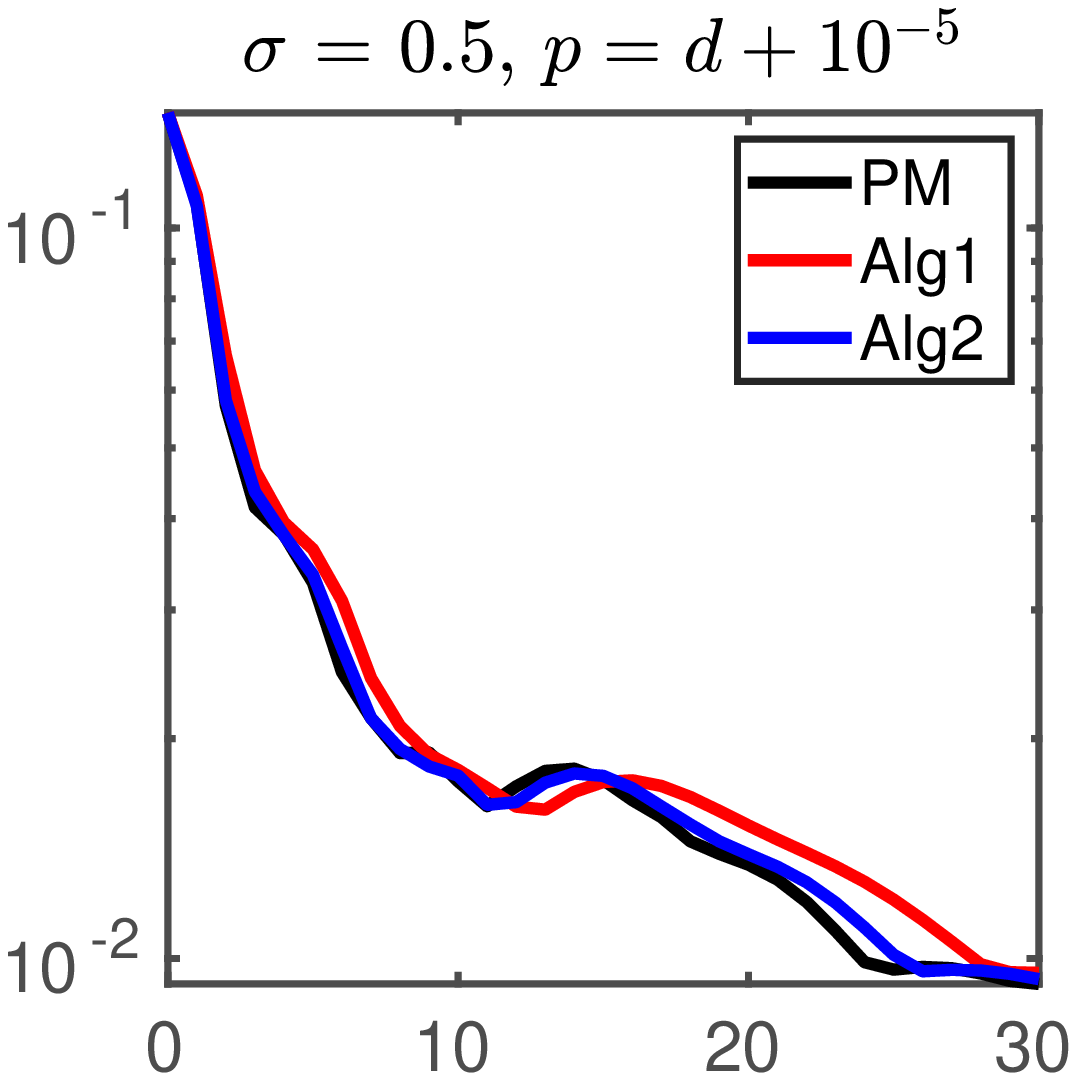}\\
	\includegraphics[width=.35\textwidth]{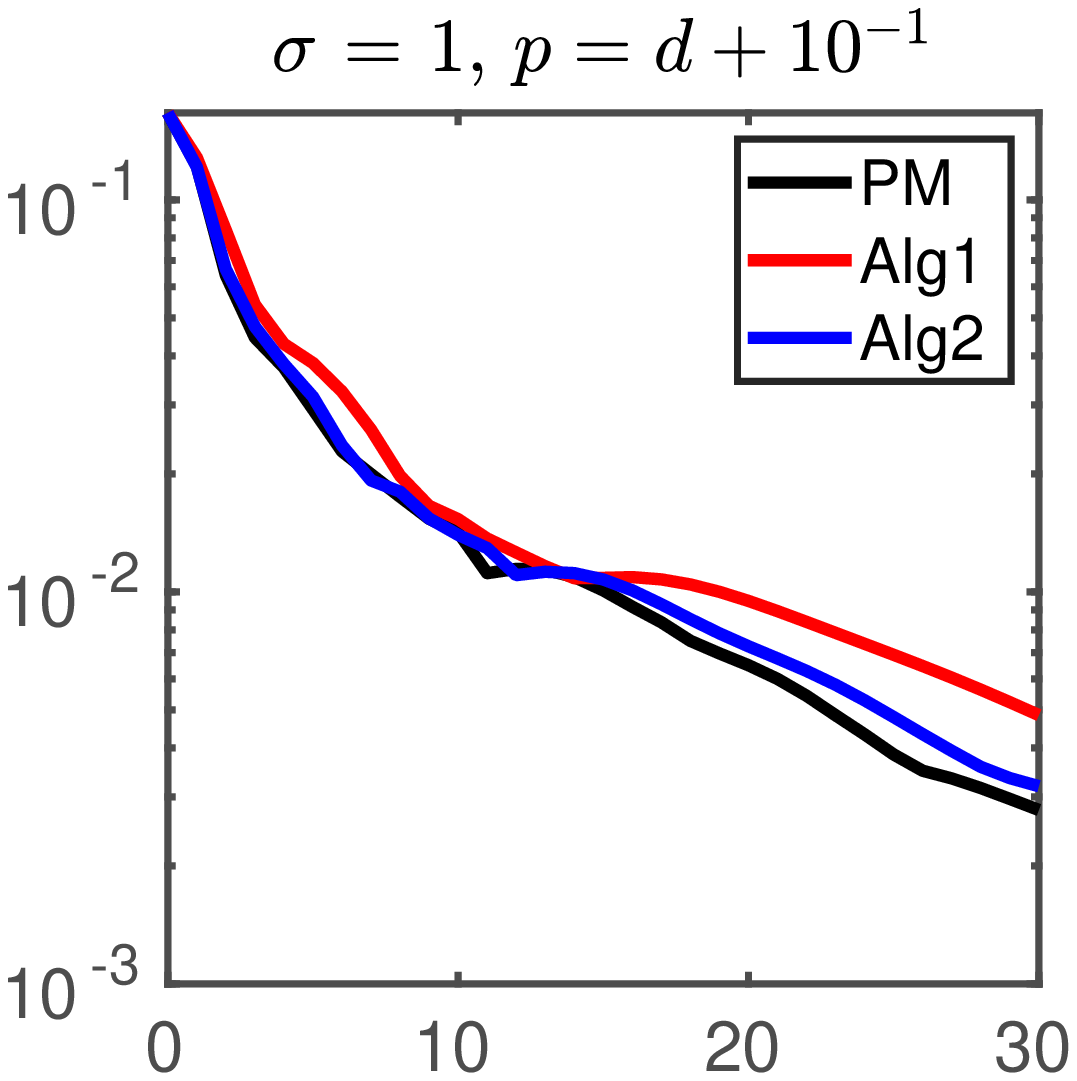}
	\includegraphics[width=.35\textwidth]{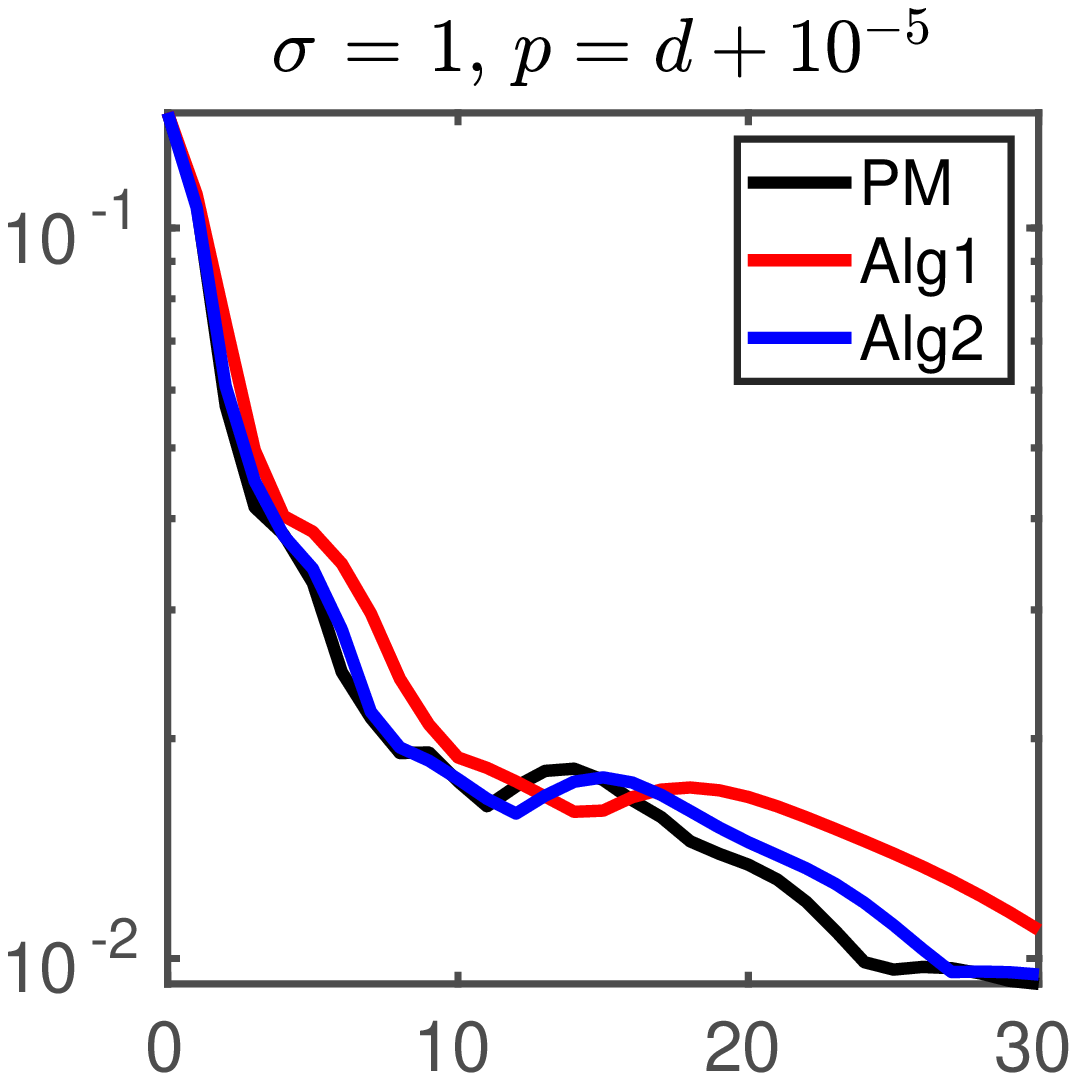}\\
	\vspace{-0.3cm}
	\caption{Experiments on $\b T\in\RR^{n\times n\times n}, \; n=1107$, generated using \texttt{gre$1107$}. }\label{fig:tcycle3}
\end{figure}

\begin{figure}[t]
	\centering
	\includegraphics[width=.35\textwidth]{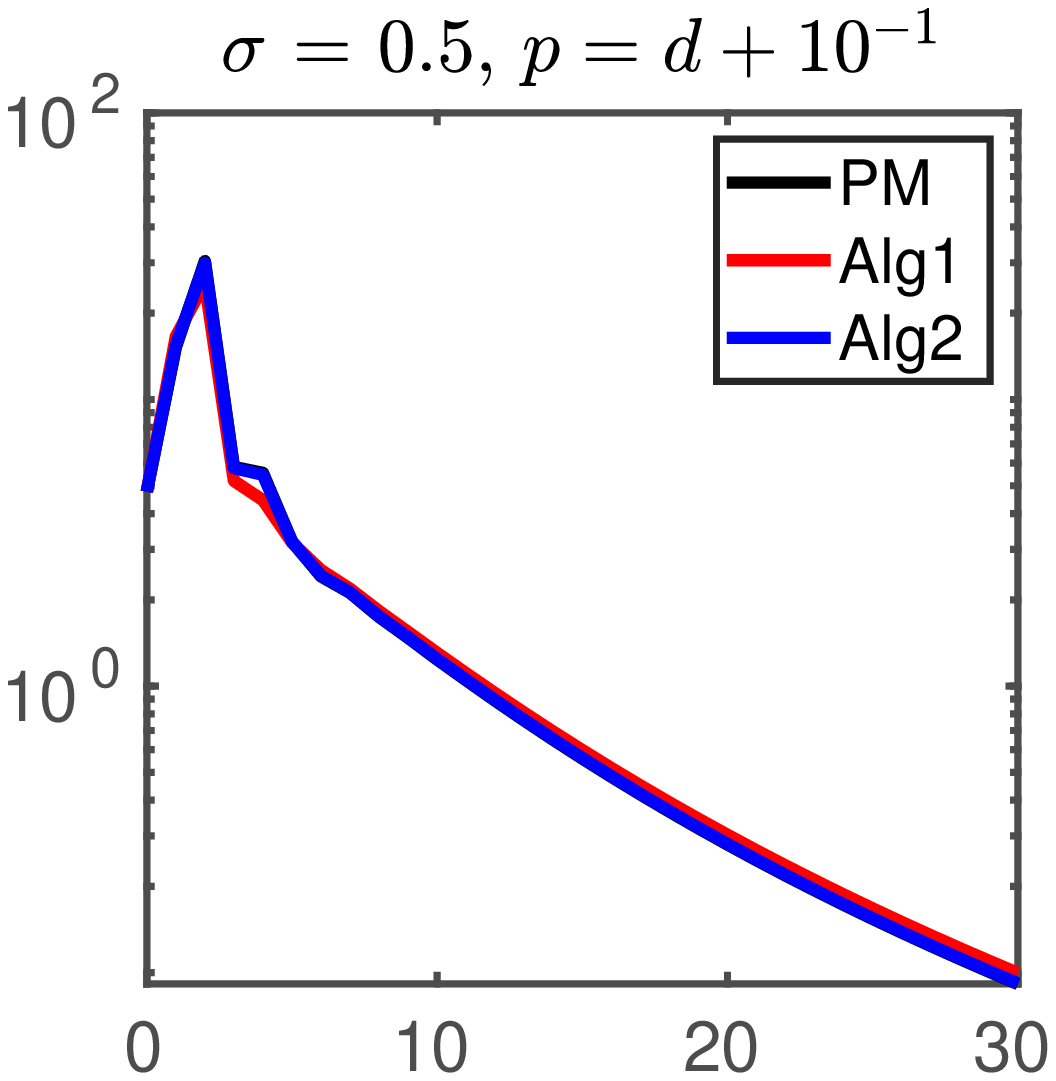}
	\includegraphics[width=.35\textwidth]{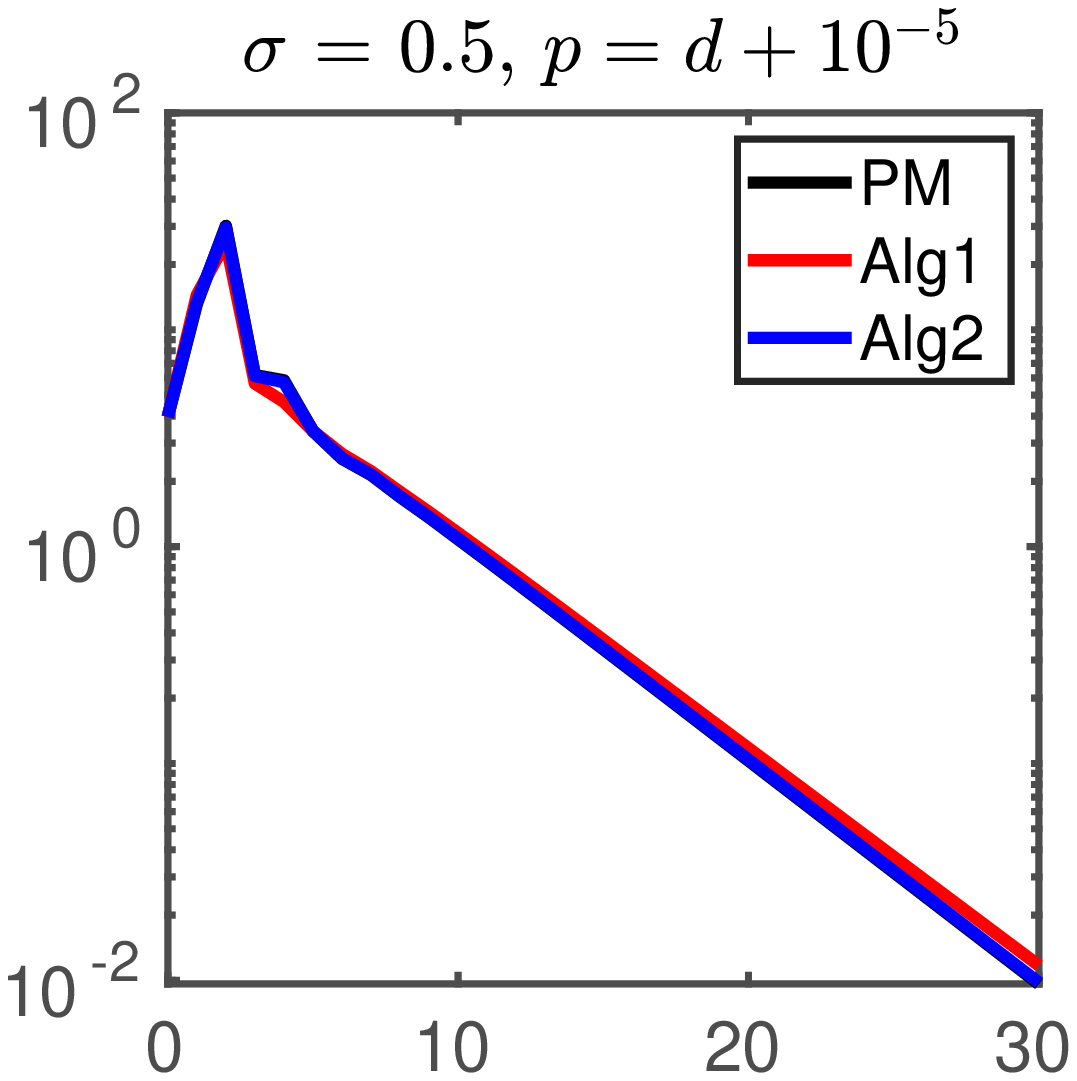}\\
	\includegraphics[width=.35\textwidth]{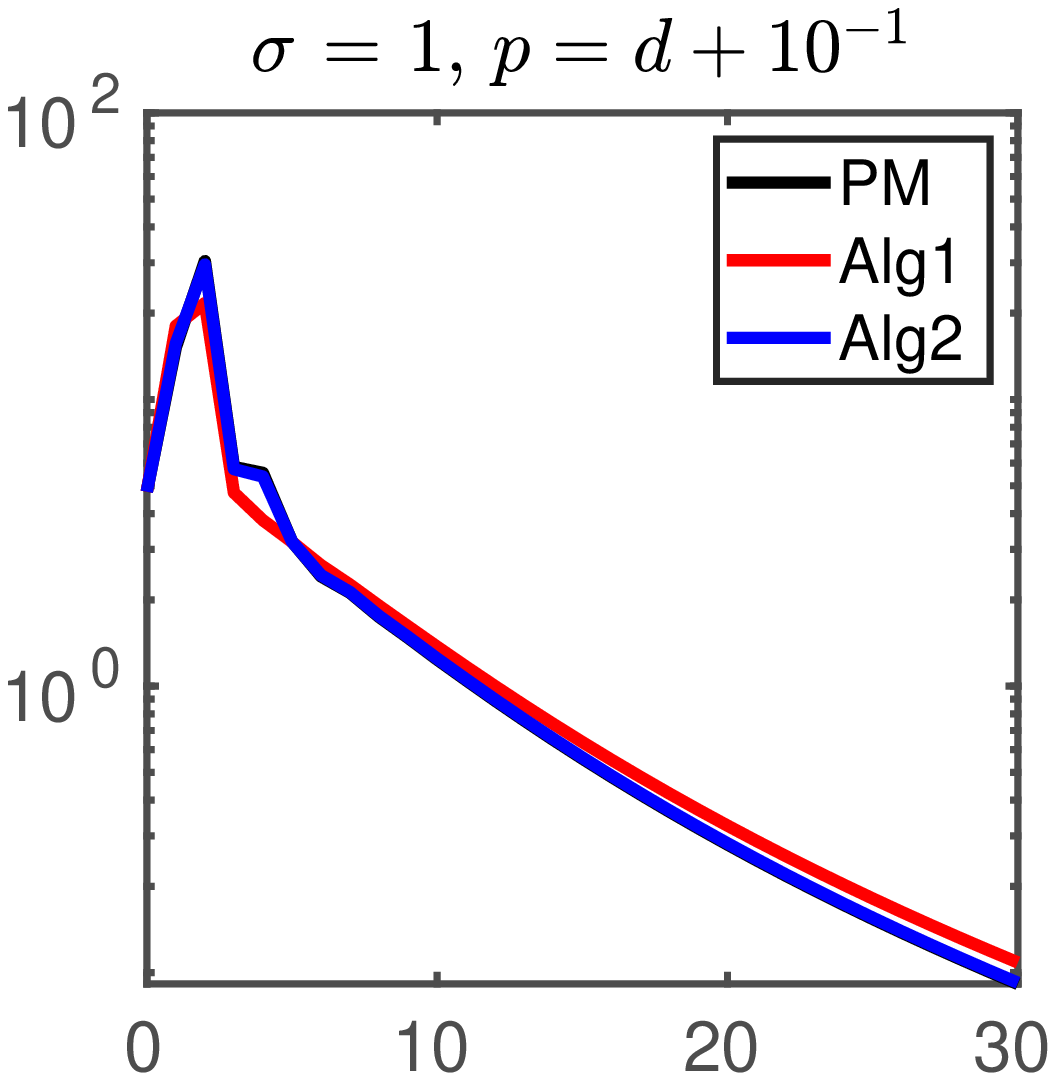}
	\includegraphics[width=.35\textwidth]{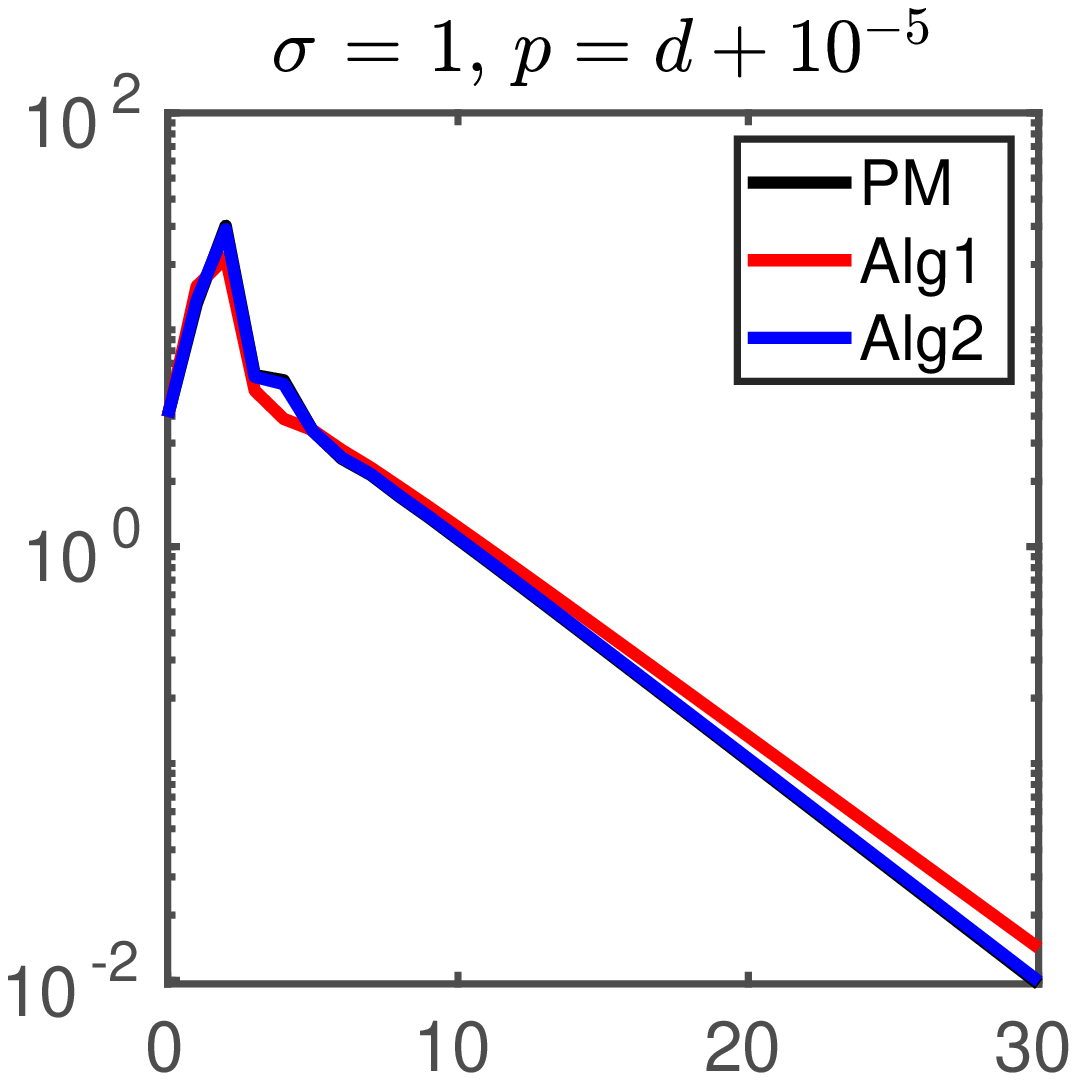}\\
	\vspace{-0.3cm}
	\caption{Experiments on $\b T\in\RR^{n\times n\times n}, \; n=9914$, generated using \texttt{wb-cs-stanford}. }\label{fig:tcycle4}
\end{figure}
We then compute the Perron $\ell^p$-eigenvector of {$\b T$} for different values of $p\approx d$, in order to quantify the hypergraph centrality of the network, following the eigenvector centrality model discussed in \cite{benson2019three}. In our tests we used  four real-world directed and undirected network datasets of different size coming from \cite{davis2011university}, as listed in  Table \ref{tab:datasets}. Results up to iteration $30$ are shown in Figures \ref{fig:tcycle1}, \ref{fig:tcycle2}, \ref{fig:tcycle3} and \ref{fig:tcycle4} where we plot the residual $\|{T}(\b x_k) -\lambda_k \Phi_p(\b x_k)\|_\infty$ for the four different three-cycle tensors obtained by the chosen datasets and for the three sequences obtained  with  Algorithms \ref{alg:shiftedPM1} and \ref{alg:shiftedPM2}, for the choices of $\sigma=0.5$ and $\sigma=1$,  and  the unshifted power method. %The starting point in all the tests is a random vector.

\begin{table}[hbt]
	\centering  \fontsize{9}{5}\selectfont
	\begin{tabular}{|llll|}
		\hline
		Problem name   &  Size & nnz({adj(G)}) & nnz(${\b{T}}$) \\
		\hline
		\texttt{Dolphins} (undirected)  &  62    & 318 & 570 \\
		\texttt{yeast} (undirected)  &  2361 & 13828 &  35965 \\
		\texttt{gre$1107$} (directed)   &    1107  &  	5664 &  11045 \\
		\texttt{wb-cs-stanford}  (directed)   & 9914 & 36854 &  101992 \\
\hline	
\end{tabular}
	\caption{Datasets' information: {the first column `size' refers to the number of nodes in the network, whereas the two colums on the right show the number of nonzeros in the adjacency matrix (i.e.\ twice the number of edges) and in the three-cycle tensor, respectively.} All the three-cycle tensors generated from these datasets are such that ${T}(\uno) \geq 0$ except for \texttt{gre$1107$} for which strict inequality holds. Source \cite{davis2011university}.}\label{tab:datasets}
\end{table}

While we observe that the residual decreases, as expected, unlike previous example tests, all the three methods do not perform well for these datasets. In fact, over $50$ iterations are often not enough to achieve more than 2 digits of precision. For this reason, we introduce in the next section an extrapolation strategy. With this technique we {speed-up the original sequence by using the \textit{Simplified Topological Shanks' Transformations} and the corresponding algorithms (called in short STEA's) \cite{brezinski2014simplified,bmrz2017software,brezinski2017simplified}, obtaining extrapolated methods that achieve competitive performance on the considered network data.} %We will see that  this additional numerical transformation

\section{Extrapolation for fixed-point iterations} \label{sec:extrapolation}
The numerical experiments carried out in  Section \ref{sec:experiments_part1} have shown that on some class of problems (Section \ref{sec:real_world_data_tensors}) the rate of convergence of Algorithms \ref{alg:shiftedPM1}, \ref{alg:shiftedPM2} and of the power method can be quite slow. This can affect and limit the applicability of such methods for real world problems. Motivated by this observation, in this section, we introduce extrapolation techniques for accelerating the convergence of the sequence $(\b x_k)$ generated by Algorithm \ref{alg:shiftedPM1} or \ref{alg:shiftedPM2}.

The theory of extrapolation methods have been developed and successfully applied to a variety of problems, such as the solution of linear and nonlinear systems, matrix eigenvalue problems, the computation of matrix functions, the solution of integral equations, and many others
\cite{bouhamidi2011extrapolated,brezinski2013extrapolationbook,sidi2017vector,brezinski2017simplified,cipolla2019extrapolation}.
These  methods transform the original sequence $(\b x_k)$  into a new sequence $(\b y_k)$  by means of a {\it sequence transformation},
which, under some assumptions, converges faster to the limit.
The idea behind such transformations is typically  to assume that the original
sequence  $(\b x_k)$ behaves like a model sequence whose limit $\b x$ can be  computed exactly by a finite algebraic process.
The set ${\cal K}$ of these model sequences is called the {\it kernel} of the transformation. If the sequence
$(\b x_k)$ belongs to the kernel ${\cal K}$, then the transformed sequence ``converges in one step'', i.e.,
the sequence is transformed into the constant sequence, where the constant is the limit of the original sequence. If the
sequence $(\b x_k)$ does not belong to the kernel but it is ``close enough'' to it, then  there is a good chance
that the transformed sequence converges, to the same limit,   faster than the original sequence.
Among the existing sequence transformations  (also called {\it extrapolation methods}), the Shanks' transformation
\cite{brezinski2019genesis,brezinski2018shanks,shanks1955non} is arguably the best all-purpose method for accelerating the convergence of a sequence.
The kernel of the vector Shanks' transformation $\mathcal K_S$ contains the set of sequences $(\b x_k)$ for which there exists an integer $h$ such that for all $k$ we~have %can be represented by the difference equation
\begin{equation}
a_0({\b x}_k-{\b x})+\cdots +a_h({\b x}_{k+h}-{\b x})=0,
\quad k=0,1,\ldots
\label{kernelnew}
\end{equation}
for some real coefficients $a_i$ such that $a_0a_h \neq 0$ and $a_0+\cdots+a_h \neq 0$.
If we assume, without loss of generalities,  that $a_0+\cdots+a_h=1$, then for
each sequence of the kernel we have%, for a generic sequence $(\b x_k)$~in~$\mathcal{K}_S$,
\begin{equation} \label{eq:limit_linear_c}
	\b x = a_0 \b x_k +\dots + a_{h} \b x_{k+h}.
\end{equation}

Of course, even if the  given sequence $(\b x_k)$ does not  belong to the Shanks' kernel, we may apply the transformation, but in this case we obtain  a set of sequences, usually denoted
 $(\b e_h(\b x_k))$,  depending on $k$ and $h$, whose elements are given by
\begin{equation}\label{eq:shanks}
{\b e}_h({\b  x}_k) =	a_0^{(k,h)}{\b  x}_k+\cdots +a_h^{(k,h)}{\b x}_{k+h}
\end{equation}
where the coefficients $a_i^{(k,h)}$ are such that if $(\b x_k) \in \mathcal K_S$ then ${\b e}_h({\b  x}_k)=\b x$ for all $k$.
%That is, if the sequence $(\b x_k)$ belongs to the kernel ${\cal K_S}$, then it is transformed into a constant sequence
%whose terms are all equal to the limit $\b x$ of the sequence $(\b x_k)$.
For a recent survey on Shanks' based transformations see \cite{brezinski2018shanks}.
% {Several vector sequence transformations
% 	based on  such a kernel %already
% 	exist and have
% 	been introduced and studied by various authors (the $\varepsilon$-algorithms, the MMPE, the MPE, the RRE, the E-algorithm, etc., see \cite{brezinski2013extrapolationbook,brezinski2018shanks} for exhaustive reviews).}
% For all of them,
% the following theorem  holds:
% %\begin{theorem}[\cite{brezinski1975generalisations}]
% \begin{theorem} \label{THkernel}
% 	If there exist $h$ and $a_0, \ldots, a_h$ with $a_0a_h \neq 0$ and $a_0+ \cdots + a_h \neq 0$ such that, for all $k$,
% 	$$
% 	a_0({\b x}_k-{\b x})+\cdots +a_h({\b x}_{k+h}-{\b x})=0,
% 	$$
% 	then, for all $k$, ${\b e}_h({\b x}_k)={\b x}$.
% \end{theorem}

Some of these Shanks'  transformations can be implemented  recursively, and, among these, the Topological Shanks' Transformations can be implemented  by the Topological Epsilon Algorithms, in short TEA's \cite{brezinski1975generalisations}.
Recently, simplified versions of these algorithms, called the {Simplified Topological Epsilon Algorithms}
(STEA's), have been introduced. These simplified algorithms have three main advantages with respect to the original ones:
the numerical stability can be sensibly improved, the rules defining the extrapolated sequence are simpler and
the computational cost is reduced both in terms of memory allocation and in terms of operations to be performed.

Finally, let us observe that Shanks' transformations can be coupled with a restarting technique which is particularly suited for fixed-point problems (see \cite{brezinski2017simplified, cipolla2019extrapolation}) and that roughly goes as follows.
Assume that we are interested in a fixed-point $\b x$ of a mapping $\mathcal F: {\mathbb R}^n \rightarrow
{\mathbb R}^n$.  We compute a certain number of basic iterates $\b x_{k+1}= \mathcal F(\b x_k)$ from a given ${\b x}_0$.
Then we apply the extrapolation algorithm to them, and we restart the
basic iterates from the computed extrapolated term and so on.
The advantage of this approach is that, under suitable regularity assumptions on $\mathcal F$ and if  $h$ is large enough, the sequence generated in this way  converges  quadratically to the fixed point of $\mathcal F$ \cite{le1992quadratic}.

In the next Section \ref{sec:topological_epsilon_algorithm}, we briefly review  the
topological  Shanks' transformations and their simplified versions (in particular the STEA2 algorithm that  is the less expensive in terms of memory requirement). For further details, see
\cite{brezinski2014simplified,brezinski2017simplified}. Whereas,  in  Section \ref{sec:stea_for_powermethod} we discuss the details of the restarted procedure and how this applies to  the specific $\ell^p$-eigenvector setting.
\begin{figure}[tb]
	\centering
	\fbox{
		\begin{tikzpicture}%[every node/.style={circle,draw=black}]
		\newcommand*{\cx}{2.5}% x cohordinate of columns
		\newcommand*{\cy}{1} % y cohordinate of rows
		\newcommand*{\s}{1.5} % arrows tips scale
		\newcommand{\myvdots}[1]{\raisebox{#1\baselineskip}{\ensuremath{\vdots}}}
		
		\node (00) at (0,8*\cy) {$\b z_0^{(0)}\!\!=\!\b x_0$};
		\node (10) at (0,7*\cy) {$\b z_0^{(1)}\!\!=\!\b x_1$};
		\node (20) at (0,6*\cy) {$\b z_0^{(2)}\!\!=\!\b x_2$};
		\node (30) at (0,5*\cy) {$\b z_0^{(3)}\!\!=\!\b x_3$};
		\node (40) at (0,4*\cy) {$\b z_0^{(4)}\!\!=\!\b x_4$};
		\node (50) at (0,3*\cy) {$\b z_0^{(5)}\!\!=\!\b x_5$};
		\node (60) at (0,2*\cy) {$\b z_0^{(6)}\!\!=\!\b x_6$};
		\node (70) at (0,1*\cy) {$\b z_0^{(7)}\!\!=\!\b x_7$};
		\node (80) at (0,0*\cy) {$\b z_0^{(8)}\!\!=\!\b x_8$};
		\node      at (0,-.5*\cy) {$\vdots$};
		\node (b0) at (0,-1.3*\cy){$(\widetilde{\b e}_0(\b x_k))\!=\! (\b x_k)$};
		
		\node (01) at (1*\cx,7*\cy) {$\b z_2^{(0)}$};
		\node (11) at (1*\cx,6*\cy) {$\b z_2^{(1)}$};
		\node (21) at (1*\cx,5*\cy) {$\b z_2^{(2)}$};
		\node (31) at (1*\cx,4*\cy) {$\b z_2^{(3)}$};
		\node (41) at (1*\cx,3*\cy) {$\b z_2^{(4)}$};
		\node (51) at (1*\cx,2*\cy) {$\b z_2^{(5)}$};
		\node (61) at (1*\cx,1*\cy) {$\b z_2^{(6)}$};
		\node      at (1*\cx,0*\cy) {$\myvdots{1}$};
		\node      at (1*\cx,-.5*\cy) {$\vdots$};
		\node (b1) at (1*\cx,-1.3*\cy){$(\widetilde{\b e}_1(\b x_k))$};
		
		\node (02) at (2*\cx,6*\cy) {$\b z_4^{(0)}$};
		\node (12) at (2*\cx,5*\cy) {$\b z_4^{(1)}$};
		\node (22) at (2*\cx,4*\cy) {$\b z_4^{(2)}$};
		\node (32) at (2*\cx,3*\cy) {$\b z_4^{(3)}$};
		\node (42) at (2*\cx,2*\cy) {$\b z_4^{(4)}$};
		\node      at (2*\cx,1*\cy) {$\myvdots{1}$};
		\node      at (2*\cx,0*\cy) {$\myvdots{1.5}$};
		\node 	   at (2*\cx,-.5*\cy) {$\vdots$};
		\node (b2) at (2*\cx,-1.3*\cy){$(\widetilde{\b e}_2(\b x_k))$};

		\node (03) at (3*\cx,5*\cy) {$\b z_6^{(0)}$};
		\node (13) at (3*\cx,4*\cy) {$\b z_6^{(1)}$};
		\node (23) at (3*\cx,3*\cy) {$\b z_6^{(2)}$};
		\node      at (3*\cx,2*\cy) {$\myvdots{1}$};
		\node      at (3*\cx,1*\cy) {$\myvdots{1.5}$};
		\node      at (3*\cx,0*\cy) {$\myvdots{1.7}$};
		\node 	   at (3*\cx,-.5*\cy) {$\vdots$};
		\node (b3) at (3*\cx,-1.3*\cy){$(\widetilde{\b e}_3(\b x_k))$};

		\node (04) at (4*\cx,4*\cy) {$\b z_8^{(0)}$};
		\node      at (4*\cx,3*\cy) {$\myvdots{1}$};
		\node      at (4*\cx,2*\cy) {$\myvdots{1.5}$};
		\node      at (4*\cx,1*\cy) {$\myvdots{1.7}$};
		\node      at (4*\cx,0*\cy) {$\myvdots{1.9}$};
		\node 	   at (4*\cx,-.5*\cy) {$\vdots$};
		\node (b4) at (4*\cx,-1.3*\cy){$(\widetilde{\b e}_4(\b x_k))$};

		\foreach \K [count=\x] in {7,5,3,1}{
			\pgfmathsetmacro\col{int(\x-1)}
			\foreach \i in {1,...,\K}{
				\pgfmathsetmacro\j{int(\i-1)}
				\pgfmathsetmacro\k{int(\i+1)}
				\draw[thick,-{>[scale=\s]}] (\i\col)-- (\j\x);
				\draw[thick,-{>[scale=\s]}] (\k\col) --(\j\x.south west);
			}
		}
		
		% \foreach \a [count=\bb] in {8,6,4,2,0}{
		% \pgfmathsetmacro\b{int(\bb-1)}
		% 	\draw[thick,dotted] (\a\b.south)--(b\b);
		% }
		
		\end{tikzpicture}
	}
	\caption{Diagramatic illustration of the ``triangular'' recursive rule \eqref{eq:recursive_rule} for the computation of the extrapolated sequence $(\tilde{\b e}_h(\b x_k)) = (\b z_{2h}^{(k)})$, for $h = 1,2,3,4$. The arrows emphasize the  dependence of $\b z_{2j+2}^{(i)}$ upon only $\b z^{(i+2)}_{2j}$ and $\b z^{(i+1)}_{2j}$.}\label{fig:triangular_scheme}
\end{figure}
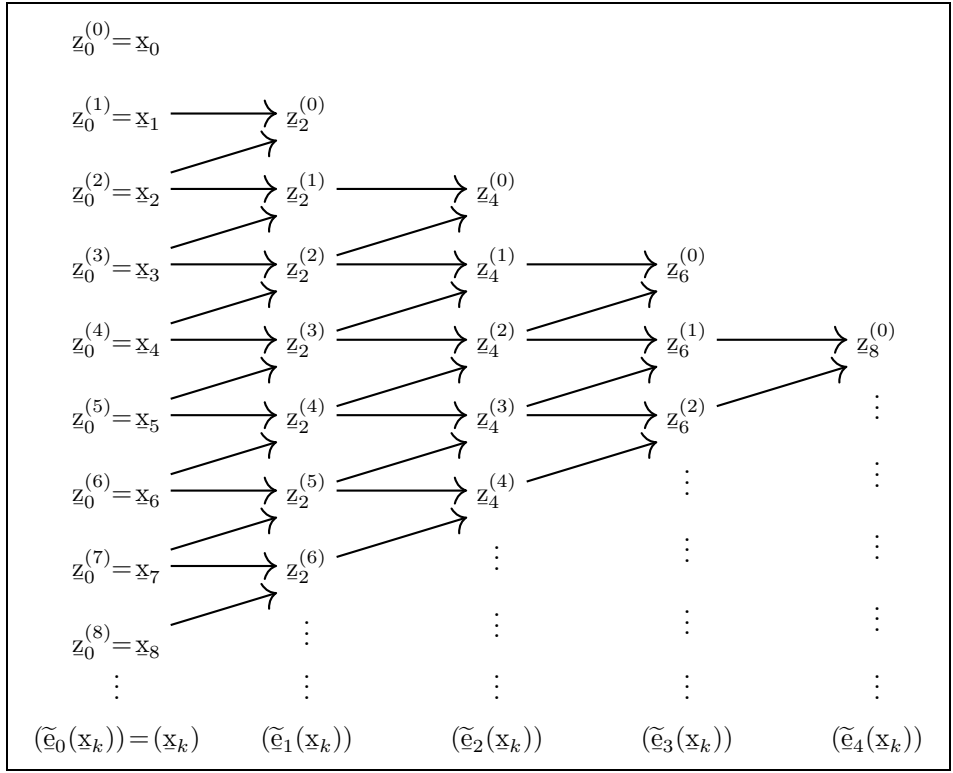
%

% \subsection{The topological Shanks transformations and STEA algorithms} \label{sec:topological_epsilon_algorithm}
\subsection{Topological Shanks' transformations and the STEA2 algorithm} \label{sec:topological_epsilon_algorithm}
For $p>d$, let $(\b x_k)$ be the sequence generated by either Algorithm \ref{alg:shiftedPM1} or \ref{alg:shiftedPM2} and let $\b x$ be the Perron $\ell^p$-eigenvector limit of such sequence.
The so-called \textit{second Topological Shanks' Transformation}  starts from  the original
sequence and, given an arbitrary  nonzero $\b y \in \RR^n$ and a fixed integer $h$, produces a new
sequence $(\widetilde{\b e}_h({\b x}_k))$ defined as
%where each term uses $2k+1$ vectors and has the form
\begin{align}
% \widehat{\b e}_h({\b  x}_k)&=a_0^{(k,h)}{\b x}_k+\cdots+a_h^{(k,h)}{\b x}_{k+h}, \label{eq:shanks_transformations1}\\
\widetilde{\b e}_h({\b x}_k)&=a_0^{(k,h)}{\b x}_{k+h}+\cdots+a_h^{(k,h)}{\b x}_{k+2h}, \label{eq:shanks_transformations2}
\end{align}
 where the coefficients $a_i^{(k,h)}$ are the solutions of the linear system
\begin{equation}\label{eq:linear_system_neq_sequence}
\begin{bmatrix}
1 & \dots & 1 \\
b_0 & \dots & b_{h} \\
\vdots & & \vdots \\
b_{h-1} & \dots & b_{2h-1}
\end{bmatrix}
\begin{bmatrix}
a_0^{(k,h)} \\
 \vdots\\
  a_h^{(k,h)}
\end{bmatrix} =
\begin{bmatrix}
1 \\
0\\
 \vdots\\
 0
\end{bmatrix}
\end{equation}
% \begin{equation}\label{eq:linear_system_neq_sequence}
% S^{(k,h)}:= \begin{bmatrix}
% 1 & \dots & 1 \\
% b_0 & \dots & b_{h} \\
% \vdots & & \vdots \\
% b_{h-1} & \dots & b_{2h-1}
% \end{bmatrix}, \;\; \begin{array}{l}b_i:= \b w^T(\b x_{k+i+1}-\b x_{k+i}),\;\; i=0,\dots,h-1, \\ \b a^{(k,h)}  = [a_0^{(k,h)}, \dots, a_h^{(k,h)}]^T\end{array}
% \end{equation}
and $b_i:= \b y^T(\b x_{k+i+1}-\b x_{k+i}),$ for $i=0,\dots,2h-1$. %Observe, moreover, that the first relation in \eqref{eq:linear_system_neq_sequence} is a normalization condition that does not restrict the generality.
Note that  for this transformation it holds that if $(\b x_k) \in \mathcal{K}_S$, then $\widetilde{\b e}_h({\bf x}_k)={\bf x}$ for all $k$, i.e., all the
transformed terms coincide with the limit. However, we often do not know whether the original
sequence satisfies \eqref{eq:shanks} and, if so, we do not know what is the correct value of $h$. So, in practice,
we fix an arbitrary integer $h$ and transform the sequence via \eqref{eq:shanks_transformations2} using such an integer.
% This transformation can be used also when the original sequence does not belong to
% the kernel, i.e.\ it does not satisfy \eqref{kernelnew}. In this case the coefficients
% $a_i$ of \eqref{eq:shanks} depend on $k$ and $h$, and this dependence is  emphasized by the upper indices of the coefficients in \eqref{eq:shanks_transformations2}.
%%
%

The second Simplified Topological $\epsilon$-Algorithm (STEA2) \cite{brezinski2014simplified} allows us to compute
the terms of the new sequence $\widetilde{\b e}_h({\bf x}_k)$ via four equivalent recursive formulas without
solving explicitly the linear system given by \eqref{eq:linear_system_neq_sequence}. Here we focus on the third one.

Define the vectors $\b z_{2j}^{(i)}$ as follows: set $\b z^{(k)}_{0}=\mathbf{x}_k$ and, for  $i,j=0,1,\dots$, compute
%Observe, moreover, that in general the simplified forms of the topological $\varepsilon$-algorithms \cite{brezinski2014simplified, brezinski2017simplified}, allow an implementation  that avoids the manipulation of the elements of $E^*$
%(in contrast to the topological $\varepsilon$-algorithms \cite{brezinski1975generalisations}), since the linear
%functional $\b y$ is applied only to the terms of the initial sequence $({\bf s}_\ell)$ and is not used in the recursive rule.
% The simplified algorithms implementing the first transformation
% \eqref{eq:shanks_transformations1} and the second one \eqref{eq:shanks_transformations2} are usually denoted by STEA1 and STEA2, respectively. In this work we focus mostly on STEA2 as it requires less memory storage. Moreover, among the four  existing equivalent updating rules \cite{brezinski2014simplified}, we use the one we observed being the most effective, defined by
% \begin{equation}
% {\widetilde{\boldsymbol {\varepsilon}}}^{(k)}_{2h+2}={\widetilde{\boldsymbol{\varepsilon}}}^{(k+1)}_{2h}+
% {\frac{\epsilon_{2h+2}^{(k)}		 -\epsilon_{2h}^{(k+1)}}{\epsilon_{2h}^{(k+2)}-\epsilon_{2h}^{(k+1)}}}({\widetilde{\boldsymbol{\varepsilon}}}^{(k+2)}_{2h}-
% {\widetilde{\boldsymbol {\varepsilon}}}^{(k+1)}_{2h}), \;\;k,h=0,1,\dots
% \end{equation}
\begin{equation}\label{eq:recursive_rule}
\b z^{(i)}_{2j+2}={\b z}^{(i+1)}_{2j}+
\frac{\epsilon_{2j+2}^{(i)}-\epsilon_{2j}^{(i+1)}}{\epsilon_{2j}^{(i+2)}-\epsilon_{2j}^{(i+1)}}  (\b z^{(i+2)}_{2j}-
\b z^{(i+1)}_{2j}), %\quad \begin{array}{l} i=k,\dots, k+2h,\\ j = 0,\dots,h-1 \end{array}
\end{equation}
 where the scalar quantities $\epsilon_{j}^{(i)}$ are given by the Wynn's scalar $\varepsilon$-algorithm \cite{wynn1956} (an algorithm implementing the Shanks' transformation for scalar sequences) applied to
the sequence $(\b y^T \b x_0), (\b y^T \b x_1), (\b y^T \b x_2),\dots$.
The link with the transformation is given by the fact that we have
$\b z^{(k)}_{2h}=\widetilde{\mathbf{e}}_h(\mathbf{x}_k)$ (see e.g.\ \cite{brezinski1975generalisations}),
and thus \eqref{eq:recursive_rule} allows us to compute the desired extrapolated sequence $(\widetilde{\mathbf{e}}_h(\mathbf{x}_k))$ with few simple operations. Indeed, note that \eqref{eq:recursive_rule} contains only sums and differences of vectors  and it relies only on two terms of a triangular scheme, as illustrated in Figure \ref{fig:triangular_scheme}.
Of course there are also the inner products for the inizialization of  Wynn's scalar
$\varepsilon$-algorithm.

\subsection{Restarted extrapolation method for $\ell^p$-eigenpairs}\label{sec:stea_for_powermethod}
When dealing with fixed-point problems $\mathcal{F}({\bf x})={\bf x}$, as previously pointed out, it is often advisable to couple the extrapolation method with a restarting technique \cite{brezinski2017simplified}.
If we consider the STEA2 scheme, the general restarted method is presented in Algorithm~\ref{alg:restarted_method}.
\begin{algorithm}[t]
\caption{Restarted extrapolation method}\label{alg:restarted_method}
	\DontPrintSemicolon
	\KwIn{Choose $2h, \;\mathsf{cycles}\; \in \mathbb{N}, \b x_0 \hbox{ and } \b y \in \RR^{n}$}
	\For {$i=0,1, \ldots, \mathsf{cycles}$ (outer iterations)}{
	% Set ${\bf s}_0=\b x_i$\;
	 Compute ${x_{0}} = {\bf{ y}}^T{{\bf{ x}}}_{0}$\;
	\For{$k=1, \dots, 2h$ (inner iterations)}{
	Compute ${\bf x}_k=\mathcal F({\bf x}_{k-1})$\;
	Compute ${x_{k}} = {\bf{ y}}^T{{\bf{ x}}}_{k}$}
	Apply STEA2 to ${\bf x}_0, \dots {\bf x}_{2h}$  and ${x}_0, \dots {x}_{2h}$ to compute   $\b z^{(0)}_{2h}=\widetilde{\b {e}}_h(\b {x}_0)$ \;
	Set $\b x_{0}=  \b z^{(0)}_{2h}$\;
	Choose ${\bf y} \in \RR^n$\;
}
\end{algorithm}
It is important to remark that, when $h=n$, where $n$ is the dimension of the problem, the sequence  %$\{\widetilde{\b e}_{h}({\bf x}_k)\}_h$
generated by Algorithm~\ref{alg:restarted_method} converges quadratically to $\mathbf x$, under suitable regularity assumptions \cite{le1992quadratic} on $\mathcal F$.

Since all the algorithms we took into consideration in Section \ref{sec:PM} are based on a fixed-point iteration, for all of them we consider the restarted extrapolation method described in Algorithm~\ref{alg:restarted_method}. In our case, $\mathcal{F}$ is either the  iteration map of Algorithm \ref{alg:shiftedPM1} or that of Algorithm \ref{alg:shiftedPM2}.
Concerning the computational complexity and the storage requirements, in our experimental investigations, we used the public available  software EPSfun Matlab toolbox, na44 package in Netlib \cite{bmrz2017software} that contains optimized versions of the STEA algorithms. The $2h+1$ vectors $\b x_0, \dots, \b x_{2h}$ are computed together with the extrapolation scheme, and thus only  $h+2$ vectors of dimension $n$ have to be stored in order to  compute $\widetilde{\b e}_{h}({\bf x}_0)$ (see \cite{brezinski2017simplified,bmrz2017software} for implementation details).
Observe that, in addition to applying the power method map $\mathcal F$, each outer cycle also requires to compute $2h+1$ scalar products.

The practical implementation and the performance of the outlined  method rely on two key parameter choices: the choice of $h$ and the choice of $\b y \in \RR^n$. As described above, the choice of $h$  is connected to the memory requirement and influences the quality of the speed-up performance. In general, also for relatively small problems, we choose a value
of $h$ smaller than $n$ (the dimension  of the problem). %, as suggested in \cite{brezinski2014simplified,brezinski2017simplified}.
This is the case, for instance, of the real-world  examples of Section \ref{sec:experiments_extrapolated}.
Concerning the choice of $\b y \in \RR^n$, this is a well-known critical point in the topological Shanks' transformations and it is usually addressed by model-dependent heuristics. In fact, no general theoretical result has been obtained so far concerning the  selection of an optimal $\b y \in \RR^n$.
In our examples, in each cycle we take a different $\b y$,  chosen as $\b y = \widetilde{\b e}_{h}({\bf x}_0)$, i.e., the last extrapolated term, in the first cycle apart where is set equal to ${\bf x}_0$. The quality of this choice is supported by the  performance we obtained and by the fact that in all our tests the resulting extrapolated vectors computed with such a choice of $\b y$ belong to the desired cone~$C_+(\b T)$.
\begin{figure}[p]
	\centering
	\includegraphics[width=.32\textwidth]{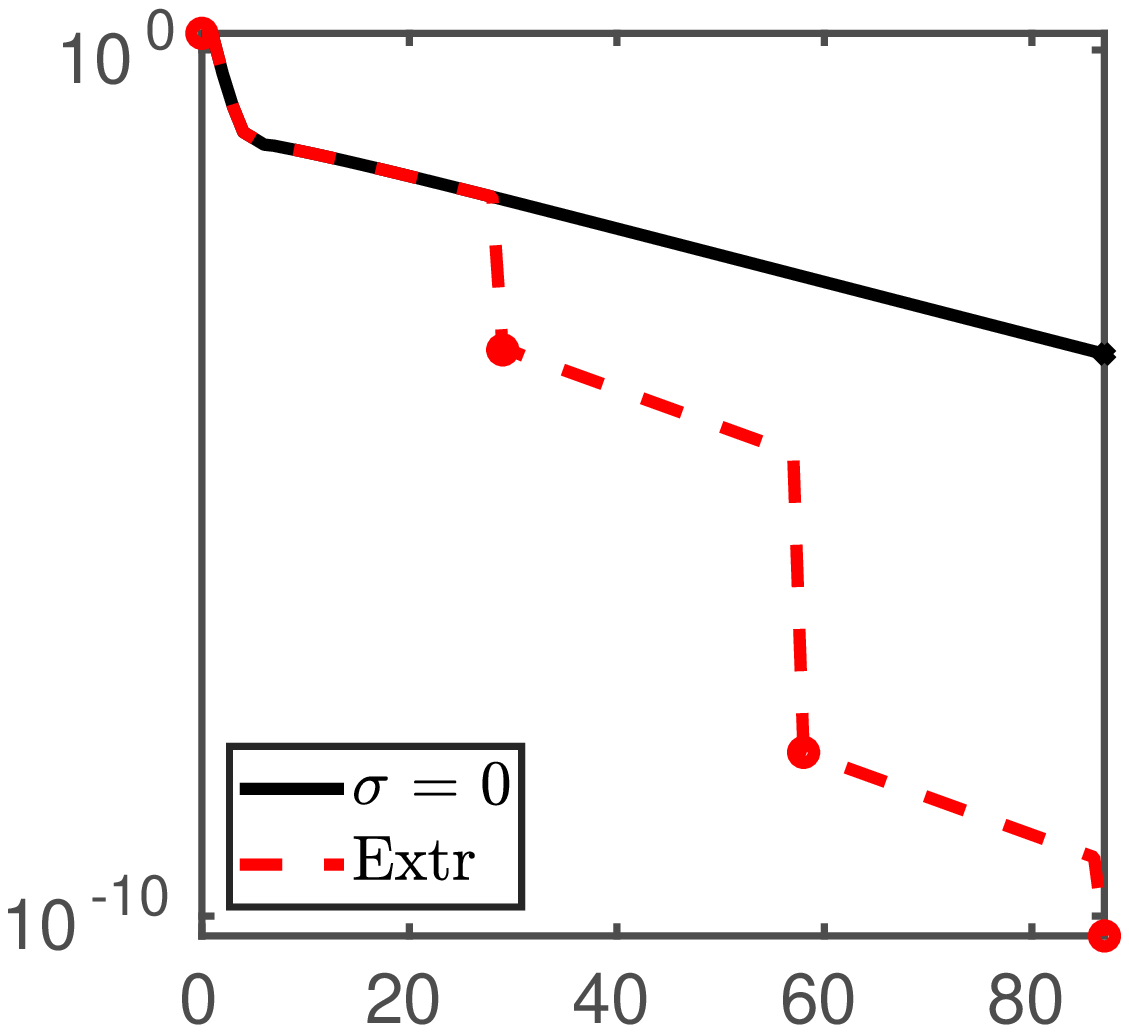}
	\includegraphics[width=.32\textwidth]{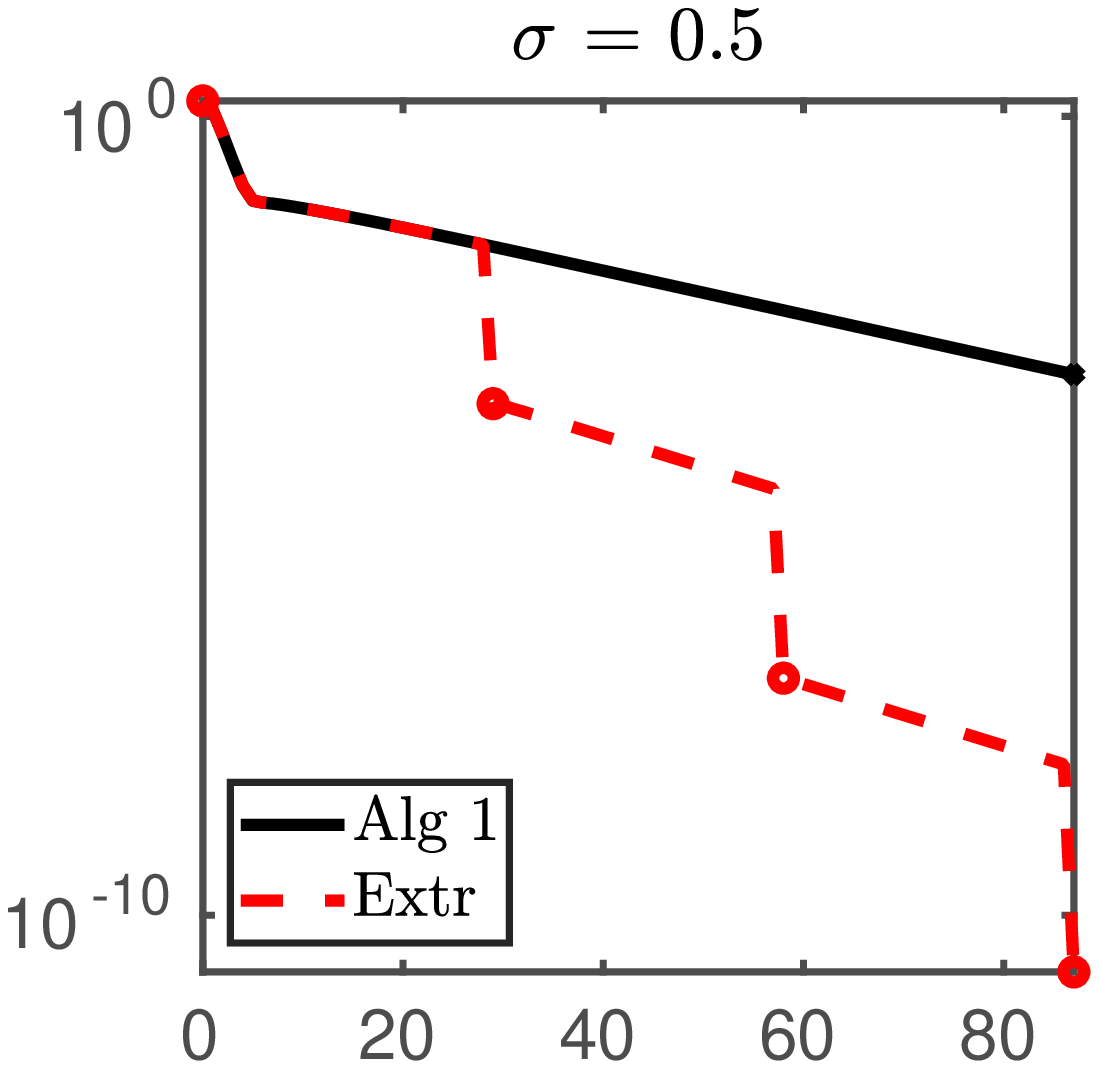}
	\includegraphics[width=.32\textwidth]{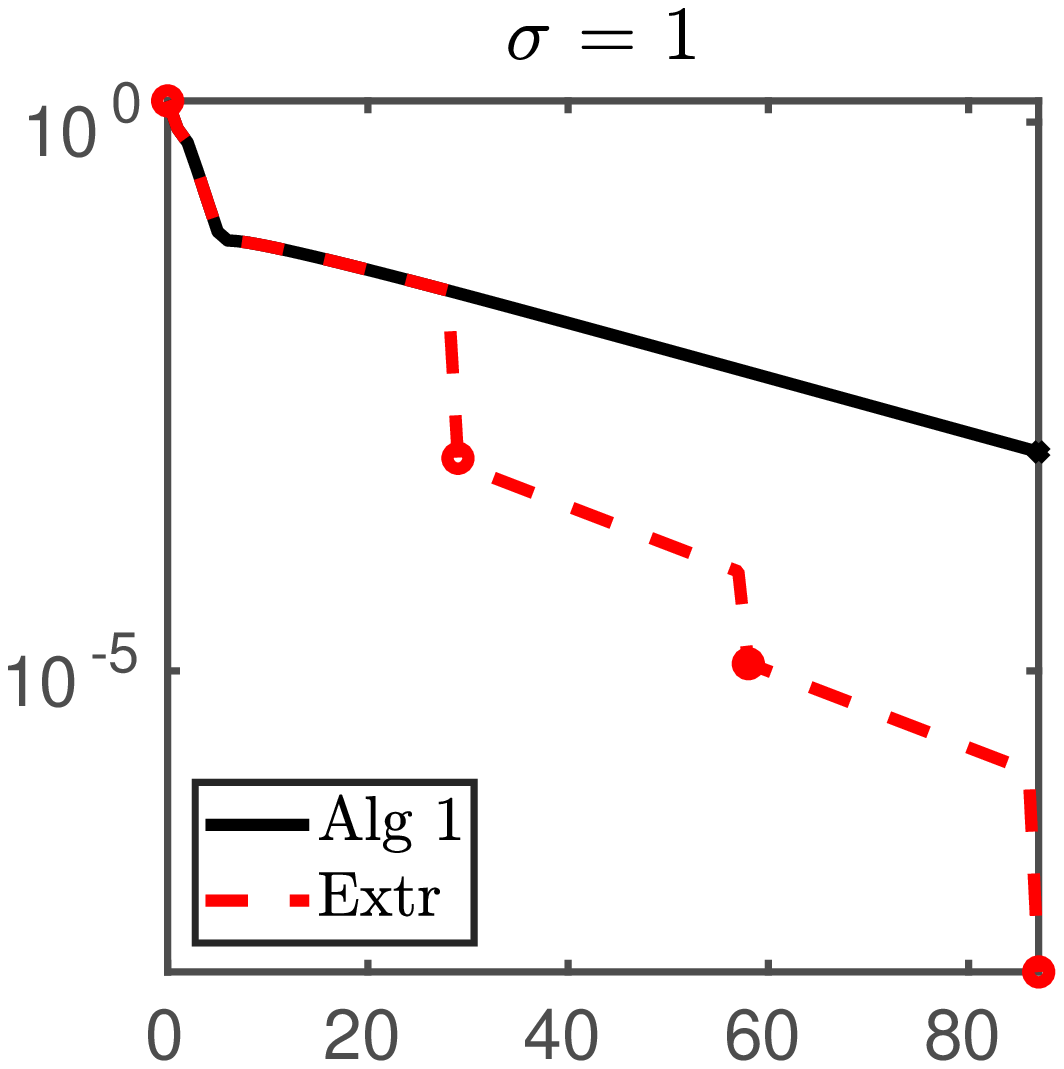}\\
	\includegraphics[width=.32\textwidth]{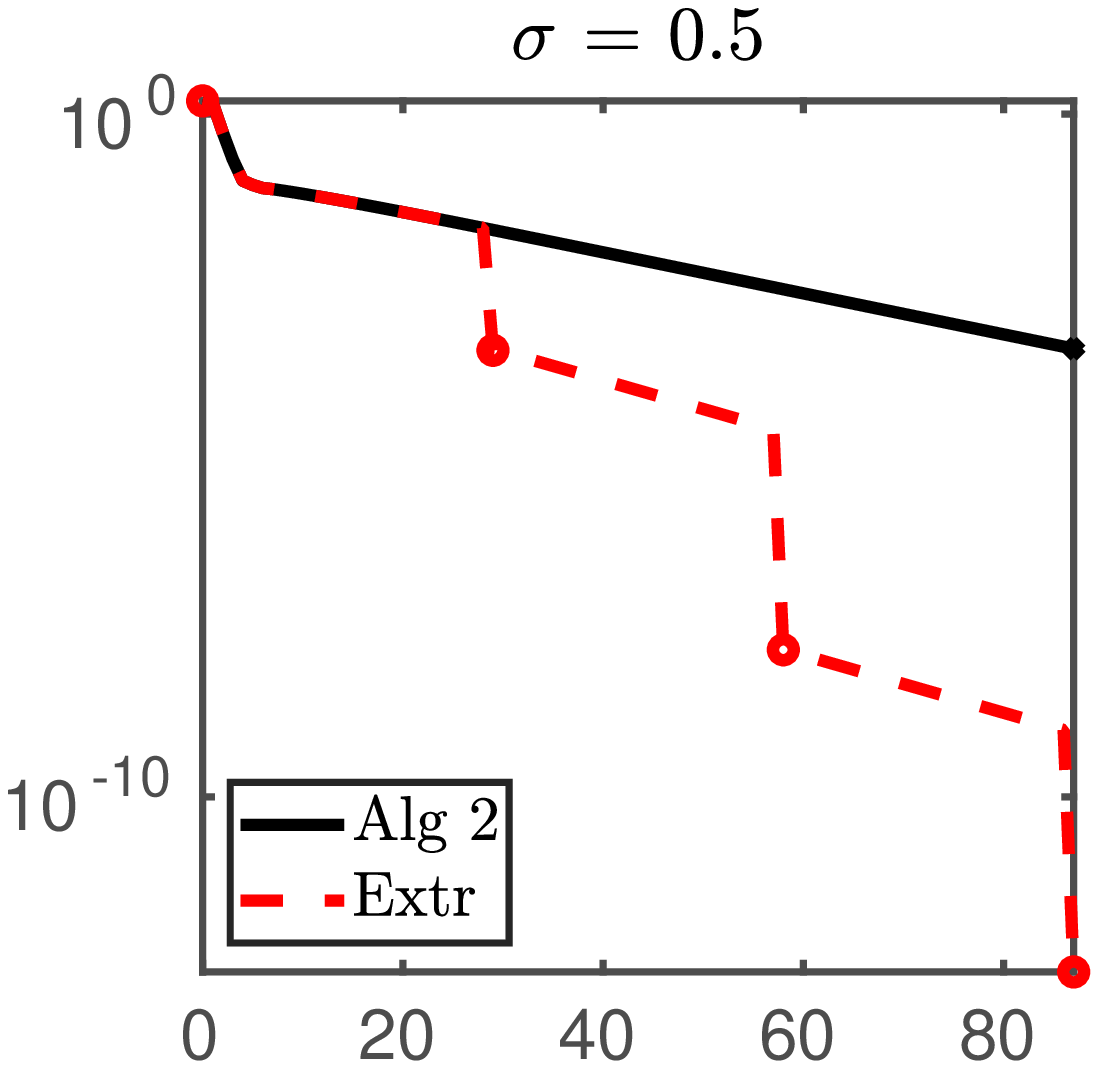}
	\includegraphics[width=.32\textwidth]{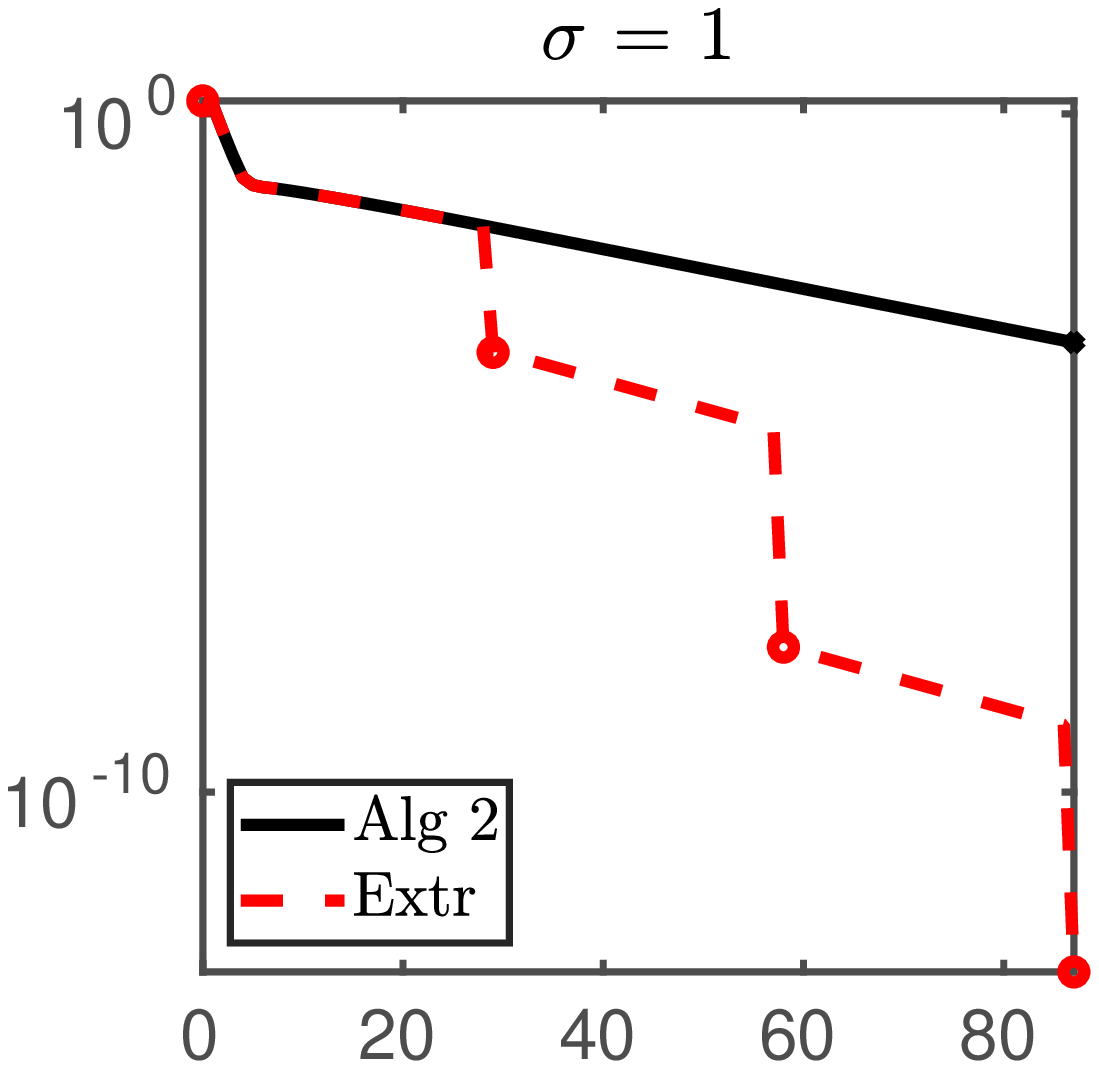}
	%\vspace{-0.3cm}
	\caption{STEA2 on $\b T\in\RR^{n\times n\times n}, \; n=62$, generated using \texttt{dolphins}.  $2h=28$, $\mathsf{cycles}=3$, $p=d+10^{-5}$.}\label{fig:etcycle1}
%\end{figure}
	
\vspace{0.4cm}
%\begin{figure}[t]
	\centering
	\includegraphics[width=.32\textwidth]{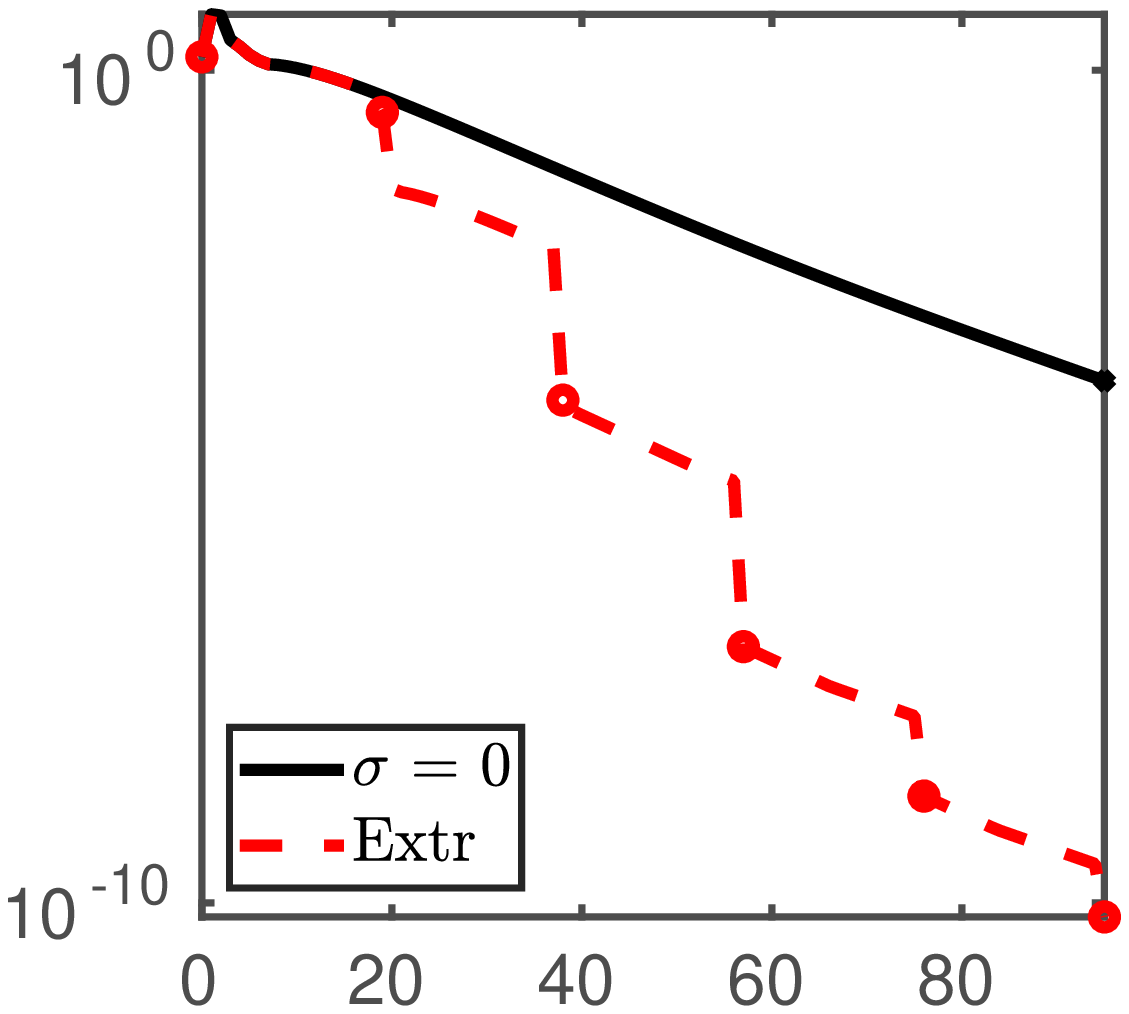}
	\includegraphics[width=.32\textwidth]{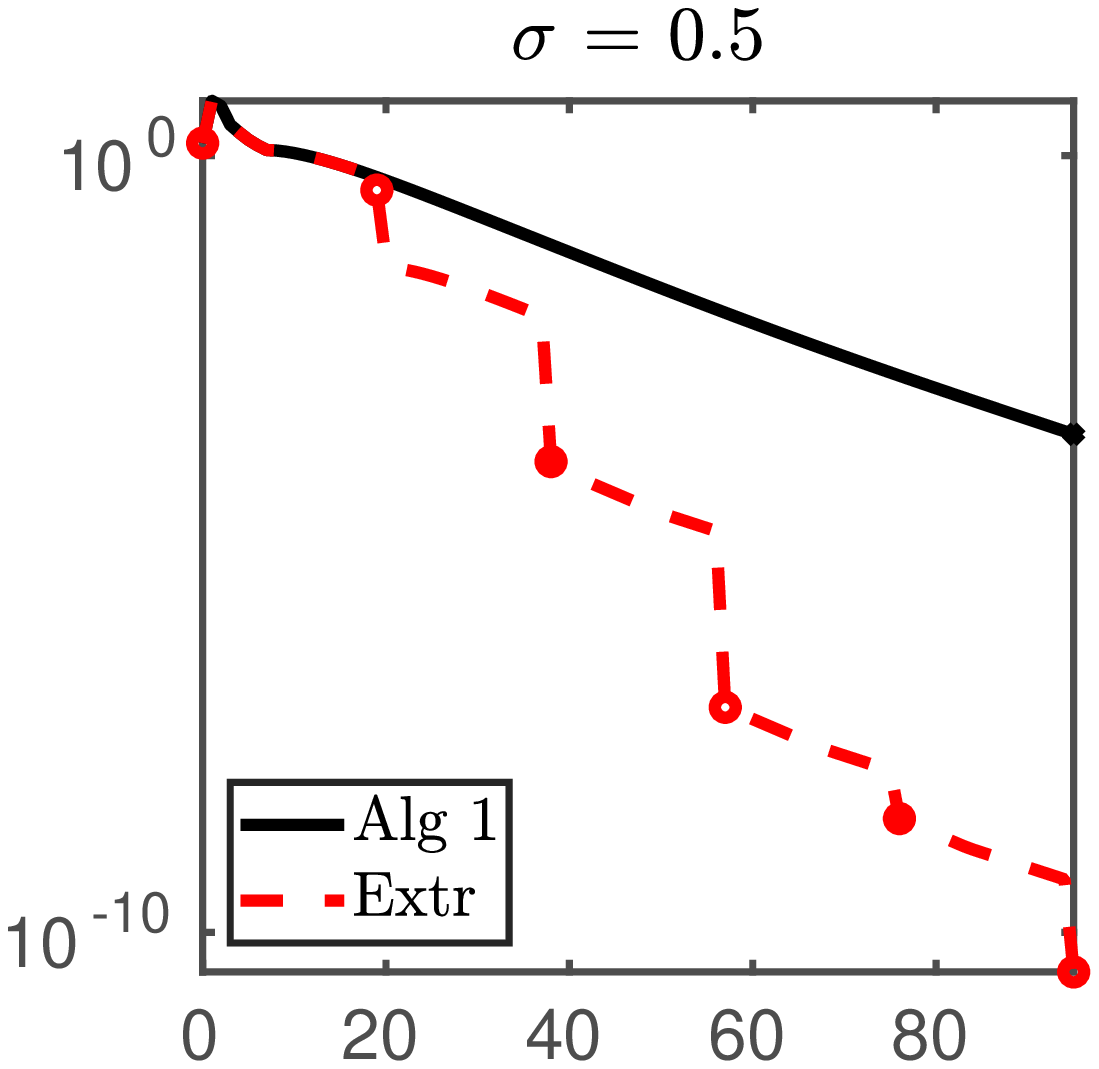}
	\includegraphics[width=.32\textwidth]{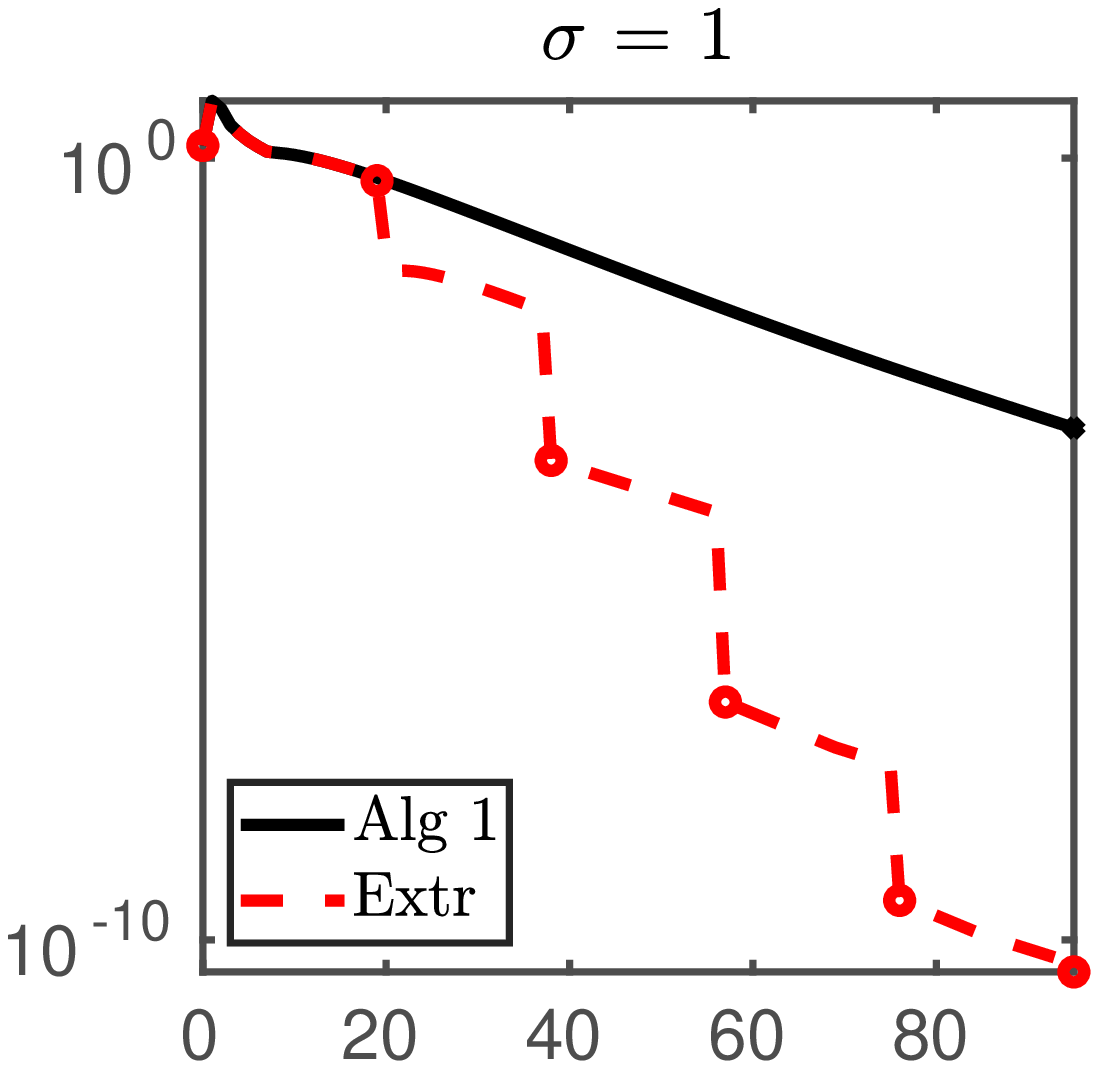}\\
	\includegraphics[width=.32\textwidth]{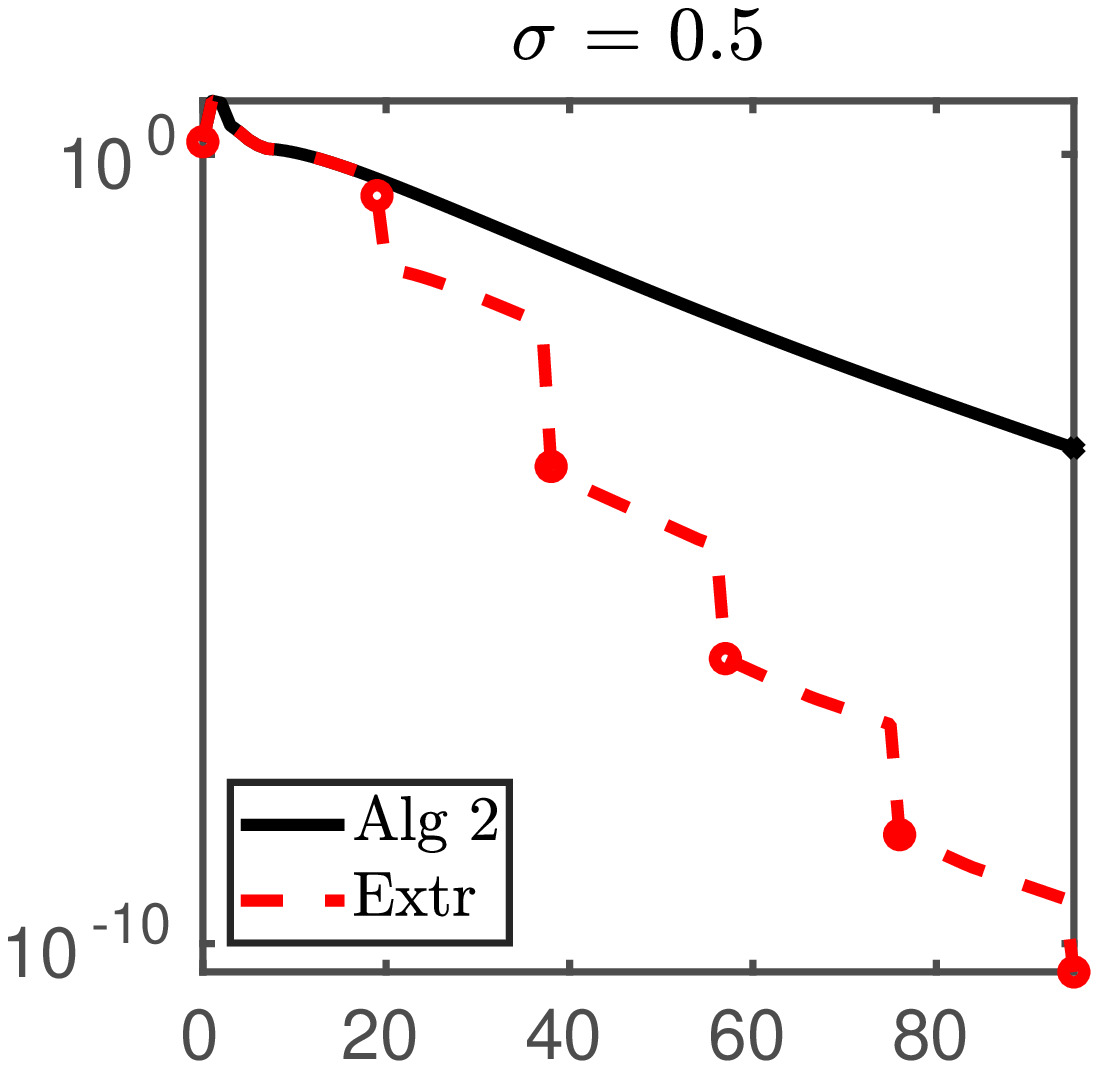}
	\includegraphics[width=.32\textwidth]{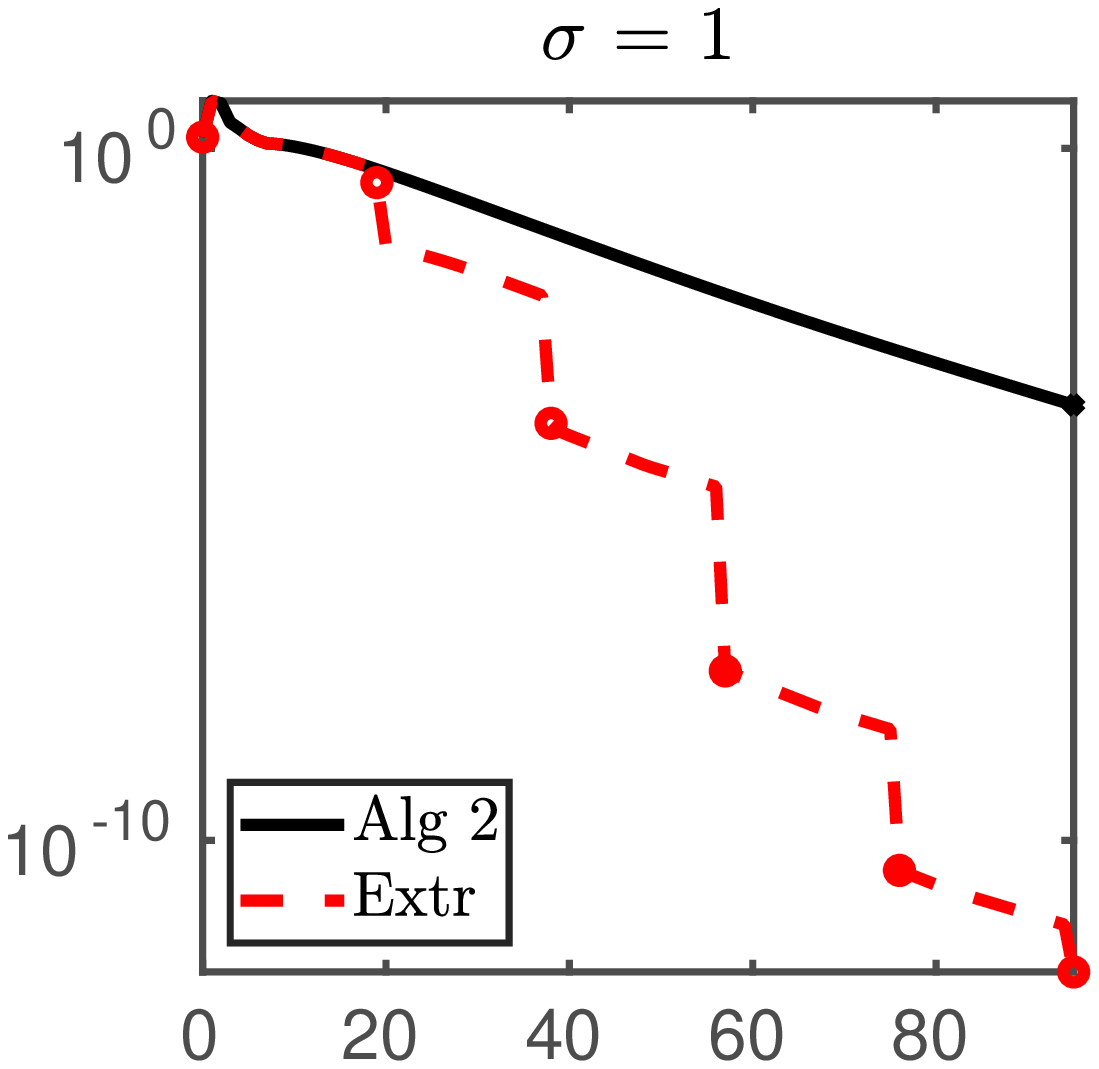}
	%\vspace{-0.3cm}
	\caption{STEA2 on $\b T\in\RR^{n\times n\times n}, \; n=2361$, generated using \texttt{yeast}. $2h=12$, $\mathsf{cycles}=6$, $p=d+10^{-5}$.}\label{fig:etcycle2}
\end{figure}

\begin{figure}[p]
	\centering
	\includegraphics[width=.32\textwidth]{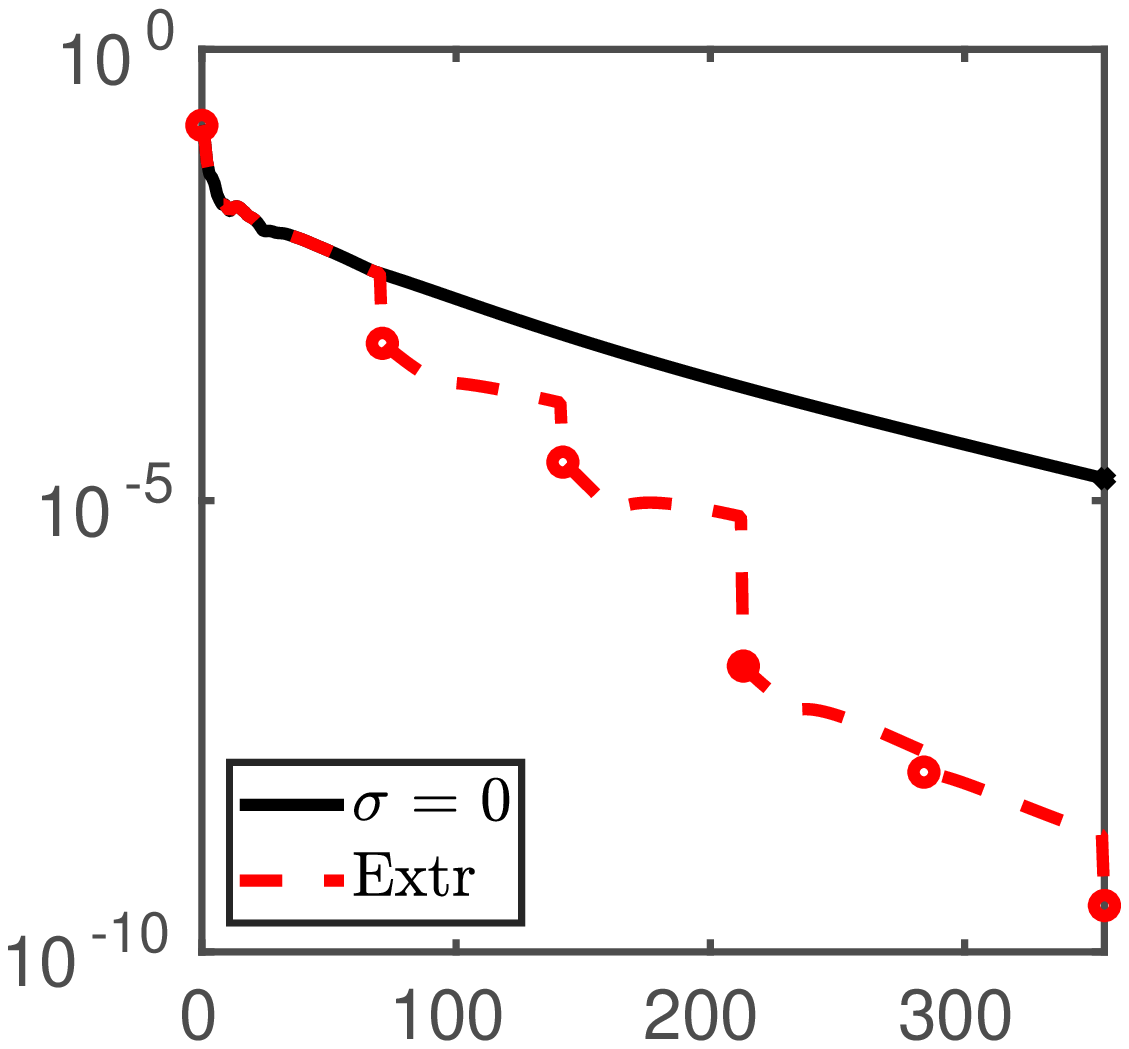}
	\includegraphics[width=.32\textwidth]{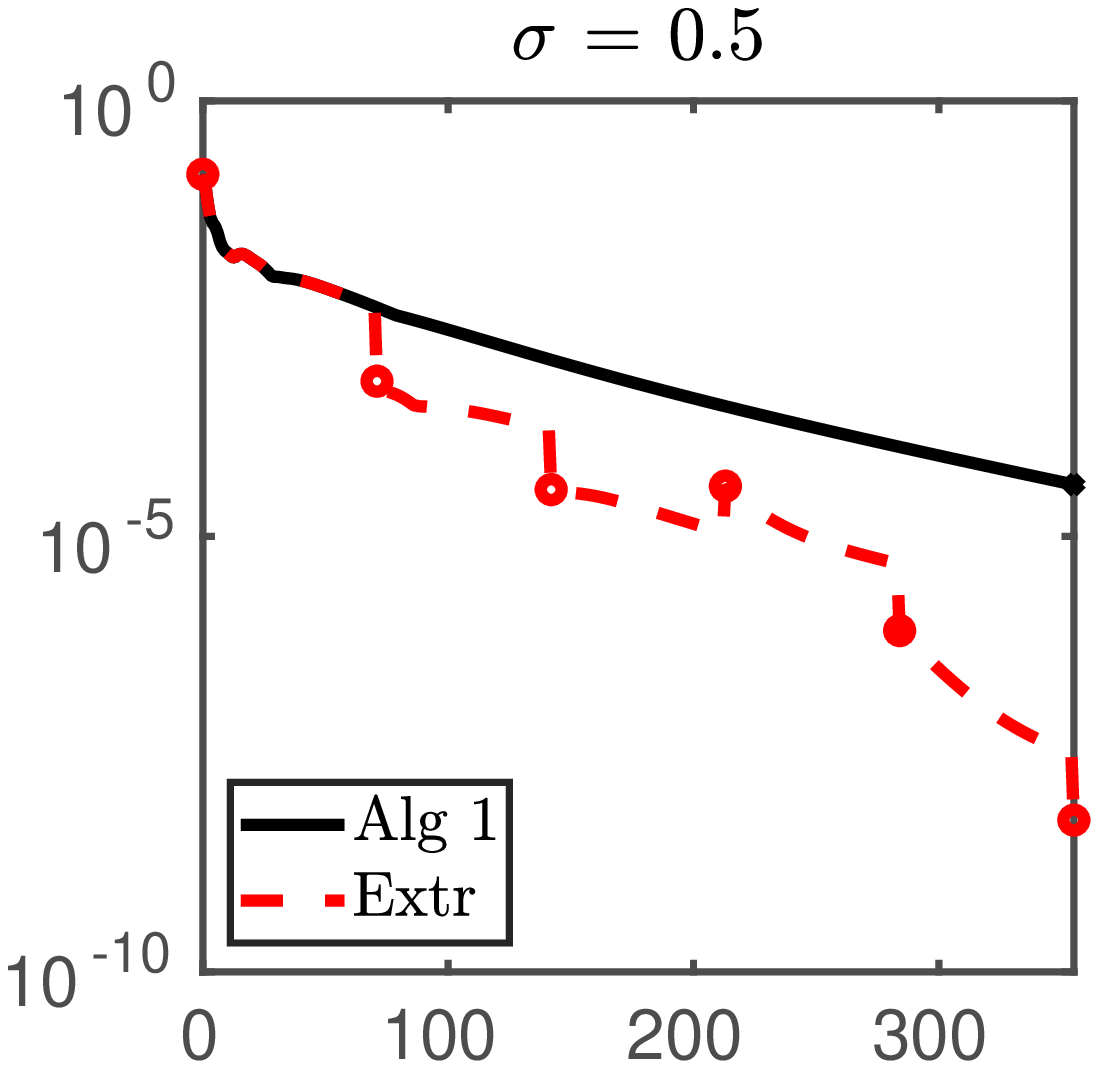}
	\includegraphics[width=.32\textwidth]{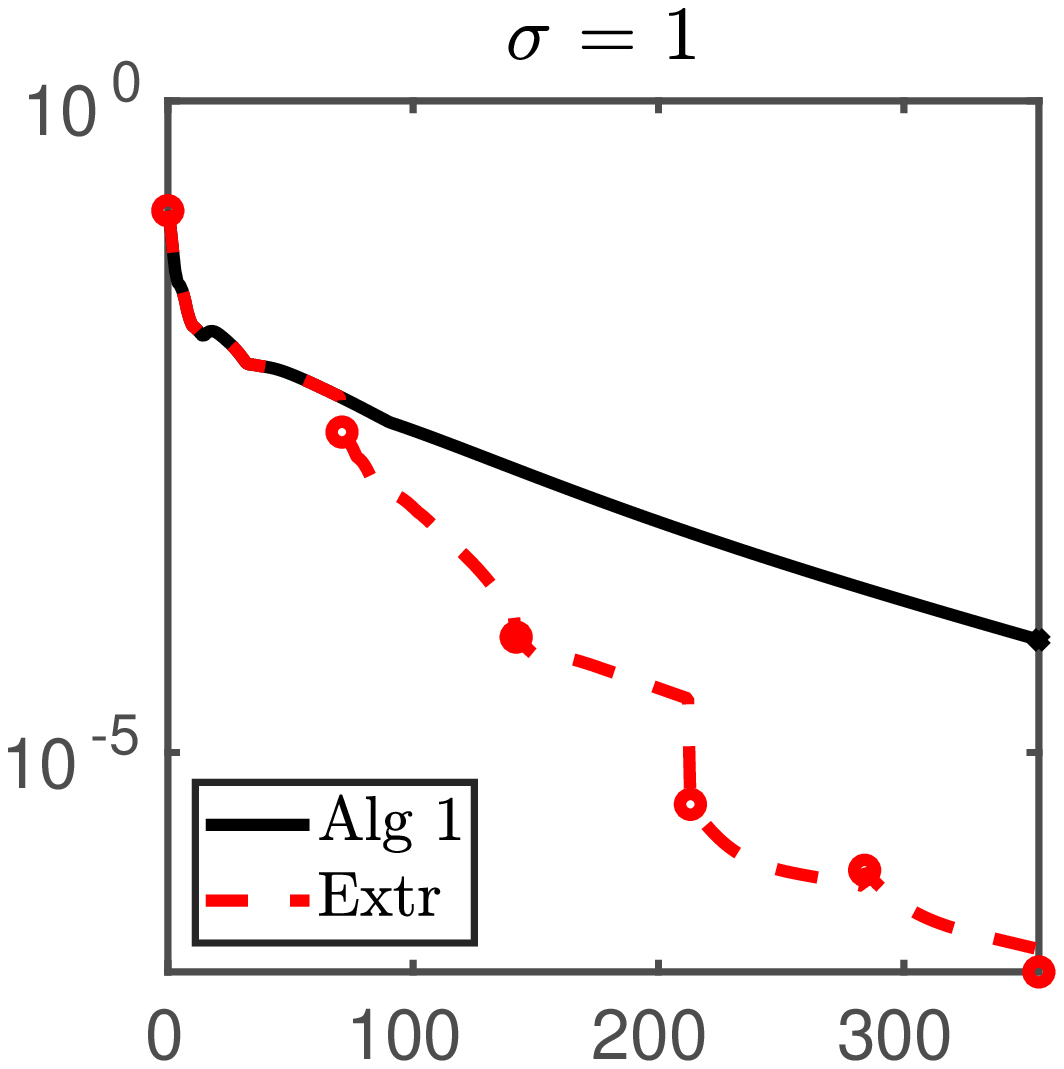}\\
	\includegraphics[width=.32\textwidth]{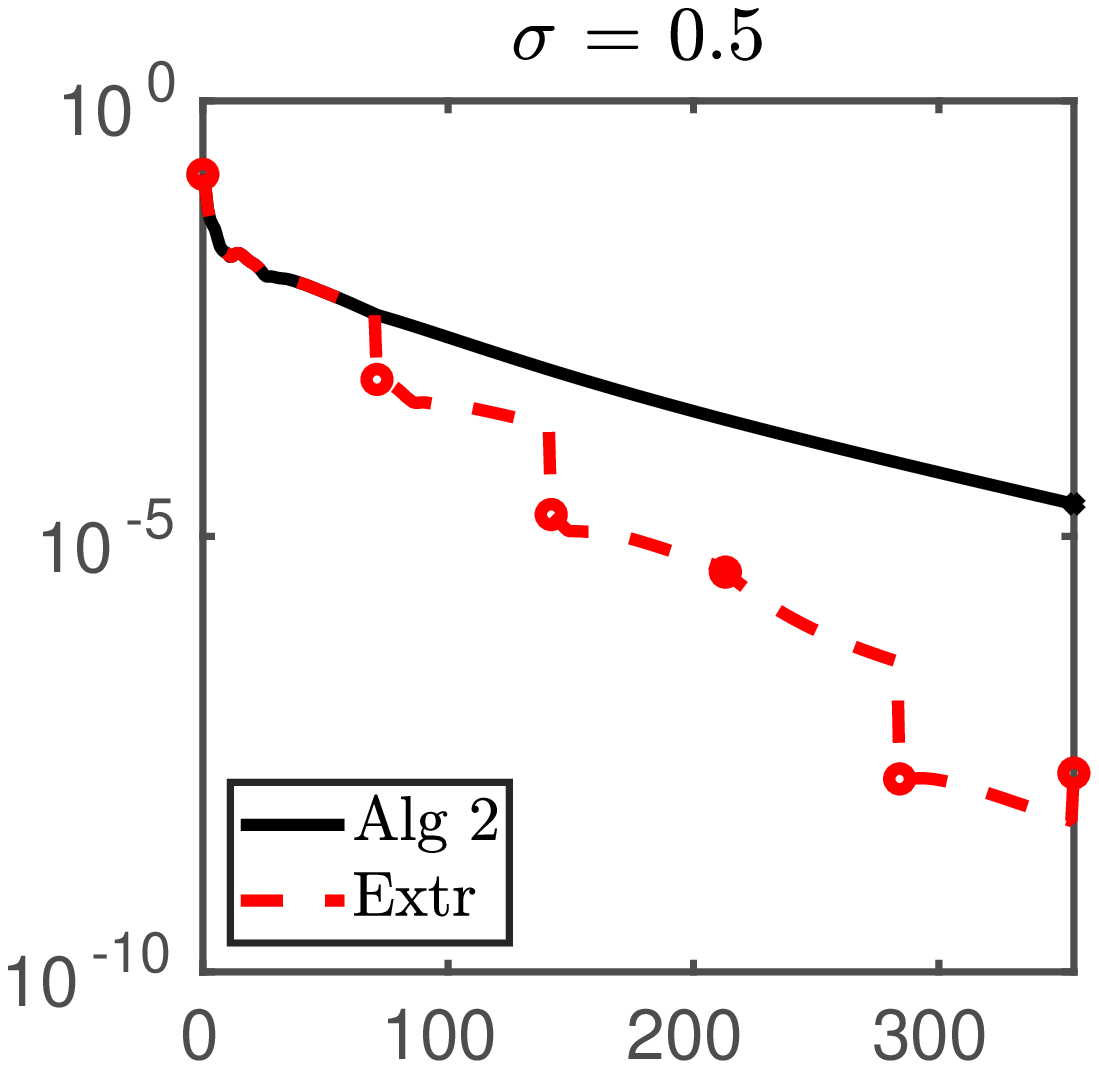}
	\includegraphics[width=.32\textwidth]{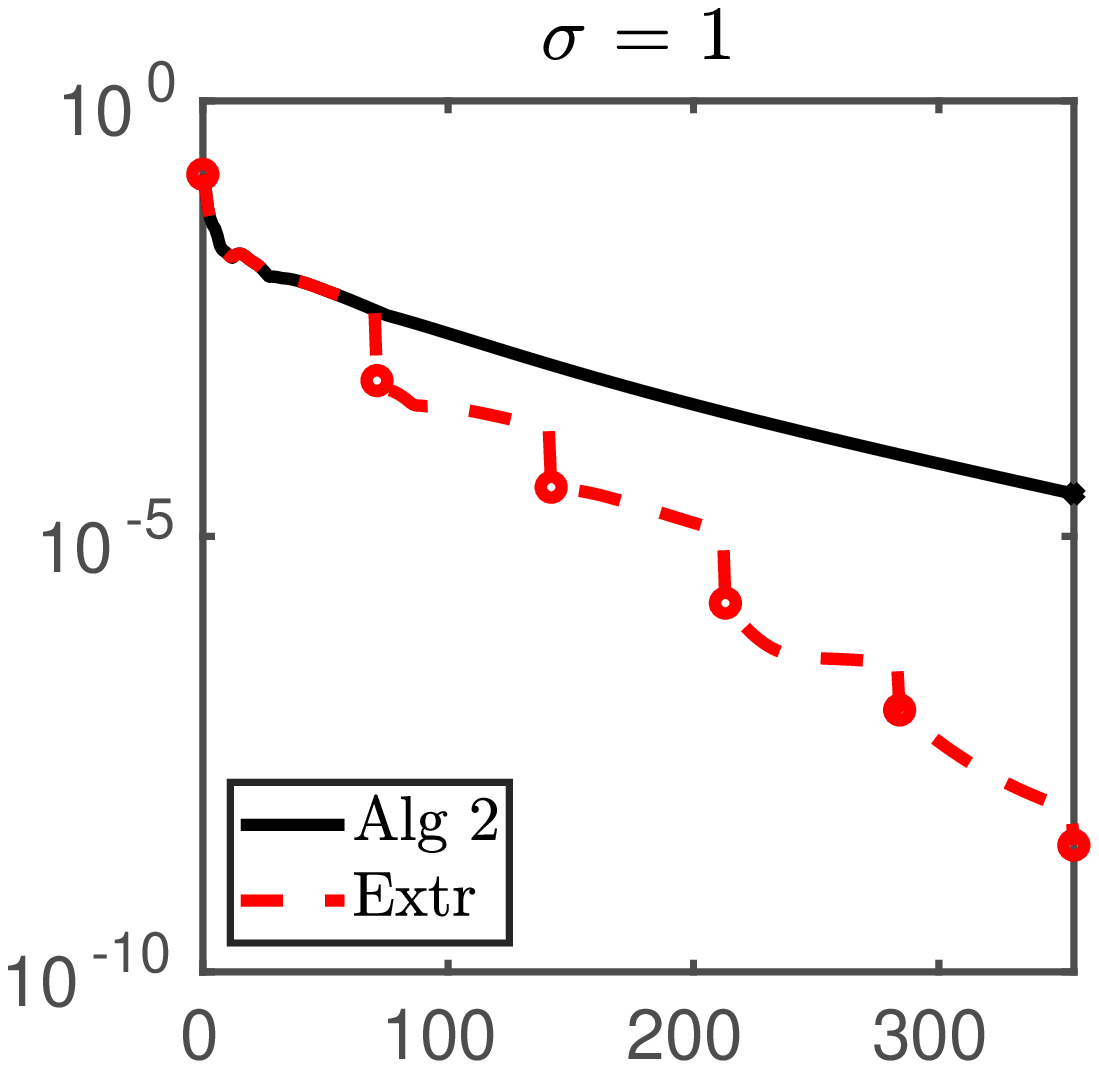}
	%\vspace{-0.3cm}
	\caption{STEA2 on $\b T\in\RR^{n\times n\times n}, \; n=1107$, generated using \texttt{gre$1107$}. $2h=70$, $\mathsf{cycles}=5$, $p=d+10^{-5}$.}\label{fig:etcycle3}
%\end{figure}

\vspace{0.4cm}
%\begin{figure}[t]
\centering
\includegraphics[width=.32\textwidth]{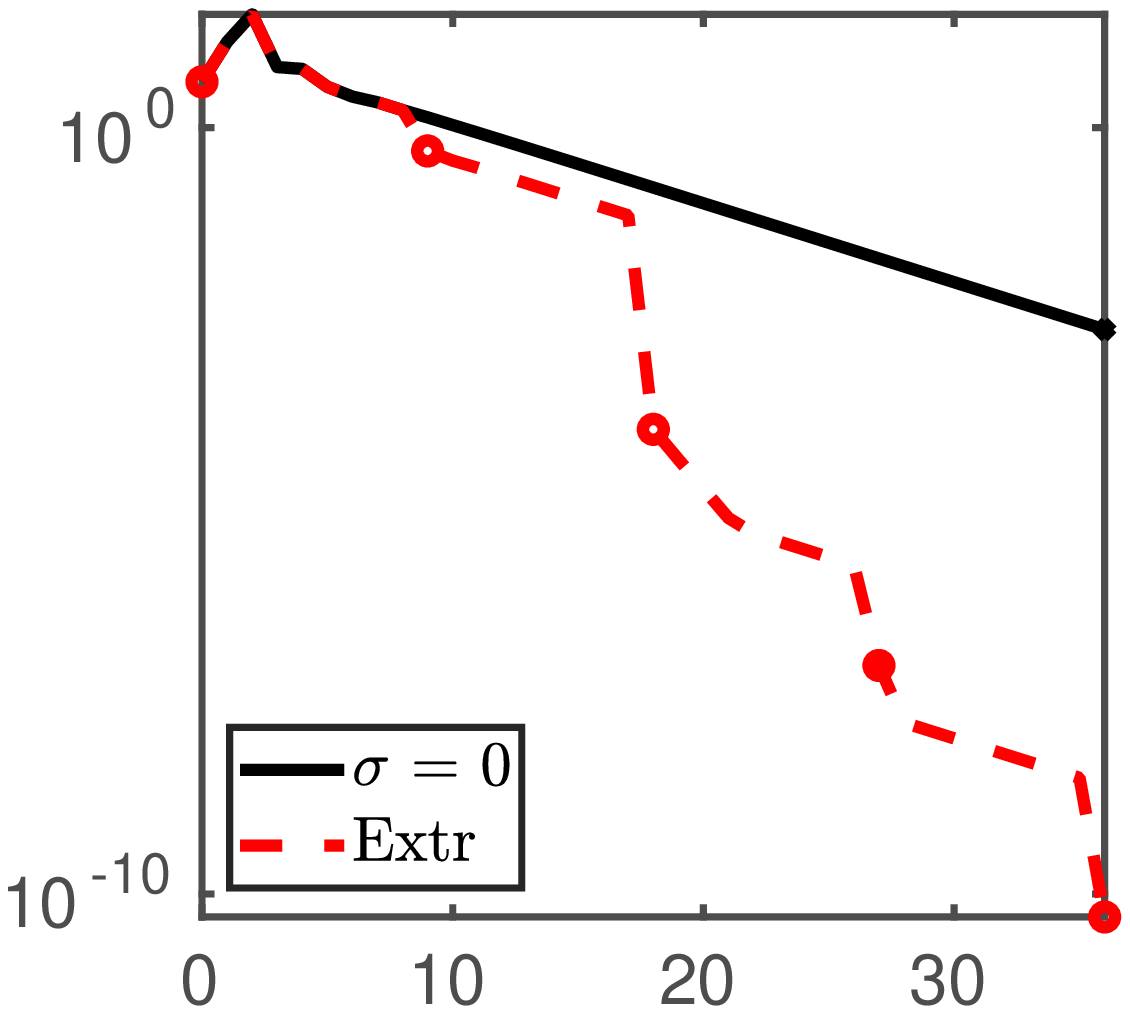}
\includegraphics[width=.32\textwidth]{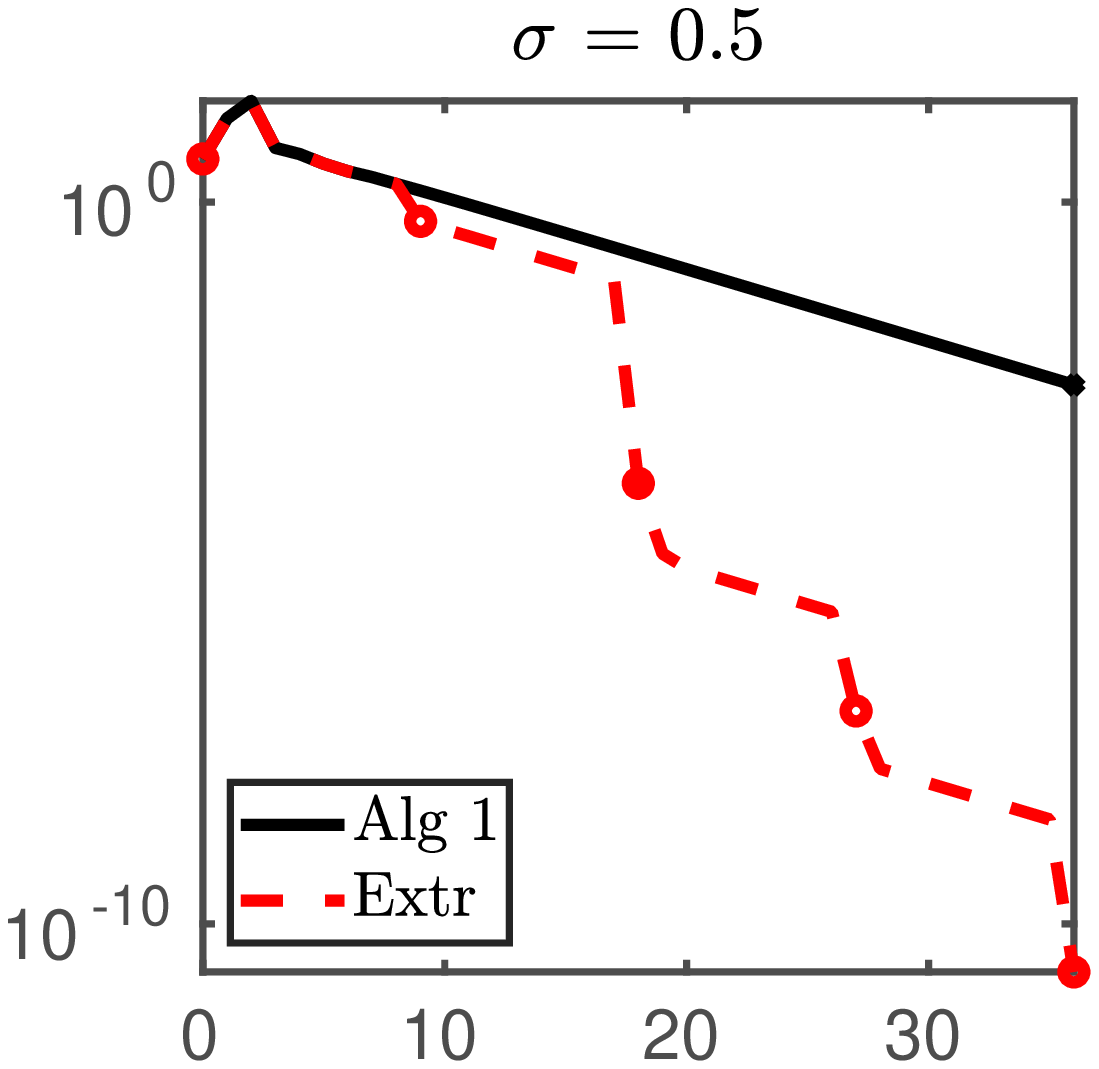}
\includegraphics[width=.32\textwidth]{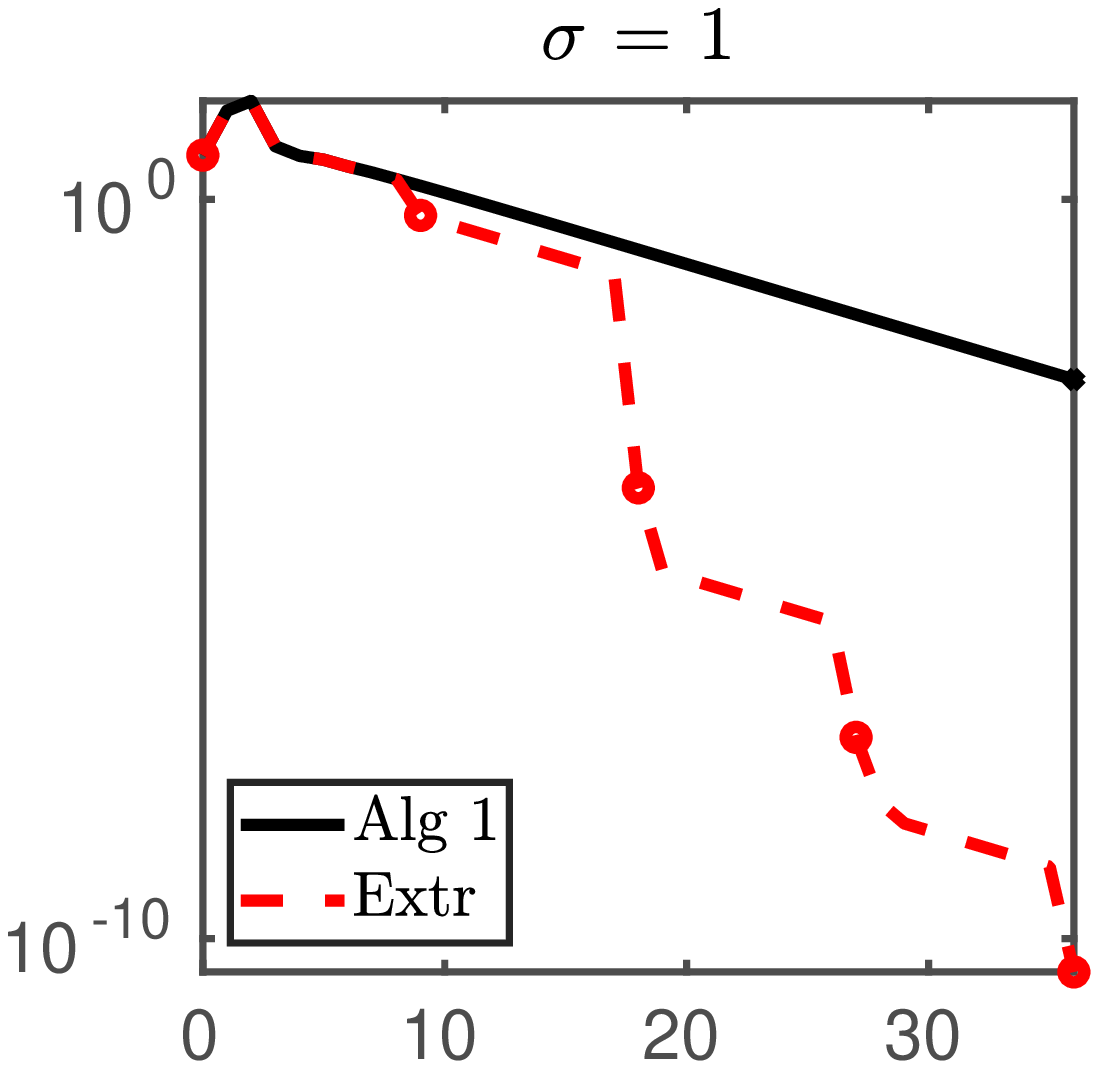}\\
\includegraphics[width=.32\textwidth]{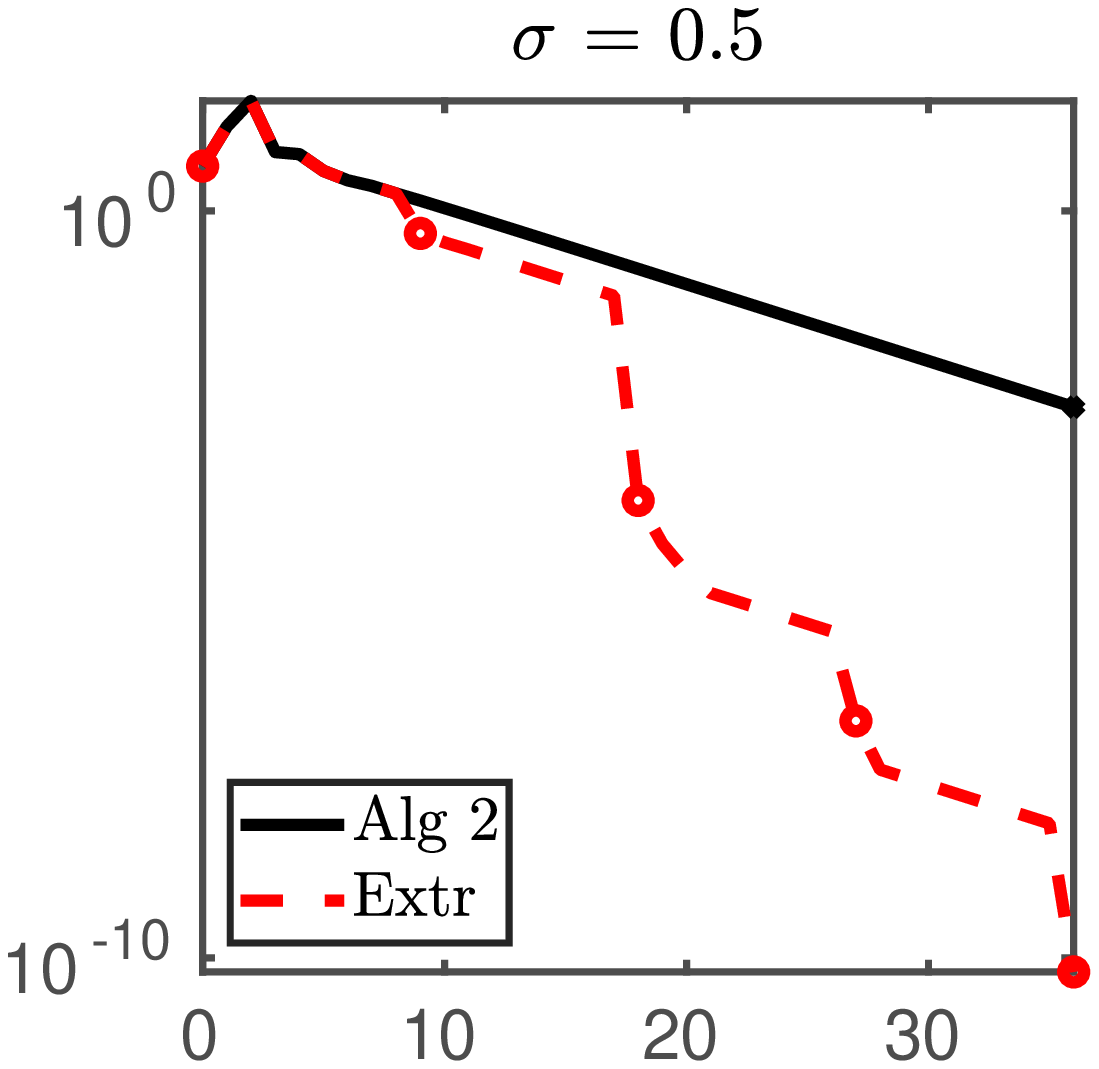}
\includegraphics[width=.32\textwidth]{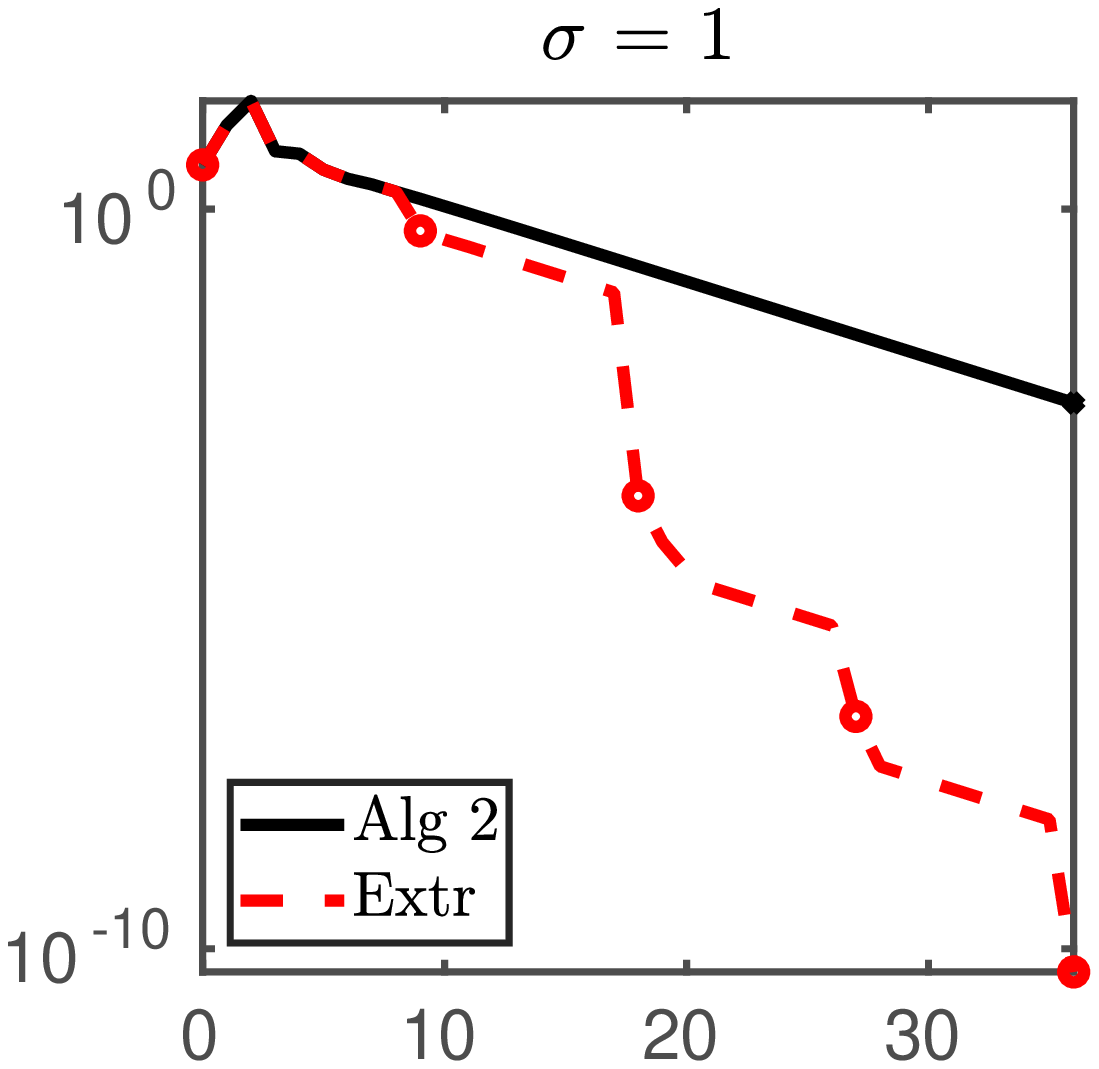}
%\vspace{-0.3cm}
\caption{STEA2 on $\b T\in\RR^{n\times n\times n}, \; n=9914$, generated using \texttt{stanford}. $2h=8$, $\mathsf{cycles}=4$, $p=d+10^{-5}$.}\label{fig:etcycle4}
\end{figure}

\section{Numerical Experiments: Part 2}\label{sec:experiments_extrapolated}

In this section we present several numerical experiments to demonstrate the advantages of the extrapolation framework we are proposing. We focus only on problems from Section \ref{sec:real_world_data_tensors} which exhibit a \textit{slow} convergence rate (even if we observe similar acceleration performance on all the datasets).  For the same reason we focus only on the case $p=d+10^{-5}$. As before, in all our experiments we consider the point-wise eigenvector residual
\begin{equation*}
\|T(\b x_k)-\lambda_k\Phi_p(\b x_k)\|_\infty
\end{equation*}
evaluated on the current iteration step.
Our numerical results are focused on the analysis of the rate of convergence for the accelerated sequence when compared to the original one. To this end, we  run Algorithm \ref{alg:restarted_method} for  a prescribed number of inner and outer iterations (i.e. we fix the value of $h$ and $\mathsf{cycles}$) without any other stopping criterion.
Results are shown in Figures \ref{fig:etcycle1}, \ref{fig:etcycle2}, \ref{fig:etcycle3} and \ref{fig:etcycle4},
where  we highlighted with a circle each restart of the outer loop, i.e., the residual generated by the  iterate defined at Line 8 of Algorithm \ref{alg:restarted_method}.

The linear functional $\b y$ is updated at the end of each outer cycle by choosing $\b y=\widetilde{\b e}_{h}({\bf x}_0)$ (for the first extrapolation step we choose $\b y = \b x_0$). For the implementation of the simplified topological $\varepsilon$-algorithm we used the public domain Matlab toolbox EPSfun \cite{bmrz2017software}. For all the reported numerical results  we choose a random starting point.

As Figures \ref{fig:etcycle1}, \ref{fig:etcycle2}, \ref{fig:etcycle3} and \ref{fig:etcycle4} show, the introduction of extrapolation techniques for the computation of $\ell^p$ eigenpairs greatly improves  the rate of convergence of the original fixed point Algorithms \ref{alg:shiftedPM1} and \ref{alg:shiftedPM2} at a cost of one scalar product more per fixed point iteration (see Section \ref{sec:extrapolation}).

Finally, let us point out that even though the acceleration phenomena are quite evident for the shifted versions of Algorithms \ref{alg:shiftedPM1} and \ref{alg:shiftedPM2}, they seem to happen  more systematically on the unshifted ones;  we find this aspect particularly interesting and  we believe it deserves further investigation.

{In Table \ref{table:cputimes}, we report the comparison of the execution times needed to produce a residual smaller than $10^{-9}$  for the unshifted extrapolated scheme and the non extrapolated one (the execution times of the other extrapolated schemes are all very similar).
\begin{table}[t]
	\centering  \fontsize{9}{5}\selectfont
	\begin{tabular}{|l|c|c|}
		\hline
		\multirow{2}{*}{Problem} & \multicolumn{2}{l|}{	\hspace{1.5cm}{Time(s)}}   \\ \cline{2-3}
		& \textit{Extrapolated} & \textit{Not Extrapolated} \\ \hline
	\texttt{Dolphins}	& \textbf{0.2373}       & 0.3062          \\ \hline
		\texttt{yeast} 	& \textbf{0.4199}       & 0.7193          \\ \hline
	\texttt{gre$1107$}	& \textbf{1.4225}       & 1.8462          \\ \hline
	\texttt{wb-cs-stanford}	& \textbf{0.3817}       & 0.6994          \\ \hline
	\end{tabular}
\caption{Comparison of execution times to produce residuals smaller than $10^{-9}$, $\sigma=0$.}
\label{table:cputimes}
\end{table}
}

\section{Conclusions}In this work we introduced two new shifted power method  schemes for computing $\ell^p$-eigenpairs
of tensors. We proved that the  introduced algorithms are guaranteed to converge for entrywise nonnegative tensors {with possibly reducible patterns} and for $p>d$,  being $d$ is the number of modes of the tensor.
Moreover, we show that for nonnegative tensors the Perron $\ell^p$-eigenvector depends continuously on the parameter $p$. This result, together with the global convergence guarantees
of the shifted power methods, allows us to propose the first method that can provably approximate the positive Perron $H$-eigenvector of a nonnegative tensor, by choosing $p\approx d$.  The methods can
suffer of a slow rate of convergence, as we observe in the numerical experiments proposed in
Section \ref{sec:experiments_part1}. For this reason, we also discuss  the employment of the
simplified $\varepsilon$-algorithm to extrapolate the power sequence and accelerate its convergence. %overcome this issue.
The numerical experiments of Section \ref{sec:experiments_extrapolated} show that the use of the proposed extrapolation method % for the computation of the $\ell^p$ eigenpairs using fixed-point iterations
substantially improves  the convergence rate of the proposed power methods for $\ell^p$-eigenvectors, at the price of a minor additional cost per step.

\section*{Acknowledgement}
The work of M.R.-Z. was partially supported  by the
University of Padua, Project no. DOR1903575/19. The work of S.C. was partially supported by the GNCS -- INdAM project ``Efficient Methods for large scale problems with applications to data analysis and preconditioning''. The work F.T. was funded by the European Union's Horizon 2020 research and innovation programme under the Marie Sk\l odowska-Curie individual fellowship ``MAGNET'' No 744014. All the authors are members of the INdAM Research group GNCS.


\begin{thebibliography}{10}
	
	\bibitem{arrigo2019multi}
	F.~Arrigo and F.~Tudisco.
	\newblock Multi-dimensional, multilayer, nonlinear and dynamic {HITS}.
	\newblock In {\em Proceedings of the 2019 SIAM International Conference on Data
		Mining}, pages 369--377. SIAM, 2019.
	
	\bibitem{benson2019three}
	A.~R. Benson.
	\newblock Three hypergraph eigenvector centralities.
	\newblock {\em SIAM J. Math. Data Sci.}, 1(2):293--312, 2019.
	
	\bibitem{benson2016higher}
	A.~R. Benson, D.~F. Gleich, and J.~Leskovec.
	\newblock Higher-order organization of complex networks.
	\newblock {\em Science}, 353(6295):163--166, 2016.
	
	\bibitem{bouhamidi2011extrapolated}
	A.~Bouhamidi, K.~Jbilou, L.~Reichel, and H.~Sadok.
	\newblock An extrapolated {TSVD} method for linear discrete ill-posed problems
	with {K}ronecker structure.
	\newblock {\em Linear Algebra Appl.}, 434(7):1677--1688, 2011.
	
	\bibitem{brezinski1975generalisations}
	C.~Brezinski.
	\newblock G{\'e}n{\'e}ralisations de la transformation de {S}hanks, de la table
	de {P}ad{\'e} et de l'$\varepsilon$-algorithme.
	\newblock {\em Calcolo}, 12(4):317--360, 1975.
	
	\bibitem{brezinski2013extrapolationbook}
	C.~Brezinski and M.~Redivo-Zaglia.
	\newblock {\em Extrapolation methods. {T}heory and {P}ractice}.
	\newblock North-Holland, Amsterdam, 1991.
	
	\bibitem{brezinski2014simplified}
	C.~Brezinski and M.~Redivo-Zaglia.
	\newblock The simplified topological $\varepsilon$-algorithms for accelerating
	sequences in a vector space.
	\newblock {\em SIAM J. Sci. Comput.}, 36(5):A2227--A2247, 2014.
	
	\bibitem{bmrz2017software}
	C.~Brezinski and M.~Redivo-Zaglia.
	\newblock {EPS}fun: a {M}atlab toolbox for the simplified topological
	$\varepsilon$-algorithm.
	\newblock {\em Netlib (http://www.netlib.org/numeralgo/)}, na44package, 2017.
	
	\bibitem{brezinski2017simplified}
	C.~Brezinski and M.~Redivo-Zaglia.
	\newblock The simplified topological $\varepsilon$-algorithms: software and
	applications.
	\newblock {\em Numer. Algorithms}, 74(4):1237--1260, 2017.
	
	\bibitem{brezinski2019genesis}
	C.~Brezinski and M.~Redivo-Zaglia.
	\newblock The genesis and early developments of {A}itken{'}s process,
	{S}hank{'}s transformation, the $\varepsilon$--algorithm, and related fixed
	point methods.
	\newblock {\em Numer. Algorithms}, 80(1):11--133, 2019.
	
	\bibitem{brezinski2018shanks}
	C.~Brezinski, M.~Redivo-Zaglia, and Y.~Saad.
	\newblock {S}hanks sequence transformations and {A}nderson acceleration.
	\newblock {\em SIAM Rev.}, 60(3):646--669, 2018.
	
	\bibitem{cartwright2013number}
	D.~Cartwright and B.~Sturmfels.
	\newblock The number of eigenvalues of a tensor.
	\newblock {\em Linear Algebra Appl.}, 438(2):942--952, 2013.
	
	\bibitem{chang2010singular}
	K.~Chang, L.~Qi, and G.~Zhou.
	\newblock Singular values of a real rectangular tensor.
	\newblock {\em J. Math. Anal. Appl.}, 370:284--294, 2010.
	
	\bibitem{chang2008perron}
	K.-C. Chang, K.~Pearson, and T.~Zhang.
	\newblock {P}erron-{F}robenius theorem for nonnegative tensors.
	\newblock {\em Commun. Math. Sci}, 6:507--520, 2008.
	
	\bibitem{cipolla2019extrapolation}
	S.~Cipolla, M.~Redivo-Zaglia, and F.~Tudisco.
	\newblock Extrapolation methods for fixed-point multilinear pagerank
	computations.
	\newblock {\em Numer. Linear Algebra Appl.}, in revision, 2018.
	
	\bibitem{comon2002tensor}
	P.~Comon.
	\newblock Tensor decompositions.
	\newblock {\em Mathematics in Signal Processing {V}}, pages 1--24, 2002.
	
	\bibitem{comon2008symmetric}
	P.~Comon, G.~Golub, L.-H. Lim, and B.~Mourrain.
	\newblock Symmetric tensors and symmetric tensor rank.
	\newblock {\em SIAM J. Matrix Anal. Appl.}, 30:1254--1279, 2008.
	
	\bibitem{davis2011university}
	T.~A. Davis and Y.~Hu.
	\newblock The {U}niversity of {F}lorida sparse matrix collection.
	\newblock {\em ACM Trans. Math. Software}, 38(1):1, 2011.
	
	\bibitem{de2005tensor}
	W.~F. de~la Vega, M.~Karpinski, R.~Kannan, and S.~Vempala.
	\newblock Tensor decomposition and approximation schemes for constraint
	satisfaction problems.
	\newblock In {\em Proceedings of the thirty-seventh annual ACM symposium on
		Theory of computing}, pages 747--754, 2005.
	
	\bibitem{de2000best}
	L.~De~Lathauwer, B.~De~Moor, and J.~Vandewalle.
	\newblock On the best rank-1 and rank-($r_1, r_2, \dots, r_n$) approximation of
	higher-order tensors.
	\newblock {\em SIAM J. Matrix Anal. Appl.}, 21:1324--1342, 2000.
	
	\bibitem{edelstein1962fixed}
	M.~Edelstein.
	\newblock On fixed and periodic points under contractive mappings.
	\newblock {\em J. Lond. Math. Soc}, 1(1):74--79, 1962.
	
	\bibitem{fasino2019higher}
	D.~Fasino and F.~Tudisco.
	\newblock Ergodicity coefficients for higher-order stochastic processes.
	\newblock {\em arXiv:1907.04841}, 2019.
	
	\bibitem{friedland2013best}
	S.~Friedland.
	\newblock Best rank one approximation of real symmetric tensors can be chosen
	symmetric.
	\newblock {\em Front. Math. China}, 8:19, 2013.
	
	\bibitem{friedland2013perron}
	S.~Friedland, S.~Gaubert, and L.~Han.
	\newblock Perron-{F}robenius theorem for nonnegative multilinear forms and
	extensions.
	\newblock {\em Linear Algebra Appl.}, 438:738--749, 2013.
	
	\bibitem{gautier2016tensor}
	A.~Gautier and M.~Hein.
	\newblock Tensor norm and maximal singular vectors of nonnegative tensors - a
	{P}erron--{F}robenius theorem, a {C}ollatz--{W}ielandt characterization and a
	generalized power method.
	\newblock {\em Linear Algebra Appl.}, 505:313--343, 2016.
	
	\bibitem{gautier2018contractivity}
	A.~Gautier and F.~Tudisco.
	\newblock The contractivity of cone--preserving multilinear mappings.
	\newblock {\em Nonlinearity}, 32:4713, 2019.
	
	\bibitem{gautier2017theorem}
	A.~Gautier, F.~Tudisco, and M.~Hein.
	\newblock The {P}erron-{F}robenius theorem for multi-homogeneous mappings.
	\newblock {\em SIAM J. Matrix Anal. Appl.}, 40:1179--1205, 2019.
	
	\bibitem{gautier2017tensor}
	A.~Gautier, F.~Tudisco, and M.~Hein.
	\newblock A unifying {P}erron-{F}robenius theorem for nonnegative tensors via
	multi-homogeneous maps.
	\newblock {\em SIAM J. Matrix Anal. Appl.}, 40:1206--1231, 2019.
	
	\bibitem{gleich2015multilinear}
	D.~F. Gleich, L.-H. Lim, and Y.~Yu.
	\newblock Multilinear {P}age{R}ank.
	\newblock {\em SIAM J. Matrix Anal. Appl.}, 36:1507--1541, 2015.
	
	\bibitem{hillar2013most}
	C.~J. Hillar and L.-H. Lim.
	\newblock Most tensor problems are {NP}-hard.
	\newblock {\em Journal of the ACM (JACM)}, 60(6):45, 2013.
	
	\bibitem{hu2016computing}
	S.~Hu, L.~Qi, and G.~Zhang.
	\newblock Computing the geometric measure of entanglement of multipartite pure
	states by means of non-negative tensors.
	\newblock {\em Physical Review A}, 93(1):012304, 2016.
	
	\bibitem{kofidis2002best}
	E.~Kofidis and P.~A. Regalia.
	\newblock On the best rank-1 approximation of higher-order supersymmetric
	tensors.
	\newblock {\em SIAM J. Matrix Anal. Appl.}, 23:863--884, 2002.
	
	\bibitem{kolda2009tensor}
	T.~G Kolda and B.~W. Bader.
	\newblock Tensor decompositions and applications.
	\newblock {\em SIAM Rev.}, 51:455--500, 2009.
	
	\bibitem{kolda2011shifted}
	T.~G. Kolda and J.~R. Mayo.
	\newblock Shifted power method for computing tensor eigenpairs.
	\newblock {\em SIAM J. Matrix Anal. Appl.}, 32:1095--1124, 2011.
	
	\bibitem{lang2011eigenvalues}
	J.~Lang and D.~Edmunds.
	\newblock {\em Eigenvalues, embeddings and generalised trigonometric
		functions}.
	\newblock Number 2016 in Lecture Notes in Math. Springer Nature, 2011.
	
	\bibitem{le1992quadratic}
	H.~Le~Ferrand.
	\newblock The quadratic convergence of the topological epsilon algorithm for
	systems of nonlinear equations.
	\newblock {\em Numer. Algorithms}, 3(1):273--283, 1992.
	
	\bibitem{lemmens22birkhoff}
	B.~Lemmens and R.~Nussbaum.
	\newblock {B}irkhoff{'}s version of {H}ilbert{'}s metric and its applications
	in analysis.
	\newblock In {\em {Handbook of Hilbert geometry}}, volume~22 of {\em {IRMA
			Lectures in Mathematics and Theoretical Physics}}, chapter~10, pages
	275--303. {European Mathematical Society}, 2014.
	
	\bibitem{li2014limiting}
	W.~Li and M.~K. Ng.
	\newblock On the limiting probability distribution of a transition probability
	tensor.
	\newblock {\em Linear Multilinear Algebra}, 62(3):362--385, 2014.
	
	\bibitem{li2012har}
	X.~Li, M.~K. Ng, and Y.~Ye.
	\newblock {HAR}: hub, authority and relevance scores in multi-relational data
	for query search.
	\newblock In {\em Proceedings of the 2012 SIAM International Conference on Data
		Mining}, pages 141--152. SIAM, 2012.
	
	\bibitem{lim2005singular}
	L.-H. Lim.
	\newblock Singular values and eigenvalues of tensors: a variational approach.
	\newblock In {\em 1st IEEE International Workshop on Computational Advances in
		Multi-Sensor Adaptive Processing}, pages 129--132, 2005.
	
	\bibitem{liu2010always}
	Y.~Liu, G.~Zhou, and N.~F. Ibrahim.
	\newblock An always convergent algorithm for the largest eigenvalue of an
	irreducible nonnegative tensor.
	\newblock {\em J. Comput. Appl. Math.}, 235:286--292, 2010.
	
	\bibitem{ng2009finding}
	M.~Ng, L.~Qi, and G.~Zhou.
	\newblock Finding the largest eigenvalue of a nonnegative tensor.
	\newblock {\em SIAM J. Matrix Anal. Appl.}, 31:1090--1099, 2009.
	
	\bibitem{nguyen2016efficient}
	Q.~Nguyen, F.~Tudisco, A.~Gautier, and M.~Hein.
	\newblock An efficient multilinear optimization framework for hypergraph
	matching.
	\newblock {\em IEEE Trans. Pattern Anal. Mach. Intell.}, 2016.
	
	\bibitem{nikias1993signal}
	C.~L. Nikias and J.~M. Mendel.
	\newblock Signal processing with higher-order spectra.
	\newblock {\em IEEE Signal Processing Magazine}, 10:10--37, 1993.
	
	\bibitem{qi2017tensor}
	L.~Qi and Z.~Luo.
	\newblock {\em Tensor analysis: spectral theory and special tensors}, volume
	151.
	\newblock SIAM, 2017.
	
	\bibitem{qi2003multivariate}
	L.~Qi and K.~L. Teo.
	\newblock Multivariate polynomial minimization and its application in signal
	processing.
	\newblock {\em J. Global Optim.}, 26:419--433, 2003.
	
	\bibitem{regalia2003monotonic}
	P.~A. Regalia and E.~Kofidis.
	\newblock Monotonic convergence of fixed-point algorithms for {ICA}.
	\newblock {\em IEEE Trans. Neural Netw. Learn. Syst.}, 14:943--949, 2003.
	
	\bibitem{shanks1955non}
	D.~Shanks.
	\newblock Non-linear transformations of divergent and slowly convergent
	sequences.
	\newblock {\em J. Math. Phys.}, 34(1-4):1--42, 1955.
	
	\bibitem{sidi2017vector}
	A.~Sidi.
	\newblock {\em Vector extrapolation methods with applications}, volume~17.
	\newblock SIAM, 2017.
	
	\bibitem{swami1997bibliography}
	A.~Swami, G.~B Giannakis, and G.~Zhou.
	\newblock Bibliography on higher-order statistics.
	\newblock {\em Signal Processing}, 60:65--126, 1997.
	
	\bibitem{arrigo2017centrality}
	F.~Tudisco, F.~Arrigo, and A.~Gautier.
	\newblock Node and layer eigenvector centralities for multiplex networks.
	\newblock {\em SIAM J. Appl. Math.}, 78(2):853--876, 2018.
	
	\bibitem{wynn1956}
	P.~Wynn.
	\newblock On a device for computing the $e_m({S_n})$ transformation.
	\newblock {\em Math. Tables Aids Comput.}, pages 91--96, 1956.
	
	\bibitem{zhang2012linear}
	L.~Zhang, L.~Qi, and Y.~Xu.
	\newblock Linear convergence of the {LZI} algorithm for weakly positive
	tensors.
	\newblock {\em J. Comput. Math.}, 30:24--33, 2012.
	
	\bibitem{zhou2013efficient}
	G.~Zhou, L.~Qi, and S.-Y. Wu.
	\newblock Efficient algorithms for computing the largest eigenvalue of a
	nonnegative tensor.
	\newblock {\em Front. Math. China}, 8(1):155--168, 2013.
	
\end{thebibliography}
\end{document}